\documentclass[notitlepage, 11pt, a4paper]{article}
\usepackage{outlines}
\usepackage[T1]{fontenc}
\usepackage[natbib=true, backend=biber, style=numeric, maxbibnames=9]{biblatex}
\usepackage{csquotes}
\usepackage{amsmath}
\usepackage{amsfonts}
\usepackage{amssymb}
\usepackage{amsthm}
\usepackage[affil-it]{authblk}
\usepackage[normalem]{ulem}
\usepackage{dsfont}
\usepackage{float}
\usepackage{times}
\usepackage{bbm}
\usepackage{mathtools}
\usepackage{caption}
\usepackage{mathrsfs}
\usepackage{comment}
\usepackage{bm}
\usepackage{enumerate}
\usepackage{lipsum}
\usepackage{titling}
\usepackage{xcolor}
\usepackage[toc,page]{appendix}
\usepackage{accents}
\usepackage[ruled,vlined,linesnumbered]{algorithm2e}
\usepackage{subcaption}
\usepackage{tablefootnote}

\SetCommentSty{mycommfont}

\SetKwInput{KwInput}{Input}                
\SetKwInput{KwOutput}{Output}              

\usepackage[hidelinks]{hyperref}
\usepackage[capitalise]{cleveref}
\usepackage[left=3cm, right=3cm, bottom=2cm, top=2cm]{geometry}

\newcommand*\diff{\mathop{}\!\mathrm{d}}

\SetKwInOut{Require}{Require}

\newtheorem{theorem}{Theorem}[section]
\newtheorem{lemma}[theorem]{Lemma}
\newtheorem{proposition}[theorem]{Proposition}
\newtheorem{corollary}[theorem]{Corollary}
\newtheorem{assumption}[theorem]{Assumption}

\theoremstyle{definition}
\newtheorem{definition}[theorem]{Definition}

\newtheorem{remark}[theorem]{Remark}

\numberwithin{equation}{section}

\newcommand{\cadlag}{c\`adl\`ag }

\newcommand{\tmwk}{\ensuremath{\mathfrak{T}_{M_1}}^{\text{wk}}}
\newcommand{\DR}{\ensuremath{D_\R}}
\newcommand{\tnot}{\ensuremath{\tau_0(\eta)}}
\newcommand{\MRR}{\ensuremath{\mathbf{M}_{\le 1}(\R)}}

\newcommand{\N}{\ensuremath{\mathbb{N}}}

\newcommand{\R}{\ensuremath{\mathbb{R}}}

\newcommand{\1}{\ensuremath{\mathds{1}}}

\newcommand\norm[1]{\left\lVert#1\right\rVert}
\newcommand\nnorm[1]{|#1|}
\newcommand\nnnorm[1]{\left|#1\right|} 

\newcommand{\prob}{\ensuremath\mathbb{P}}
\newcommand{\E}{\mathbb{E}}

\newcommand{\ind}{\mathbbm{1}}
\newcommand{\ubar}[1]{\underaccent{\bar}{#1}}

\usepackage{pgfplots}
\usepackage{tkz-fct}

\pgfplotsset{every axis/.append style={
    axis x line=middle,    
    axis y line=middle,    
    axis line style={<->}, 
    xlabel={$x$},          
    ylabel={$y$},          
    },
    cmhplot/.style={color=blue,mark=none,line width=1pt,<->},
    soldot/.style={color=blue,only marks,mark=*},
    holdot/.style={color=blue,fill=white,only marks,mark=*},
}

\tikzset{>=stealth}

\pgfplotsset{compat=1.18}

\begin{document}
\title{Contagious McKean--Vlasov problems with common noise: from smooth to singular feedback through hitting times}

\author{Ben Hambly\thanks{Mathematical Institute, University of Oxford, United Kingdom, Email: hambly@maths.ox.ac.uk},
\and Alda\"ir Petronilia\thanks{Mathematical Institute, University of Oxford, United Kingdom, Email: aldair.petronilia@bnc.ox.ac.uk, Corresponding author},
\and Christoph Reisinger\thanks{Mathematical Institute, University of Oxford, United Kingdom, Email: christoph.reisinger@maths.ox.ac.uk},
\and Stefan Rigger\thanks{University of Vienna, Austria, Email: stefan.rigger@hotmail.com},
\and Andreas S{\o}jmark\thanks{Department of Statistics, London School of Economics, United Kingdom, Email: a.sojmark@lse.ac.uk}}
\date{April 22, 2025} 

\maketitle

\begin{abstract}
\noindent We consider a family of McKean--Vlasov equations arising as the large particle limit of a system of interacting particles on the positive half-line with common noise and feedback. Such systems are motivated by structural models for systemic risk with contagion. This contagious interaction is such that when a particle hits zero, the impact is to move all the others toward the origin through a kernel which smooths the impact over time. We study a rescaling of the impact kernel under which it converges to the Dirac delta function so that the interaction happens instantaneously and the limiting singular McKean--Vlasov equation can exhibit jumps. Our approach provides a novel method to construct solutions to such singular problems that allows for more general drift and diffusion coefficients and we establish weak convergence to relaxed solutions in this setting. With more restrictions on the coefficients we can establish an almost sure version showing convergence to strong solutions. Under some regularity conditions on the contagion, we also show a rate of convergence up to the time the regularity of the contagion breaks down. Lastly, we perform some numerical experiments to investigate the sharpness of our bounds for the rate of convergence.
\end{abstract}


\section{Introduction}

In this paper, we study the limiting behaviour of the family of conditional McKean--Vlasov equations
\begin{equation}\label{eq: EQUATION WITH SMOOTHING IN INTRODUCITON}
	\left\{
	\begin{array}{r@{{}={}}l}
		\diff{}X_t^\varepsilon &\begin{array}[t]{@{}l}
			b(t,X_t^\varepsilon,\bm{\nu}_t^\varepsilon)\diff t + \sigma(t,X^{\varepsilon}_t)\sqrt{1 - \rho(t,\bm{\nu}^\varepsilon_t)^2}\diff{}W_t + \sigma(t,X^{\varepsilon}_t)\rho(t,\bm{\nu}^\varepsilon_t) \diff{}W_t^0 \\[0.25em]
			- \alpha(t) \diff\mathfrak{L}_t^\varepsilon,\\[0.25em]
		\end{array} \\[0.25em]
		\tau^\varepsilon & \inf\{t > 0 \; : \; X_t^\varepsilon \le 0\},\\[0.25em]
		\mathbf{P}^\varepsilon & \prob\left[\left.X^\varepsilon \in \,\cdot\,\right|W^0 \right],\quad\bm{\nu}_t^\varepsilon = \prob\left[ \left.X_t^\varepsilon \in \cdot, \tau^\varepsilon > t\right|W^0\right],\\[0.25em]
		L_t^\varepsilon & \prob\left[\left.\tau^\varepsilon \le t \right|W^0 \right],\quad\mathfrak{L}^\varepsilon_t = \int_0^t \kappa^{\varepsilon}(t-s)L_s^\varepsilon\diff s,
	\end{array}
	\right.
\end{equation}
as $\varepsilon$ tends to zero. Here, $W$ and $W^0$ are independent standard Brownian motions, and $\kappa^\varepsilon$ is a rescaled mollifier which converges to the Dirac delta as $\varepsilon$ goes to $0$. Such equations arise as the limit of a large particle system, where $W$ is usually referred to as the \emph{idiosyncratic noise} (of a representative particle) and $W^0$ as the \emph{common noise}. Also for the same reason, $L^\varepsilon$ is referred to as the \emph{loss} process and quantifies the amount of mass that has crossed the boundary at zero by time $t$. A solution to this system consists of the random probability measure $\mathbf{P}^\varepsilon$ and the loss process $L^\varepsilon$, conditional on $W^0$.

In addition to the more classical measure dependence of the coefficients that characterise McKean--Vlasov equations, there is
a further feedback mechanism through the loss process $L^\varepsilon$: depending on the value of $\alpha(t) \ge 0$, $L^\varepsilon$ pushes $X^\varepsilon$ towards zero, causing the value of $L^\varepsilon$ to increase, hence pushing $X^\varepsilon$ even closer to 0.
The integral kernel $\kappa^{\varepsilon}$, which is parameterised by some $\varepsilon>0$, is a key element of the model and captures a latency in the transmission of $L^\varepsilon$ to $X^\varepsilon$ present in real-world systems. Precise conditions on the coefficient functions will be given later.

One motivation for this model arises in systemic risk, where $X^\varepsilon$ represents the \emph{distance-to-default} of a prototypical institution in a financial network with infinitely many entities, see for example \cite{hambly2019spde}. In this setting, $L_t^\varepsilon$ denotes the proportion of institutions that have defaulted by time $t$ and is the cause of \emph{endogenous contagion} through the feedback mechanism. In this model, we use the kernel $\kappa^{\varepsilon}$ to capture feedback where, when a financial institution defaults and their positions are unwound, the counterparties experience a decrease in the value of their assets over time.
A model for bank runs using such a smooth transmission of boundary losses was analysed in \cite{burzoni2021mean} in the presence of common noise. Moreover, \cite{inglis2015mean} study a related mean-field model for neurons interacting gradually through threshold hitting times, albeit without common noise.

If the support of $\kappa^\varepsilon$ is contained in the interval $[0, \gamma]$ with $\gamma \ll 1$, then 
the integral $\int_0^t\kappa^\varepsilon(t - s)L_s^\varepsilon \diff s$ is approximately equal to $L_t^\varepsilon$. In this article, we prove convergence in the following sense: if we fix a kernel $\kappa$ and rescale it with a variable $\varepsilon>0$ by $\kappa^{\varepsilon}(t) = \varepsilon^{-1} \kappa(\varepsilon^{-1} t)$, then we have convergence in the $M_1$-topology of $X^{\varepsilon}$ to $X$, where $X$ is a (relaxed) solution to
\begin{equation}\label{eq: FIRST EQUATION IN THE PAPER}
	\left\{
	\begin{array}{r@{{}={}}l}
		\diff{}X_t &\begin{array}[t]{@{}l}
			b(t,X_t,\bm{\nu}_t)\diff t + \sigma(t,X_t)\sqrt{1 - \rho(t,\bm{\nu}_t)^2}\diff{}W_t + \sigma(t,X_t)\rho(t,\bm{\nu}_t) \diff{}W_t^0\\[0.25em]
			- \alpha(t)\diff L_t,\\[0.25em]
		\end{array} \\[0.25em]
		\tau & \inf\{t > 0 \; : \; X_t \le 0\},\\[0.25em]
		\mathbf{P} & \prob\left[\left.X \in \,\cdot\,\right|W^0,\,\mathbf{P} \right],\quad\bm{\nu}_t = \prob\left[ \left.X_t \in \cdot, \tau > t\right|W^0,\,\mathbf{P}\right],\\[0.25em]
		L_t & \prob\left[\left.\tau \le t \right|W^0,\,\mathbf{P} \right].
	\end{array}
	\right.
\end{equation}


Intuitively, \eqref{eq: EQUATION WITH SMOOTHING IN INTRODUCITON} is a smoothed approximation to \eqref{eq: FIRST EQUATION IN THE PAPER}. 
Our motivation for taking the limit as $\varepsilon \to 0$ is to investigate the convergence to the system where the feedback is felt instantaneously, which captures the situation when the latency is small compared to the timescale of interest. 
It is well known that equations of the form \eqref{eq: FIRST EQUATION IN THE PAPER} may develop jump discontinuities, as we will elaborate below.



Variants and special cases of \eqref{eq: FIRST EQUATION IN THE PAPER} have been the subject of extensive research in the field. In the simplest scenario, where $b$ and $\rho$ are both zero, $\sigma$ is equal to $1$, and $\alpha$ is a positive constant, we obtain the probabilistic formulation of the \emph{supercooled Stefan problem}. 
The Stefan problem, introduced in \cite{stefan1889einige},
describes the temperature and the phase boundary of a material undergoing a phase transition, typically from a solid to a liquid. The supercooled Stefan problem describes the freezing of a supercooled liquid (i.e. a liquid which is below its freezing point).
Many authors constructed classical solutions 
\citep[]{fasano1981new,fasano1983critical,fasano1989singularities,lacey1985ill,HOWISONSINGULARITY}
for time intervals where $L_t$ is regular.
In the PDE literature, it was first established in \cite{sherman1970general} that $L_t$ may explode in finite time.


In the probabilistic formulation where the initial condition has finite mass, $L_t$ is bounded but there
can still exist a $t_* \in (0,\,\infty)$ such that $\lim_{s \uparrow t_*} L_s^\prime = \infty$, referred to as a blow-up. 
Moreover,  given suitable assumptions, for $\alpha(t)$ sufficiently large, a jump of $t \mapsto L_t$ must occur, as per \citep[Theorem~1.1]{hambly2019mckean}. 
	With added common noise, there is a set of paths of positive probability where a jump must happen, \citep[Theorem~2.1]{ledger2021mercy}. 
	

	The probabilistic reformulation provides a natural way to restart the system following a blow-up. 
	From this perspective, for arbitrary initial conditions, see \citep[Example~2.2]{hambly2019mckean}, there may be infinitely many solutions: 
	it may be possible for two solutions to be equal up to the first jump time $t$ and then take jumps of different sizes. To address this ambiguity that arises at a jump time, a condition is typically imposed that selects the smallest possible jump sizes. This condition is known as the \emph{physical jump condition}, defined as:
	\begin{equation}\label{eq: PHYSICAL JUMP CONDITION IN THE INTRODUCTION}
		\Delta L_t = \inf \{x > 0\,:\, \bm{\nu}_{t-}([0,\alpha x]) < x\}, 
	\end{equation}
	where $\Delta L_t \coloneqq L_t - L_{t-}$. The intuitive interpretation of \eqref{eq: PHYSICAL JUMP CONDITION IN THE INTRODUCTION} is that if we take the density of $X_{t-}\ind_{\tau \ge t}$ and displace it by an $\alpha x$ amount towards $0$, then the mass of the system below zero is exactly $x$. So it is the minimal amount by which we may displace our density such that the displacement and the mass below zero correspond. From a modelling perspective, the physical jump condition is the preferred choice of jump sizes due to its economic and physical interpretations. 
	
	
	Extensive research has been conducted to investigate various properties of \emph{physical solutions} to the equation 
	with $b=\rho=0$, $\sigma=1$, and $\alpha$ a positive constant
	\citep{nadtochiy2019particle,MEANFIELDTHROUGHHITTINGTIMESNADOTOCHIY,hambly2019mckean,ledger2020uniqueness,ledger2021mercy,lipton2021semi,kaushansky2020convergence,cuchiero2020propagation}. The paper \cite{delarue2022global} establishes that when $X_{0-}$ possesses a bounded density that changes monotonicity finitely many times, then $L$ is unique, and for any $t \ge 0$, $L$ is continuously differentiable on $(t,\, t+ h)$ for some $h > 0$. Additionally, in \cite{hambly2019mckean}, it is demonstrated that for an initial condition with a bounded density that is H\"older continuous near the boundary, $L$ is unique, continuous, and has a weak derivative until some explosion time. The work in \cite{ledger2020uniqueness} extends these results by showing that if the initial condition possesses an $L^2$-density, then we have uniqueness for a short time after the explosion time. Moreover, irrespective of the initial condition, there exists a minimal loss process that will be dominated by any other loss process that solves the equation \citep{cuchiero2020propagation}. It has been established that such minimal solutions are physical \citep[Theorem~6.5]{cuchiero2020propagation}. However, it remains unclear whether physical solutions are necessarily minimal due to the lack of uniqueness for general initial conditions.
	
	Returning to \eqref{eq: FIRST EQUATION IN THE PAPER}, recent advances have been made in the study of general coefficients, specifically $t \mapsto b(t)$, $t \mapsto \sigma(t)$, and $t \mapsto \rho(t)$, in the presence of common noise. 
	In Remark 2.5 from \citep{MEANFIELDTHROUGHHITTINGTIMESNADOTOCHIY}, a generalized Schauder fixed-point argument is presented to construct \emph{strong solutions} in this setting. Strong solutions refer to the property $\mathbf{P} = \prob(X \in \cdot \mid W^0)$, indicating that the random probability measure $\mathbf{P}$ is adapted to the $\sigma$-algebra generated by the common noise. In \citep{ledger2021mercy}, an underlying finite particle system was shown to converge to \emph{relaxed} (or \emph{weak}) \emph{solutions} (see Definition \ref{def: DEFINITION OF RELAXED SOLUTIONS GENERALISED}), satisfying the aforementioned physical jump condition (with coefficients $(t,x) \mapsto b(t,x)$, $t\mapsto \sigma(t)$, and $t\mapsto \rho(t)$). Weak/relaxed solutions are characterised by having $\mathbf{P} = \prob(X \in \cdot \mid W^0,\mathbf{P})$, instead of $\mathbf{P} = \prob(X \in \cdot \mid W^0)$, see Definition \ref{def: DEFINITION OF RELAXED SOLUTIONS GENERALISED}. As the empirical distributions of the finite particle systems converge weakly to $\mathbf{P}$, there is no guarantee that $\mathbf{P}$ will be adapted to the $\sigma$-algebra generated by the common noise. The existence of such \emph{strong} solutions in the sense just discussed, for the common noise problem satisfying the physical jump condition \eqref{eq: PHYSICAL JUMP CONDITION IN THE INTRODUCTION}, has not yet been addressed in the literature.

	The main contributions and structure of this paper are as follows:
	\begin{itemize}
		\item 
		Firstly, in Section \ref{sec: SECTION ON GENERALISED CONVERGENCE TO DELAYED SOLUTIONS},
		we prove Theorem \ref{thm:sec2: THE MAIN EXISTENCE AND CONVERGENCE THEOREM OF SOLUTIONS GENERALISED FOR SECTION 2} and Corollary \ref{thm:sec2: THE MAIN EXISTENCE AND CONVERGENCE THEOREM OF PHYSICAL SOLUTIONS FOR SECTION 2}
		showing the weak convergence of solutions of \eqref{eq: EQUATION WITH SMOOTHING IN INTRODUCITON} to relaxed solutions of \eqref{eq: FIRST EQUATION IN THE PAPER} as $\varepsilon \rightarrow 0$, i.e., as the smoothed feedback mechanism becomes instantaneous in the limit.
		As a by-product, this gives a novel method for establishing the existence of solutions to \eqref{eq: FIRST EQUATION IN THE PAPER},
		avoiding time regularity assumptions on $\sigma \rho$ as needed in \cite{ledger2021mercy}. Furthermore, we derive an upper bound on the jump sizes, Theorem \ref{thm:sec2: THE MAIN EXISTENCE AND CONVERGENCE THEOREM OF SOLUTIONS GENERALISED FOR SECTION 2}, and, under additional assumptions on the coefficients, Corollary \ref{thm:sec2: THE MAIN EXISTENCE AND CONVERGENCE THEOREM OF PHYSICAL SOLUTIONS FOR SECTION 2}, show that the loss process $L$ satisfies the physical jump condition
		\eqref{eq: PHYSICAL JUMP CONDITION IN THE INTRODUCTION}.
		\item
		Secondly, in Section \ref{sec: SECTION ON THE CONVERGENCE FROM STEFAN}, we show in Theorem \ref{thm: Propogation of Minimality} that, if the coefficients depend solely on time and $\alpha$ is a constant, then we may upgrade our mode of convergence from weak to almost sure.
		As a consequence of the method employed, we can guarantee that the limiting loss process will be $W^0$-measurable and 
		satisfy the physical jump condition. In addition, we have the existence of strong solutions in this setting.
		\item
		Lastly, in Section \ref{subsec: RATES OF CONVERGENCE OF THE MOLLIFIED PROCESS}, for constant coefficients and without common noise, we provide in Proposition \ref{prop: MAIN RESULT PROPOSITION IN THE GENERAL CASE} an explicit rate of convergence of the smoothed approximations to the singular system prior to the first time the regularity of the loss function breaks down.
		We also give numerical tests of the convergence order in scenarios of different regularity, with and without common noise.
	\end{itemize}
	
	
	
	\section{Weak convergence of smoothed feedback systems}\label{sec: SECTION ON GENERALISED CONVERGENCE TO DELAYED SOLUTIONS}
	
	Fix a finite time horizon $T > 0$, and let $\mathcal{P}(A)$ denote the set of probability measures on a measurable space $(A,\mathcal{A})$. When $A$ is a metric space, $\mathcal{B}(A)$ denotes the Borel $\sigma$-algebra. Let $\mathbf{M}_{\le 1}(A)$ denote the space of sub-probability measures, which we shall endow with the topology of weak convergence. For any interval $I$ and metric space $X$, let $C(I, X)$ denote the space of continuous functions from $I$ to $X$. Similarly, $D(I,X)$ denotes the space of \cadlag functions from $I$ to $X$. We shall employ the shorthand notation $\mathcal{C}_X$ and $D_X$ for $C(I, X)$ and $D(I, X)$, respectively, when the interval $I$ is clear.
	
	For every $\varepsilon > 0$, we fix a probability space $(\Omega^\varepsilon, \mathcal{F}^\varepsilon,\,\prob^\varepsilon)$ that supports two independent Brownian motions. To simplify the notation, we will denote these Brownian motions by $W$ and $W^{0}$; however, it is important to note that they may not be equal for different values of $\varepsilon$. Similarly, we adopt the simplified notations $\prob$ and $\E$ to refer to $\prob^\varepsilon$ and the expectation under the measure $\prob^\varepsilon$ respectively. In this section, we characterise the weak limit of the system given by the following equation as $\varepsilon$ tends to zero:
	\begin{equation}\label{eq: MOLLIFIED MCKEAN VLASOV EQUATION IN THE SYSTEMIC RISK MODEL WITH COMMON NOISE GENERALISED}
		\left\{
		\begin{array}{r@{{}={}}l}
			\diff{}X_t^\varepsilon &\begin{array}[t]{@{}l}
				b(t,X_t^\varepsilon,\bm{\nu}_t^\varepsilon)\diff t + \sigma(t,X^{\varepsilon}_t)\sqrt{1 - \rho(t,\bm{\nu}^\varepsilon_t)^2}\diff{}W_t + \sigma(t,X^{\varepsilon}_t)\rho(t,\bm{\nu}^\varepsilon_t) \diff{}W_t^0\\[0.25em]
				- \alpha(t)\diff\mathfrak{L}_t^\varepsilon,\\[0.25em]
			\end{array} \\[0.25em]
			\tau^\varepsilon& \inf\{t > 0 \; : \; X_t^\varepsilon \le 0\},\\[0.25em]
			\mathbf{P}^\varepsilon & \prob\left[\left.X^\varepsilon \in \,\cdot\,\right|W^0 \right],\quad\bm{\nu}_t^\varepsilon \coloneqq \prob\left[ \left.X_t^\varepsilon \in \cdot, \tau^\varepsilon > t\right|W^0\right],\\[0.25em]
			L_t^\varepsilon & \prob\left[\left.\tau^\varepsilon \le t \right|W^0\right],\quad\mathfrak{L}^\varepsilon_t = \int_0^t \kappa^{\varepsilon}(t-s)L_s^\varepsilon\diff s,
		\end{array}
		\right.
	\end{equation}
	where $t \in [0,\,T]$. The coefficient $b$ ($\sigma$, $\rho$, or $\alpha$, respectively) is a measurable map from $[0,\,T] \times \R \times \MRR$ ($[0,\,T] \times \R$, $[0,\,T] \times \MRR$, or $[0,\,T]$, respectively) into $\R$. The initial condition, denoted by $X_{0-}$, is assumed to be independent of the Brownian motions and positive almost surely. Finally, we define $\kappa^\varepsilon(t) \coloneqq \varepsilon^{-1}\kappa(t \varepsilon^{-1})$.
	
	One way to view $X^\varepsilon$ is as the mean-field limit of an interacting particle system where particles interact through their first hitting time of zero. The interactions among particles are smoothed out over time by convolving with the kernel $\kappa^\varepsilon$. As $\varepsilon$ approaches zero, the effect of interactions occurs over increasingly smaller time intervals. As $\kappa^\varepsilon$ is a mollifier, it is natural to expect, as $\varepsilon$ tends to zero, $\mathfrak{L}_t^\varepsilon$ to converge to the instantaneous loss at time $t$. That is to say, along a suitable subsequence, the random tuple $\{(\mathbf{P}^\varepsilon,\,W^0,\,W)\}_{\varepsilon > 0}$ would have a limit point $(\mathbf{P},\, W^0,\, W)$ where $\mathbf{P} = \prob\left[X \in \cdot \mid W^0\right]$ and $X$ solves
	\begin{equation}\label{eq: SINGULAR MCKEAN VLASOV EQUATION IN THE SYSTEMIC RISK MODEL WITH COMMON NOISE GENERALISED}
		\left\{
		\begin{array}{r@{{}={}}l}
			\diff{}X_t &\begin{array}[t]{@{}l}
				b(t,X_t,\bm{\nu}_t)\diff t + \sigma(t,X_t)\sqrt{1 - \rho(t,\bm{\nu}_t)^2}\diff{}W_t + \sigma(t,X_t)\rho(t,\bm{\nu}_t) \diff{}W_t^0\\[0.25em]
				- \alpha(t)\diff L_t,\\[0.25em]
			\end{array} \\[0.25em]
			\tau & \inf\{t > 0 \; : \; X_t \le 0\},\\[0.25em]
			\mathbf{P} & \prob\left[\left.X \in \,\cdot\,\right|W^0 \right],\quad\bm{\nu}_t \coloneqq \prob\left[ \left.X_t \in \cdot, \tau > t\right|W^0\right],\\[0.25em]
			L_t & \prob\left[\left.\tau \le t \right|W^0\right],
		\end{array}
		\right.
	\end{equation}
	with $X_{0} = X_{0-} + \alpha(0)L_0$. In this system, the feedback is felt instantaneously and is characterised by the common noise $W^0$. In what follows, for technical reasons, we construct an extension $\tilde{X}$ of the process $X$. For an arbitrary stochastic process $Z$, we define its extended version as follows,
	\begin{equation}\label{eq: HOW WE DEFINE THE EXTENSION TO STOCHASTIC PROCESSES}
		\tilde{Z}_t = 
		\begin{cases}
			\begin{aligned}
				&Z_{0-} &\qquad &t \in [-1,\,0), \\
				&Z_{t} &\qquad &t \in [0,\,T], \\
				&Z_{T} + W_t - W_T &\qquad &t \in (T,\,T +1].
			\end{aligned}
		\end{cases}
	\end{equation}
	We artificially extend the processes to be constant on $[-1,0)$ and by a pure Brownian noise term on $(T,T+1]$. Therefore, the extension to $\mathbf{P}^\varepsilon$ is $\tilde{\mathbf{P}}^\varepsilon \coloneqq \prob(\tilde{X}^{\varepsilon} \in \cdot \mid W^0)$. Consequently, the random measure $\tilde{\mathbf{P}}^\varepsilon$ remains $W^0$-measurable. We show that the collection of measures $\{\tilde{\mathbf{P}}^\varepsilon\}_{\varepsilon > 0}$ is tight; hence there exists a subsequence $(\varepsilon_n)_{n \ge 1}$ that converges to zero such that $\tilde{\mathbf{P}}^{\varepsilon_n}$ converges weakly to the random measure $\mathbf{P}$. However, as the mode of convergence is weak, we cannot expect that the limit point $\mathbf{P}$ is also measurable with respect to $W^0$.
	
	Hence, we relax our notion of solution to \eqref{eq: SINGULAR MCKEAN VLASOV EQUATION IN THE SYSTEMIC RISK MODEL WITH COMMON NOISE GENERALISED},
	which leads to the definition of \emph{relaxed solutions} employed in the literature when studying the mean-field limit of particle systems with common noise \citep{ledger2021mercy} and also in the mean-field game literature with common source of noise \citep{carmona:hal-01144847}.
	
	\begin{definition}[Relaxed solutions] \label{def: DEFINITION OF RELAXED SOLUTIONS GENERALISED}
		Let the coefficient functions $b,\,\sigma,\,\rho$, and $\alpha$ be given along with the initial condition $X_{0-}$ at time $t = 0-$. We define a relaxed solution to \eqref{eq: SINGULAR MCKEAN VLASOV EQUATION IN THE SYSTEMIC RISK MODEL WITH COMMON NOISE GENERALISED} as a family $(X,\, W,\, W^0,\,\, \mathbf{P})$ on a filtered probability space $(\Omega,\,\mathcal{F},\,\prob)$ such that
		\begin{equation}\label{eq: SINGULAR RELAXED MCKEAN VLASOV EQUATION IN THE SYSTEMIC RISK MODEL WITH COMMON NOISE GENERALISED}
			\left\{
			\begin{array}{r@{{}={}}l}
				\diff{}X_t &\begin{array}[t]{@{}l}
					b(t,X_t,\bm{\nu}_t)\diff t + \sigma(t,X_t)\sqrt{1 - \rho(t,\bm{\nu}_t)^2}\diff{}W_t + \sigma(t,X_t)\rho(t,\bm{\nu}_t) \diff{}W_t^0\\[0.25em]
					- \alpha(t)\diff L_t,\\[0.25em]
				\end{array} \\[0.25em]
				\tau & \inf\{t > 0 \; : \; X_t \le 0\},\\[0.25em]
				\mathbf{P} & \prob\left[\left.X \in \,\cdot\,\right|W^0,\,\mathbf{P} \right],\quad\bm{\nu}_t \coloneqq \prob\left[ \left.X_t \in \cdot, \tau > t\right|W^0,\,\mathbf{P}\right],\\[0.25em]
				L_t & \prob\left[\left.\tau \le t \,\right|W^0,\,\mathbf{P} \right],
			\end{array}
			\right.
		\end{equation}
		with $X_0 = X_{0-} + \alpha(0)L_0$, $L_{0-} = 0$, $X_{0-} \perp (W,\,W^0,\,\mathbf{P})$, and $(W^0,\mathbf{P}) \perp W$, where $(W,\,W^0)$ is a two-dimensional Brownian motion, $X$ is a \cadlag process, and $\mathbf{P}$ is a random probability measure on the space of \cadlag paths $D([-1,T+1],\R)$.
	\end{definition}
	
	As the drift and correlation function depend on a flow of measures, we still want them to satisfy some notion of linear growth and Lipschitzness in the measure component. We will also require some spatial and temporal regularity such that \eqref{eq: MOLLIFIED MCKEAN VLASOV EQUATION IN THE SYSTEMIC RISK MODEL WITH COMMON NOISE GENERALISED} is well-posed. We will suppose that our coefficients $b,\,\sigma$, $\rho,\,\kappa$ and $\alpha$ satisfy the following assumptions.
	
	\begin{assumption}\label{ass: MODIFIED AND SIMPLIED FROM SOJMARK SPDE PAPER ASSUMPTIONS II}
		\begin{enumerate}[(i)]
			\item \label{ass: MODIFIED AND SIMPLIED FROM SOJMARK SPDE PAPER ASSUMPTIONS II ONE} (Regularity of $b$) For all $t \in [0,\,T]$ and $\mu \in \MRR$, the map $x \mapsto b(t,x,\mu)$ is $\mathcal{C}^2(\R)$. Moreover, there exists a constant $C_b > 0$ such that
			\begin{align*}
				&\nnnorm{b(t,x,\mu)} \le C_b\left(1 + \nnnorm{x} + \langle\mu,\,\nnorm{\,\cdot\,}\rangle\right),\quad\nnorm{\partial_x^{(n)}b(t,x,\mu)}\le C_b,\quad n = 1,\,2,\\
				&\nnnorm{b(t,x,\mu) - b(t,x,\tilde{\mu})} \le C_b \left(1 + \nnnorm{x} + \langle\mu,\,\nnnorm{\,\cdot\,}\rangle\right)d_0(\mu,\,\tilde{\mu}), 
			\end{align*}
			where
			\begin{equation*}
				d_0(\mu,\,\tilde{\mu}) = \sup\left\{\nnnorm{\langle\mu - \tilde{\mu},\, \psi\rangle}\,:\, \norm{\psi}_{\operatorname{Lip}} \le 1,\, \nnnorm{\psi(0)} \le 1\right\}
			\end{equation*}
			for any $\mu,\tilde{\mu} \in \MRR$.
			\item \label{ass: MODIFIED AND SIMPLIED FROM SOJMARK SPDE PAPER ASSUMPTIONS II TWO} (Space/time regularity of $\sigma$) The map $(t,\,x) \mapsto \sigma(t,\,x)$ is $\mathcal{C}^{1,2}([0,\, T] \times \R)$. Moreover, there exists a constant $C_{\sigma} > 0$ such that
			\begin{equation*}
				\nnorm{\sigma(t,\,x)}\le C_\sigma,\quad \nnorm{\partial_t\sigma(t,\,x)}\le C_\sigma ,\quad\textnormal{and}\quad\nnorm{\partial_x^{(n)}\sigma(t,x)}\le C_\sigma \quad \textnormal{for} \quad n =1,\,2.
			\end{equation*}
			\item \label{ass: MODIFIED AND SIMPLIED FROM SOJMARK SPDE PAPER ASSUMPTIONS II THREE} ($d_1$-Lipschitzness of $\rho$) For all $t \in [0,\,T]$, there exists a constant $C_{\rho} > 0$ such that
			\begin{equation*}
				\nnnorm{\rho(t,\mu) - \rho(t,\tilde{\mu})} \le C_\rho \left(1 + \langle\mu,\,\nnnorm{\,\cdot\,}\rangle\right)d_1(\mu,\,\tilde{\mu}),
			\end{equation*}
			where
			\begin{equation*}
				d_1(\mu,\,\tilde{\mu}) = \sup\left\{\nnnorm{\langle\mu - \tilde{\mu},\, \psi\rangle}\,:\, \norm{\psi}_{\operatorname{Lip}} \le 1,\, \norm{\psi}_{\infty} \le 1\right\}
			\end{equation*}
			for any $\mu,\tilde{\mu} \in \MRR$.
			\item \label{ass: MODIFIED AND SIMPLIED FROM SOJMARK SPDE PAPER ASSUMPTIONS II FOUR}  (Non-degeneracy) For all $t \in [0,\,T]$, $x \in \R$, and $\mu \in \MRR$, the constants $C_\sigma$ and $C_\rho$ assumed above are such that $0 < C_\sigma^{-1} \le \sigma(t,x)$ and $0 \le \rho(t,\mu) \le 1 - C_{\rho}^{-1}$.
			\item \label{ass: MODIFIED AND SIMPLIED FROM SOJMARK SPDE PAPER ASSUMPTIONS II FIVE} (Temporal regularity of $\alpha$) The map $t \mapsto \alpha(t)$ is $\mathcal{C}^{1}([0,\, T])$ and increasing with $\alpha(0) \ge 0$.
			\item \label{ass: MODIFIED AND SIMPLIED FROM SOJMARK SPDE PAPER ASSUMPTIONS THREE} (Sub-Gaussian initial law) The initial law, $\nu_{0-}$ is sub-Gaussian,
			\begin{equation*}
				\exists\, \gamma > 0 \quad \text{s.t.} \quad \nu_{0-}(\lambda,\infty) = O(e^{-\gamma \lambda^2}) \quad \text{as} \quad \lambda \to \infty,
			\end{equation*}
			and has a density $V_{0-} \in L^2(0,\infty)$ such that $\norm{xV_{0-}}_{L^2}^2 = \int_0^\infty\nnorm{xV_{0-}(x)}^2 \diff{}x< \infty$.
			\item \label{ass: MODIFIED AND SIMPLIED FROM SOJMARK SPDE PAPER ASSUMPTIONS FOUR} (Regularity of mollifier) The function $\kappa \in \mathcal{W}^{1,1}_0(\R_+)$, the Sobolev space with one weak derivative in $L^1$ and zero at $0$, such that $\kappa$ is non-negative, and $\norm{\kappa}_1 = 1$.
		\end{enumerate}
	\end{assumption}

	Under Assumption \ref{ass: MODIFIED AND SIMPLIED FROM SOJMARK SPDE PAPER ASSUMPTIONS II}, \citep{hambly2019spde} 
	showed existence and uniqueness of solutions to a stochastic partial differential equation (SPDE) that any solution to \eqref{eq: MOLLIFIED MCKEAN VLASOV EQUATION IN THE SYSTEMIC RISK MODEL WITH COMMON NOISE GENERALISED} will satisfy.
	
	\begin{theorem}[{\citep[Theorem~2.6]{hambly2019spde}}]\label{thm: EXISTENCE AND UNIQUESS OF SOLUTIONS THEOREM FROM SPDE THAT ONLY INCLUDES THE UNIQUESS PART}
		There is a unique strong solution to the SPDE
		\begin{equation}\label{eq: THE SPDE EQUATION FROM THE SPDE PAPER}
			\begin{aligned}
				\diff\langle\mu_t,\,\phi\rangle =& \langle\mu_t,\,b(t,\,\cdot,\,\mu_t)\partial_x\phi\rangle \diff t + \frac{1}{2}\langle\mu_t,\,\sigma(t,\,\cdot)^2\partial_{xx}\phi\rangle \diff t\\
				&+ \langle\mu_t,\,\sigma(t,\,\cdot)\rho(t,\,\mu_t)\partial_x\phi\rangle \diff W^0_t -  \langle\mu_t,\,\alpha(t)\partial_x\phi\rangle \diff \mathfrak{L}_t^\varepsilon, \\
				\mathfrak{L}_t^\varepsilon \coloneqq& \int_0^t \kappa^\varepsilon(t - s)L_s \diff s, \\ L_t \coloneqq& 1 - \mu_t(0,\infty),
			\end{aligned}
		\end{equation}
		where the coefficients $b,\,\sigma$, $\rho,\,\kappa^\varepsilon$, and $\alpha$ satisfy Assumption \ref{ass: MODIFIED AND SIMPLIED FROM SOJMARK SPDE PAPER ASSUMPTIONS II}, and $\phi \in \mathscr{C}_0$, the set of Schwartz functions that are zero at $0$. 
		
	\end{theorem}
	
	From \citep[Theorem~2.6]{hambly2019spde}, we can deduce the existence of solutions to \eqref{eq: MOLLIFIED MCKEAN VLASOV EQUATION IN THE SYSTEMIC RISK MODEL WITH COMMON NOISE GENERALISED}.
	
	\begin{theorem}[{\citep[Theorem~2.7]{hambly2019spde}}]\label{thm: EXISTENCE AND UNIQUESS OF SOLUTIONS THEOREM FROM SPDE}
		Let $(\bm{\nu}^\varepsilon,\,W^0)$ be the unique strong solution to the SPDE \eqref{eq: THE SPDE EQUATION FROM THE SPDE PAPER}. Then, for any Brownian motion $W \perp (X_{0-},W^0)$, we have
		\begin{equation*}
			\bm{\nu}^\varepsilon_t = \prob \left[X_t^\varepsilon \in \cdot,\, \tau^\varepsilon > t\mid W^0\right], 
		\end{equation*}
		where $X^\varepsilon$ is the solution to the conditional McKean--Vlasov diffusion
		\begin{equation*}
			\left\{
			\begin{array}{r@{{}={}}l}
				\diff{}X_t^\varepsilon & \begin{array}[t]{@{}l}
					b(t,X_t^\varepsilon,\bm{\nu}^\varepsilon_t)\diff t + \sigma(t,X_t^\varepsilon)\sqrt{1 - \rho(t,\bm{\nu}^\varepsilon_t)^2}\diff{}W_t + \sigma(t,X_t^\varepsilon)\rho(t,\bm{\nu}^\varepsilon_t) \diff{}W_t^0 \\[0.25em]
					- \alpha(t)\diff\mathfrak{L}_t^\varepsilon,\\[0.25em]
				\end{array} \\[0.25em]
				\mathbf{P}^\varepsilon & \prob\left[\left.X^\varepsilon \in \,\cdot\,\right|W^0 \right],\\[0.25em]
				L_t^\varepsilon & \prob\left[\left.\tau^\varepsilon \le t \,\right|W^0 \right],\quad\mathfrak{L}^\varepsilon_t = \int_0^t \kappa^{\varepsilon}(t-s)L_s^\varepsilon\diff s, \\
				\tau^\varepsilon & \inf\{t > 0\,:\, X^\varepsilon_t \le 0\},
			\end{array}
			\right.
		\end{equation*}
		with initial condition $X_{0-} \sim \nu_{0-}$.
	\end{theorem}
	
	With \citep[Theorem~2.6]{hambly2019spde} and \citep[Theorem~2.7]{hambly2019spde}, it can be inferred that solutions to \eqref{eq: MOLLIFIED MCKEAN VLASOV EQUATION IN THE SYSTEMIC RISK MODEL WITH COMMON NOISE GENERALISED} are unique. The existence of solutions to \eqref{eq: MOLLIFIED MCKEAN VLASOV EQUATION IN THE SYSTEMIC RISK MODEL WITH COMMON NOISE GENERALISED} allows us to introduce the main result of this section, showing that solutions to \eqref{eq: SINGULAR RELAXED MCKEAN VLASOV EQUATION IN THE SYSTEMIC RISK MODEL WITH COMMON NOISE GENERALISED} exist as limit points of the collection of smoothed equations.

	\begin{theorem}[Existence and convergence generalised]\label{thm:sec2: THE MAIN EXISTENCE AND CONVERGENCE THEOREM OF SOLUTIONS GENERALISED FOR SECTION 2}
		Let $\tilde{X}^\varepsilon$ be the extended version of $X^\varepsilon$ in \eqref{eq: MOLLIFIED MCKEAN VLASOV EQUATION IN THE SYSTEMIC RISK MODEL WITH COMMON NOISE GENERALISED} and set $\tilde{\mathbf{P}}^\varepsilon = \operatorname{Law}(\tilde{X}^\varepsilon\mid W^0)$. Then, the family of random tuples $\{(\tilde{\mathbf{P}}^\varepsilon,\,W^0,\,W)\}_{\varepsilon > 0}$ is tight. Any subsequence $\{(\tilde{\mathbf{P}}^{\varepsilon_n},\,W^0,\,W)\}_{n \ge 1}$, for a positive sequence $(\varepsilon_n)_{n \ge 1}$ which converges to zero, has a further subsequence which converges weakly to some $(\mathbf{P},\,W^0,\,W)$. Here, $W^0$ and $W$ are standard Brownian motions, $\mathbf{P}$ is a random probability measure $\mathbf{P}\,:\, \Omega \to \mathcal{P}(\DR)$. 
		
		Given this limit point, there is a background space which carries a stochastic process $X$ such that $(X,\, W^0,\, W,\,\mathbf{P})$ is a relaxed solution to \eqref{eq: SINGULAR RELAXED MCKEAN VLASOV EQUATION IN THE SYSTEMIC RISK MODEL WITH COMMON NOISE GENERALISED} in the sense of Definition \ref{def: DEFINITION OF RELAXED SOLUTIONS GENERALISED}. $X_{0-}$, $W$ and $(W^0, \mathbf{P})$ are all mutually independent and 
		\begin{equation}\label{eq:  UPPER BOUND ON THE JUMPS OF THE LIMITING LOSS DEFINED IN thm:sec2: THE MAIN EXISTENCE AND CONVERGENCE THEOREM OF SOLUTIONS GENERALISED FOR SECTION 2}
			\Delta L_t = \prob\left[\left.\inf_{s < t} X_s > 0,\, X_t \le 0 \right | W^0,\, \mathbf{P} \right] \le \inf\left\{ x \ge 0\,:\, \bm{\nu}_{t-}([0,\,\alpha(t) x]) < x\right\} \qquad \textnormal{a.s.}
		\end{equation}
		for all $t \ge 0$.
	\end{theorem}
	
	\color{black}
	
	The notation $\operatorname{Law}(\tilde{X}^\varepsilon\mid W^0)$ stands for the conditional law of $\tilde{X}^\varepsilon$ given $W^0$, which indeed defines a random $W^0$-measurable probability measure on $D([-1,T+1],\R)$. Under stronger assumptions, namely $b,\,\sigma,$ and $\rho$ being of the form $(t,x) \mapsto b(t,x)$, $t \mapsto \sigma(t)$, $t \mapsto \rho(t)$ and $\alpha$ is a positive constant, there are established results in the literature for a lower bound on the jumps of the loss function. By Proposition 3.5 in \citep[]{ledger2021mercy}, the jumps of the loss satisfy
	\begin{equation*}
		\Delta L_t \ge \inf\left\{ x \ge 0\,:\, \bm{\nu}_{t-}([0,\,\alpha x]) < x\right\} \qquad \textnormal{a.s.}
	\end{equation*}
	Due to the generality of the coefficients, we were not able to establish if \eqref{eq:  UPPER BOUND ON THE JUMPS OF THE LIMITING LOSS DEFINED IN thm:sec2: THE MAIN EXISTENCE AND CONVERGENCE THEOREM OF SOLUTIONS GENERALISED FOR SECTION 2} holds with equality. The primary reason is the lack of independence between the term driven by the idiosyncratic noise and the remainder of the terms that $X$ is composed of. Hence, the technique employed in \citep[Proposition~3.5]{ledger2021mercy} may not be readily applied or extended to our setting. Regardless, given these two results, under stronger assumptions, we have the following existence result.
	
	\begin{corollary}[Existence of physical solutions]\label{thm:sec2: THE MAIN EXISTENCE AND CONVERGENCE THEOREM OF PHYSICAL SOLUTIONS FOR SECTION 2}
		Let the coefficients $b,\,\sigma,$ and $\rho$ be of the form $(t,x) \mapsto b(t,x)$, $t \mapsto \sigma(t)$, $t \mapsto \rho(t)$ and satisfy Assumption \ref{ass: MODIFIED AND SIMPLIED FROM SOJMARK SPDE PAPER ASSUMPTIONS II}. Then provided $\alpha(t) \equiv \alpha > 0$ and constant, there exists a relaxed solution to
		\begin{equation}\label{eq: SINGULAR RELAXED MCKEAN VLASOV EQUATION IN THE SYSTEMIC RISK MODEL WITH COMMON NOISE}
			\left\{
			\begin{array}{r@{{}={}}l}
				\diff{}X_t & b(t,X_t)\diff t + \sigma(t)\sqrt{1 - \rho(t)^2}\diff{}W_t + \sigma(t)\rho(t) \diff{}W_t^0 - \alpha\diff L_t,\\[0.25em]
				\tau & \inf\{t > 0 \; : \; X_t \le 0\},\\[0.25em]
				\mathbf{P} & \prob\left[\left.X \in \,\cdot\,\right|W^0,\,\mathbf{P} \right],\quad\bm{\nu}_t \coloneqq \prob\left[ \left.X_t \in \cdot, \tau > t\right|W^0,\,\mathbf{P}\right],\\[0.25em]
				L_t & \prob\left[\left.\tau \le t\,\right|W^0,\,\mathbf{P} \right].
			\end{array}
			\right.
		\end{equation}
		Moreover, we have the minimal jump constraint
		\begin{equation*}
			\Delta L_t = \inf\left\{ x \ge 0\,:\, \bm{\nu}_{t-}([0,\,\alpha x]) < x\right\} \qquad \textnormal{a.s.}
		\end{equation*}
		for all $t \ge 0$. This determines the jump sizes of $L$.
	\end{corollary}

	This presents a minor generalisation of the results in \citep{ledger2021mercy}. In their work, the authors imposed the condition that $t \mapsto \sigma(t)\rho(t)$ must be H\"older continuous with an exponent strictly greater than $1/2$. This was necessary in \citep[Lemma~3.16]{ledger2021mercy} to deduce the independence between the idiosyncratic noise and the common noise in the limiting system \eqref{eq: SINGULAR RELAXED MCKEAN VLASOV EQUATION IN THE SYSTEMIC RISK MODEL WITH COMMON NOISE} due to their particle system approach to the problem. Here, such explicit assumptions on the regularity of $\rho$ are not necessary, as we obtain the independence between the idiosyncratic noise and the common noise from the approximating system without additional conditions. Consequently, we can consider Corollary \ref{thm:sec2: THE MAIN EXISTENCE AND CONVERGENCE THEOREM OF PHYSICAL SOLUTIONS FOR SECTION 2} to be an extension of Theorem 3.2 in \citep{ledger2021mercy}.
	
	
	\subsection{Limit points of the smoothed system}
	
	In order to show the existence of a limit point of $X^\varepsilon$, we must first choose a suitable topology to establish convergence. By Theorem 2.4 in \citep{hambly2019spde}, $L^\varepsilon$ is continuous for every $\varepsilon > 0$, but the loss of the limiting process may in fact jump. Skorokhod's $M_1$-topology is sufficiently rich to facilitate the convergence of continuous functions to those with jumps. 
	
	
	\begin{remark}
		The equation \eqref{eq: EQUATION WITH SMOOTHING IN INTRODUCITON} has been posed in slightly more generality (with $\alpha(t, X_t^\varepsilon,\bm{\nu}_t^\varepsilon)$) in \cite{hambly2019spde}. As the convergence is strictly in the $M_1$-topology, not in the $J_1$-topology, we only consider $\alpha$ to be of the form $\alpha(t)$. If we considered it to also be a function of $X_t$ and/or $\bm{\nu}_t$, then we cannot expect to obtain an equation of the form \eqref{eq: FIRST EQUATION IN THE PAPER} in the limit as $\varepsilon \downarrow 0$, due to $X$, $\bm{\nu}$, and $L$ having jumps at the same time.    
	\end{remark}

	The theory in \citep{whitt2002stochastic} requires our c\`adl\`ag  processes to be uniformly right-continuous at the initial time point and left-continuous at the terminal time point, when working with functions on compact time domains. As we are starting from an arbitrary initial condition $X_{0-}$ which is positive almost surely, the limiting process may exhibit a jump immediately at time $0$ given sufficient mass near the boundary. For this reason, we shall embed the process from $D([0,T],\R)$ into $D([-1,\Bar{T}],\R)$, where $\Bar{T} = T + 1$, using the extension defined in \eqref{eq: HOW WE DEFINE THE EXTENSION TO STOCHASTIC PROCESSES}. Unless stated otherwise, for notational convenience we shall denote the latter space, $D([-1,\Bar{T}],\R)$, by $D_\R$. Recall, $\tilde{\mathbf{P}}^\varepsilon$ is defined to be the law of $\tilde{X}^\varepsilon$ conditional on $W^0$. That is $\tilde{\mathbf{P}}^\varepsilon \coloneqq \operatorname{Law}(\tilde{X}^\varepsilon \mid W^0)$.
	
	To show tightness and convergence of the collection of random measures $\{\tilde{\mathbf{P}}^\varepsilon\}_{\varepsilon > 0}$, we shall follow the ideas in \citep{delarue2015particle} and \citep{ledger2021mercy}. To begin, we first derive a Gr\"onwall-type estimate of the smoothed system uniformly in $\varepsilon$. These estimates are necessary to show the tightness of $\{\tilde{\mathbf{P}}^\varepsilon\}_{\varepsilon > 0}$ and the existence of a limiting random measure. In the following Proposition and its sequels, $C$ will denote a constant that may change from line to line, and we will denote the dependencies of the value of $C$ in its subscript. To further simplify notation, we use $Y_t^\varepsilon$, $Y_t^{0,\varepsilon}$, and $\mathcal{Y}_t^{\varepsilon}$ to denote 
	\begin{equation*}
		\int_0^t \sigma(u, X_u^\varepsilon)\sqrt{1 - \rho^2(u,\bm{\nu}^\varepsilon_u)} \diff W_u,\qquad \int_0^t \sigma(u,X_u^\varepsilon)\rho(u,\bm{\nu}^\varepsilon_u) \diff W^0_u, \quad \text{and} \quad Y_t^\varepsilon + Y_t^{0,\varepsilon}
	\end{equation*}
	respectively. We shall use $\tilde{Y}_t^\varepsilon$, $\tilde{Y}_t^{0,\varepsilon}$, and $\tilde{\mathcal{Y}}_t^{\varepsilon}$ to denote their corresponding extensions as defined in \eqref{eq: HOW WE DEFINE THE EXTENSION TO STOCHASTIC PROCESSES}.
	
	\begin{proposition}[Gr\"onwall upper bound]\label{prop: GRONWALL TYPE UPPERBOUND ON THE SUP OF THE MCKEAN VLASOV EQUATION WITH CONVOLUTION IN THE BANKING MODEL SIMPLIFIED AND INDEXED GENERALISED}
		For any $p \ge 1$ and $t \le T$, there exists a constant $C_{\alpha,b,p,T,\sigma} > 0$ independent of $\varepsilon > 0$ such that 
		\begin{equation}\label{eq: GRONWALL TYPE UPPERBOUND ON THE LOSS FUNCTION FOR THE SIMPLIFIED SYSTEM WITH MOLLIFIERS GENERALISED}
			\E \left[\underset{s \le T}{\sup}{ \nnnorm{X_s^\varepsilon}}^p\right] \le C_{\alpha,b,p,T,\sigma}.
		\end{equation}
	\end{proposition}
	
	\begin{proof}
		By the linear growth condition on $b$ and the triangle inequality, we have
		\begin{equation*}
			\nnnorm{X_t^\varepsilon} \le \nnnorm{X_{0-}} + C_b \int_0^t 1 + \underset{u \le s}{\sup} \nnnorm{X_u^\varepsilon} + \E\left[\left.\nnnorm{X^\varepsilon_{s \wedge \tau^\varepsilon}}\right|W^0\right] \diff s + \underset{s \le T}{\sup} \nnnorm{\mathcal{Y}^\varepsilon_s} + \norm{\alpha}_{\infty}.
		\end{equation*}
		By \citep[Lemma~A.3]{hambly2019spde}, $\int_0^T \E\nnnorm{X^\varepsilon_{s \wedge \tau^\varepsilon}}^p \diff s < \infty$ for any $p \ge 1$. Therefore, a simple application of Gr\"onwall's inequality shows that
		\begin{equation*}
			\underset{s \le t}{\sup}\nnnorm{X_s^\varepsilon} \le C_{T,b,\alpha} \left(\nnnorm{X_{0-}} + \int_0^t \E\left[\left.\nnnorm{X^\varepsilon_{s \wedge \tau^\varepsilon}}\right|W^0\right] \diff s + \underset{s \le T}{\sup} \nnnorm{\mathcal{Y}^\varepsilon_s} + 1\right).
		\end{equation*}
		By \eqref{ass: MODIFIED AND SIMPLIED FROM SOJMARK SPDE PAPER ASSUMPTIONS THREE} in Assumption \ref{ass: MODIFIED AND SIMPLIED FROM SOJMARK SPDE PAPER ASSUMPTIONS II}, $X_{0-}$ has finite $L^p$-moments for every $p > 0$. Furthermore, by employing the Burkholder-Davis-Gundy inequality to control $\E[\sup_{s \le T} \nnorm{\mathcal{Y}^\varepsilon_s}]$, we may deduce that $\E[\sup_{s \le t}\nnnorm{X_s^\varepsilon}^p] < \infty$ for all $t \ge 0$ and $p \ge 1$. Now, observing that $\nnnorm{X^\varepsilon_{t \wedge \tau^\varepsilon}} \le \sup_{s \le t}\nnnorm{X_s^\varepsilon}$, we have by the monotonicity of expectation and Jensen's inequality that
		\begin{align*}
			\underset{s \le t}{\sup}\nnnorm{X_s^\varepsilon}^p &\le C_{T,b,\alpha}^p \left(\nnnorm{X_{0-}} + \int_0^t \E\left[\left.\nnnorm{X^\varepsilon_{s \wedge \tau^\varepsilon}}\right|W^0\right] \diff s + \underset{s \le T}{\sup} \nnnorm{\mathcal{Y}^\varepsilon_s} + 1\right)^p\\
			&\le C_{T,b,\alpha,p} \left(\nnnorm{X_{0-}}^p + \int_0^t \E\left[\left.\underset{u \le s}{\sup}\nnnorm{X_u^\varepsilon}^p\right|W^0\right] \diff s + \underset{s \le T}{\sup} \nnnorm{\mathcal{Y}^\varepsilon_s}^p + 1\right).
		\end{align*}
		Taking expectations and applying Fubini's theorem followed by Gr\"onwall's inequality, we obtain
		\begin{equation}\label{eq: FIRST EQUATION IN prop: GRONWALL TYPE UPPERBOUND ON THE SUP OF THE MCKEAN VLASOV EQUATION WITH CONVOLUTION IN THE BANKING MODEL SIMPLIFIED AND INDEXED GENERALISED}
			\E\left[\underset{s \le T}{\sup}\nnnorm{X_s^\varepsilon}^p\right] \le C_{T,b,\alpha,p} \E \left[\nnnorm{X_{0-}}^p + \underset{s \le T}{\sup} \nnnorm{\mathcal{Y}^\varepsilon_s}^p + 1\right].
		\end{equation}
		Lastly, by the Burkholder-Davis-Gundy inequality and \eqref{ass: MODIFIED AND SIMPLIED FROM SOJMARK SPDE PAPER ASSUMPTIONS II TWO} from Assumption \ref{ass: MODIFIED AND SIMPLIED FROM SOJMARK SPDE PAPER ASSUMPTIONS II}, we may bound \eqref{eq: FIRST EQUATION IN prop: GRONWALL TYPE UPPERBOUND ON THE SUP OF THE MCKEAN VLASOV EQUATION WITH CONVOLUTION IN THE BANKING MODEL SIMPLIFIED AND INDEXED GENERALISED} independently of $\varepsilon$. This completes the proof.
	\end{proof}

	The collection $\{\tilde{\mathbf{P}}^\varepsilon\}_{\varepsilon > 0}$ consists of $\mathcal{P}(D_\mathbb{R})$-valued random measures. To show that this collection is tight, we need to find, for any $\gamma>0$, a compact set $K_\gamma$ in $\mathcal{P}(D_\R)$ so that $\operatorname{Law}(\tilde{\mathbf{P}}^\varepsilon)(K_\gamma) = \mathbb{P}( \tilde{\mathbf{P}}^\varepsilon \in K_\gamma)\le \gamma$ for all $\varepsilon>0$. In other words, we need the probability measures $\{\operatorname{Law}(\tilde{\mathbf{P}}^\varepsilon)\}_{\varepsilon>0}$ to be tight in $\mathcal{P}(\mathcal{P}(D_\mathbb{R}))$. Rather than tackling this directly, we follow an indirect approach of first showing that the measures $\{\operatorname{Law}(\tilde{X}^\varepsilon)\}_{\varepsilon > 0}$ are tight in $\mathcal{P}(D_\mathbb{R})$ and then we construct $K_\gamma$ from there. Due to how the processes are defined, the latter follows easily from properties of the $M_1$-topology and \citep[Theorem~1]{AVRAM198963}.
	
	\color{black}
	\begin{proposition}[Tightness of smoothed random measures]\label{prop: TIGHTHNESS OF THE EMPIRICAL MEASURES IN THE SIMPLIFIED BANKING MODEL WITH MOLLIFICATIONS}
		Let $\tmwk$ denote the topology of weak convergence on $\mathcal{P}(D_\R)$ induced by the $M_1$-topology on $D_\R$. Then, the collection $\{ \tilde{\mathbf{P}}^\varepsilon \}_{\varepsilon>0} =\{\operatorname{Law}(\tilde{X}^\varepsilon \mid W^0)\}_{\varepsilon > 0}$ is tight on $(\mathcal{P}(D_\R),\tmwk)$ under Assumption \ref{ass: MODIFIED AND SIMPLIED FROM SOJMARK SPDE PAPER ASSUMPTIONS II}.
	\end{proposition}
	
	\begin{proof}
		Define $\tilde{P}^\varepsilon \coloneqq \operatorname{Law}(\tilde{X}^\varepsilon)$. By \citep[Theorem~12.12.3]{whitt2002stochastic}, we need to verify two conditions to show the tightness of the measures on $D_\R$ endowed with the $M_1$-topology:
		\begin{enumerate}[(i)]
			\item $\lim_{\lambda \to \infty} \sup_{\varepsilon > 0} \tilde{P}^\varepsilon\left(\left\{x \in D_\R\;:\; \norm{x}_{\infty} > \lambda \right\}\right) = 0$.
			\item For any $\eta > 0$, we have $\lim_{\delta \to 0} \sup_{\varepsilon > 0} \tilde{P}^\varepsilon\left(\left\{x \in D_\R\;:\; w_{M_1}(x,\delta) \ge \eta \right\}\right) = 0$, where $w_{M_1}$ is the oscillatory function of the $M_1$-topology, defined as in \citep[Section~12, Equation~12.2]{whitt2002stochastic}.
		\end{enumerate}
		To show the first condition, we observe that by the definition of the extension of our process, we have
		\begin{equation*}
			\sup_{t \le \Bar{T}}\nnorm{\tilde{X}_t^\varepsilon}
			\le \sup_{t \le T}\nnorm{X_t^\varepsilon}
			+\sup_{t \le 1} \nnorm{ W_{T  + t} - W_T}
			+ \norm{\alpha}_{\infty}.
		\end{equation*}
		Then, it is clear by Markov's inequality and Proposition \ref{prop: GRONWALL TYPE UPPERBOUND ON THE SUP OF THE MCKEAN VLASOV EQUATION WITH CONVOLUTION IN THE BANKING MODEL SIMPLIFIED AND INDEXED GENERALISED} that for any $\lambda > 0$, 
		\begin{equation*}
			\prob\left[\sup_{t \le \Bar{T}} \nnorm{\tilde{X}_t^\varepsilon} > \lambda \right] = O(\lambda^{-1})
		\end{equation*}
		uniformly in $\varepsilon$. Therefore, by taking the supremum over $\varepsilon$ and then the $\limsup$ over $\lambda$, the first condition holds. We shall not show the second condition directly. By \citep[Theorem~1]{AVRAM198963}, the second condition is equivalent to showing:
		\begin{enumerate}[(I)]
			\item There exists some $C > 0$, uniformly in $\varepsilon$, such that $\prob\left[\mathcal{H}_\R(\tilde{X}_{t_1}^\varepsilon,\tilde{X}_{t_2}^\varepsilon,\tilde{X}_{t_3}^\varepsilon)\ge \eta \right] \le C\eta^{-4}\nnorm{t_3 - t_1}^2$ for all $\eta > 0$ and $-1 \le t_1 \le t_2 \le t_3 \le \Bar{T}$ where $\mathcal{H}_\R(x_1,x_2,x_3) = \inf_{\lambda \in [0,1]}\nnorm{x_2 - (1- \lambda)x_1 - \lambda x_3}$.
			\item \sloppy{$\lim_{\delta \to 0} \sup_{\varepsilon > 0} \prob\left[\sup_{t \in (-1,-1+\delta)}\nnorm{\tilde{X}_t^\varepsilon - \tilde{X}_{-1}} + \sup_{t \in (\Bar{T} - \delta,\Bar{T})}\nnorm{\tilde{X}_t^\varepsilon - \tilde{X}_{\Bar{T}}} \ge \eta\right] = 0$ for all $\eta > 0$.}
		\end{enumerate}
		Note that by Assumption \ref{ass: MODIFIED AND SIMPLIED FROM SOJMARK SPDE PAPER ASSUMPTIONS II}, we have $\alpha$ is non-decreasing and non-negative. Therefore, by the properties of Lebesgue-Stielitjes integration, $t \mapsto \int_0^t \alpha(s) \diff \mathfrak{L}_s^\varepsilon$ is non-decreasing. As monotone functions are immaterial to the $M_1$ modulus of continuity, we have
		\begin{equation*}
			\mathcal{H}_\R(\tilde{X}_{t_1}^\varepsilon,\tilde{X}_{t_2}^\varepsilon,\tilde{X}_{t_3}^\varepsilon) \le \nnnorm{Z_{t_1} - Z_{t_2}} + \nnorm{Z_{t_2} - Z_{t_3}},
		\end{equation*}
		where $Z$ is given by
		\begin{equation*}
			Z_t = X_{0-} + \int_0^{t \wedge T} b(u,X_u^\varepsilon,\bm{\nu}_u^\varepsilon) \diff u + \tilde{\mathcal{Y}}_t
		\end{equation*}
		for $t \ge 0$ and $Z_t = X_{0-}$ for $t < 0$. Hence, to show (I), it is sufficient to bound the increments of $Z$. Note that when $s < t < -1$, $Z$ is constant. Therefore, trivially we have $\E \left[\nnnorm{Z_t - Z_s}^4\right] \le C(t-s)^2$ for any $C > 0$. When  $0 \le s < t,$ by the formula above for $Z$ we have that
		\begin{equation}\label{eq: EQUATION ONE IN THIGTNESS PROPOSITION OF THE LAWS OF THE SIMPLIED AND INDEXED BANKING MODEL FROM SOJMARK SPDE PAPER}
			Z_t - Z_s = \int_{s \wedge T}^{t \wedge T} b(u,X_u^\varepsilon,\bm{\nu}_u^\varepsilon) \diff u + \tilde{\mathcal{Y}}_t - \tilde{\mathcal{Y}}_s.
		\end{equation}
		Employing the linear growth condition on $b$ and Proposition \ref{prop: GRONWALL TYPE UPPERBOUND ON THE SUP OF THE MCKEAN VLASOV EQUATION WITH CONVOLUTION IN THE BANKING MODEL SIMPLIFIED AND INDEXED GENERALISED},
		\begin{equation}\label{eq: EQUATION TWO IN THIGTNESS PROPOSITION OF THE LAWS OF THE SIMPLIED AND INDEXED BANKING MODEL FROM SOJMARK SPDE PAPER}
			\E\left[\nnnorm{ \int_{s \wedge T}^{t \wedge T} b(u,X_u^\varepsilon,\bm{\nu}^\varepsilon_u) \diff u}^4\right] \le C(t - s)^4\left(1 + \E\left[\sup_{u \le T} \nnorm{X_u^\varepsilon}^4\right]\right) = O((t - s)^2)
		\end{equation}
		uniformly in $\varepsilon$. By the Burkholder-Davis-Gundy inequality and the upper bound on $\sigma$, it is clear that
		\begin{equation}\label{eq: EQUATION THREE IN THIGTNESS PROPOSITION OF THE LAWS OF THE SIMPLIED AND INDEXED BANKING MODEL FROM SOJMARK SPDE PAPER}
			\E\left[\nnnorm{ \tilde{\mathcal{Y}}_t - \tilde{\mathcal{Y}}_s}^4\right] = O((t - s)^2)
		\end{equation}
		uniformly in $\varepsilon$. Therefore, by Markov's inequality,
		\begin{align*}
			\prob\left[\mathcal{H}_\R(\tilde{X}_{t_1}^\varepsilon,\tilde{X}_{t_2}^\varepsilon,\tilde{X}_{t_3}^\varepsilon)\ge \eta \right]
			&\le \eta^{-4}\E\left[\mathcal{H}_\R(\tilde{X}_{t_1}^\varepsilon,\tilde{X}_{t_2}^\varepsilon,\tilde{X}_{t_3}^\varepsilon)^4\right]\\
			&\le C \eta^{-4}\E\left[ \nnnorm{Z_{t_1} - Z_{t_2}}^4 + \nnorm{Z_{t_2} - Z_{t_3}}^4\right]\\
			&\le C \eta^{-4} \left((t_2 - t_1)^2 + (t_3 - t_2)^2\right)\\
			&\le C \eta^{-4}(t_3 - t_1)^2,
		\end{align*}
		where all the constants hold uniformly in $\varepsilon$. To verify the second condition, we observe that for any $\eta > 0$ and $\delta < 1$,
		\begin{align*}
			\prob\left[\sup_{t \in (-1,-1+\delta)}\nnorm{\tilde{X}_t^{\varepsilon} - \tilde{X}^\varepsilon_{-1}}\ge \frac{\eta}{2}\right] &= \prob\left[\sup_{t \in (-1,-1+\delta)}\nnorm{X_{0-} - X_{0-}}\ge \frac{\eta}{2}\right] = 0, \textnormal{ and}\\
			\prob\left[\sup_{t \in (\Bar{T} - \delta,\Bar{T})}\nnorm{\tilde{X}_t^\varepsilon - \tilde{X}_{\Bar{T}}} \ge \frac{\eta}{2}\right] &= \prob\left[\sup_{t \in (\Bar{T} - \delta,\Bar{T})}\nnorm{W_t - W_{\Bar{T}}} \ge \frac{\eta}{2}\right] = O(\delta^2),
		\end{align*}
		uniformly in $\varepsilon$. Hence, we have shown that 
		\begin{equation*}
			\prob\left[\sup_{t \in (-1,-1+\delta)}\nnorm{\tilde{X}_t^\varepsilon - \tilde{X}^\varepsilon_{-1}} + \sup_{t \in (\Bar{T} - \delta,\Bar{T})}\nnorm{\tilde{X}_t^\varepsilon - \tilde{X}^\varepsilon_{\Bar{T}}} \ge \eta\right] = O(\delta^{2}) \quad \text{for all} \quad \eta > 0,\,\delta < 1.
		\end{equation*}
		Therefore, together, conditions (I) and (II) show
		\begin{equation*}
			\sup_{\varepsilon > 0} \tilde{P}^\varepsilon\left(\left\{x \in D_\R\;:\; w_{M_1}(x,\delta) \ge \eta \right\}\right) = O(\delta^{2}),
		\end{equation*}
		for all $\delta < 1$ uniformly in $\varepsilon$. This shows condition (ii). 
		
		Based on the above, we can now employ Markov's inequality and Prokhorov's theorem to deduce that $\{ \operatorname{Law}(\tilde{\mathbf{P}}^\varepsilon) \}_{\varepsilon > 0}$ is tight in $\mathcal{P}(\mathcal{P}(D_\mathbb{R}))$. To begin, fix a $\gamma > 0$. For any $l,k \in \N$, we may find a $\lambda_l,\,\delta_{k,l} > 0$ such that
		\begin{align*}
			\tilde{P}^\varepsilon(A^\complement_{k,l}) &< \gamma2^{-(k + 2l+1)} \quad \forall k\in \N \qquad \text{uniformly in $\varepsilon$,}\\
			\text{where} \qquad \qquad \qquad A_{0,l} &= \{x \in D_\R\;:\; \norm{x} \le \lambda_l\},\\
			A_{k,l} &= \left\{x\in D_\R\;:\; w_{M_1}(x,\delta_{k,l}) < \frac{1}{k + 2l}\right\}.
		\end{align*}
		We define $A_l = \cap_{k \ge 0} A_{k,l}$. By \citep[Theorem~12.12.2]{whitt2002stochastic}, $A_l$ has compact closure in the $M_1$-topology. The closure of $A_l$ is denoted by $\Bar{A}_{l}$. Furthermore, by construction $\tilde{P}^\varepsilon(A_l^\complement) \le \sum_{k \ge 0} \tilde{P}^\varepsilon(A_{k,l}^\complement) \le \gamma4^{-l}$. By the subadditivity of measures and Markov's inequality,
		\begin{equation}\label{eqn: CHAIN OF EQUALITIES SHOW THAT WE HAVE TIGHTNESS OF OUR CONDITIONAL MEASURES}
			\prob\left[\bigcup_{l = 1}^\infty \left\{\tilde{\mathbf{P}}^{\varepsilon}(A^\complement_{l}) > 2^{-l}\right\}\right] \le \sum_{l \ge 1} \prob\left[\tilde{\mathbf{P}}^{\varepsilon}(A^\complement_l) > 2^{-l}\right]
			\le \sum_{l \ge 1} 2^{l}\E\left[\tilde{\mathbf{P}}^{\varepsilon}(A^\complement_l)\right]
			\le \sum_{l \ge 1} \frac{\gamma 2^l}{4^l} = \gamma. 
		\end{equation}
		Finally, let $K \coloneqq \{\mu \in \mathcal{P}(D_\R)\,:\, \mu(\Bar{A}_{l}^\complement) \le 2^{-l} \, \forall \, l \in \N\}$. As $D_\R$ endowed with the $M_1$-topology is a Polish space, Prokhorov's theorem may be applied, and it suffices to show that the set of measures $K$ is tight; hence, $K$ will then have compact closure in $\mathcal{P}(D_\R)$ by Prokhorov's theorem. It is clear by construction that the set of measures $K$ is tight as the sets $\Bar{A}_{l}$ are compact in $D_\R$ endowed with the $M_1$-topology. By \eqref{eqn: CHAIN OF EQUALITIES SHOW THAT WE HAVE TIGHTNESS OF OUR CONDITIONAL MEASURES}, we have $\operatorname{Law}(\tilde{\mathbf{P}}^\varepsilon)(\Bar{K}^\complement) \le \gamma$, uniformly in $\varepsilon$. As $\gamma$ was arbitrary, this completes the proof.
	\end{proof}
	
	
	\subsection{Continuity of hitting times}
	
	Note that $(D_\R,M_1)$ is a Polish space by \citep[Theorem~12.8.1]{whitt2002stochastic} and its Borel $\sigma$-algebra is generated by the marginal projections, \citep[Theorem~11.5.2]{whitt2002stochastic}. Hence, the topological space $(\mathcal{P}(\DR), \tmwk)$ is also a Polish space. Therefore, by invoking Prokhorov's Theorem, \citep[Theorem~5.1]{billingsley2013convergence}, tightness is equivalent to being sequentially pre-compact. So, we may choose a weakly convergent subsequence $\{\tilde{\mathbf{P}}^{\varepsilon_n}\}_{n \ge 1}$ for a positive sequence $(\varepsilon_n)_{n \ge 1}$ which converges to zero. Let $\mathbf{P}^*$ denote the limit point of this sequence. Using this limit point, we will construct a probability space and a stochastic process that will be a solution to \eqref{eq: SINGULAR RELAXED MCKEAN VLASOV EQUATION IN THE SYSTEMIC RISK MODEL WITH COMMON NOISE GENERALISED}. 
	
	Before proceeding, we seek to show that for a co-countable set of times $t$, $L_t^{\varepsilon_n} = \tilde{\mathbf{P}}^{\varepsilon_n}(\tnot \le t)$ converges weakly to $L_t^* \coloneqq \mathbf{P}^*(\tnot \le t)$, where $\tau_0$ is a function on $\DR$ whose value is the first hitting time of $0$. To be explicit, we define
	\begin{equation*}
		\mathbb{T} \coloneqq \left\{t \in [-1,\Bar{T}]\;:\; \E \left[\mathbf{P}^*(\eta_t = \eta_{t-})\right] = 1, \E \left[\mathbf{P}^*(\tnot = t)\right] = 0 \right\},
	\end{equation*}
	and
	\begin{equation}\label{eq: FIRST HITTING TIME OF ZERO OF THE CANONICAL PROCESS ON THE SPACE OF CADLAG FUNCTIONS}
		\tau_0(\eta) \coloneqq \inf\{t \ge -1\,:\, \eta_t \le 0\}
	\end{equation}
	with the convention that $\inf\{\emptyset\} = \Bar{T}$. Our first result is that for $\operatorname{Law}(\mathbf{P}^*)$-almost every measure $\mu$, $\mu$-almost every path $\eta \in \DR$ is constant on the interval $[-1,0)$.
	
	\begin{lemma}\label{lem: LEMMA STATING THAT P* IS SUPPORTED ON PATHS THAT ARE CONSTANT BEFORE TIME 0}
		For $\operatorname{Law}(\mathbf{P}^*)$-almost every measure $\mu$, we have that $\sup_{s < 0} \nnorm{\eta_s - \eta_{-1}} = 0$ for $\mu$-almost every path $\eta$.
	\end{lemma}
	
	\begin{proof}
		As $(\mathcal{P}(\DR),\tmwk)$ is a Polish space, we may apply Skorokhod's Representation Theorem \citep[Theorem~6.7]{billingsley2013convergence}. Hence, there exists a common probability space, and $\mathcal{P}(\DR)$-valued random variables $(\mathbf{Q}^n)_{n \ge 1}$ and $\mathbf{Q}^*$ such that
		\begin{equation*}
			\operatorname{Law}(\mathbf{Q}^n) = \operatorname{Law}(\tilde{\mathbf{P}}^{\varepsilon_n}), \qquad \operatorname{Law}(\mathbf{Q}^*) = \operatorname{Law}(\mathbf{P}^{*}), \quad \textnormal{and} \quad \mathbf{Q}^n \to \mathbf{Q}^* \quad \textnormal{a.s.} 
		\end{equation*}
		It is straightforward to see-by \citep[Theorem~13.4.1]{whitt2002stochastic}-the following maps from $\DR$ into itself \begin{equation*}
			\eta \mapsto \left(t \mapsto \inf_{s \le t} \{\eta_s - \eta_{-1}\}\right) \qquad \qquad \eta \mapsto \left(t \mapsto \sup_{s \le t} \{ \eta_s - \eta_{-1}\}\right)
		\end{equation*}
		are continuous. Now, for a $t \in \mathbb{T} \cap (-1,0)$, the maps $c_t$ and $\tilde{c}_t$ from $D_\R$ onto $\R$ such that
		\begin{equation*}
			c_t(\eta) = \inf_{s \le t} \{\eta_s - \eta_{-1}\} \qquad \qquad \tilde{c}_t(\eta) = \sup_{s \le t} \{\eta_s - \eta_{-1}\}
		\end{equation*}
		are continuous. Therefore, by the Continuous Mapping Theorem, $c_t^\#\mathbf{Q}^{n} \to c_t^\#\mathbf{Q}^*$ and $\tilde{c}_t^\#\mathbf{Q}^{n} \to \tilde{c}_t^\#\mathbf{Q}^*$ almost surely in $(\mathcal{P}(D_\R),\tmwk)$. 
		
		Fix a $\gamma > 0$. Then by the Portmanteau Theorem and Fatou's lemma, we have
		\begin{align*}
			\E\left[\mathbf{P}^*\left(\inf_{s \le t} \{ \eta_s - \eta_{-1}\} < - \gamma\right)\right] &= \E\left[\mathbf{Q}^*\left(\inf_{s \le t} \{ \eta_s - \eta_{-1}\} < - \gamma\right)\right]\\
			&\le \E\left[\liminf_{n \to \infty}\mathbf{Q}^n\left(\inf_{s \le t} \{ \eta_s - \eta_{-1}\} < - \gamma\right)\right]\\
			&\le \liminf_{n \to \infty} \E\left[\mathbf{Q}^n\left(\inf_{s \le t} \{ \eta_s - \eta_{-1}\} < - \gamma\right)\right]\\
			&= \liminf_{n \to \infty} \E\left[\tilde{\mathbf{P}}^{\varepsilon_n}\left(\inf_{s \le t} \{ \eta_s - \eta_{-1}\} < - \gamma\right)\right]\\
			&=  \liminf_{n \to \infty} \tilde{P}^{\varepsilon_n}\left(\inf_{s \le t} \{\eta_s - \eta_{-1}\} < -\gamma\right) = 0,
		\end{align*}
		where the last equality follows from the embedding of $X^{\varepsilon_n}$ from $D([0,T],\R)$ into $D([-1,\Bar{T}],\R)$. So, by continuity of measure and the Monotone Convergence Theorem, as $\gamma$ was arbitrary, we have
		\begin{equation*}
			\E\left[\mathbf{P}^*\left(\inf_{s \le t} \{ \eta_s - \eta_{-1}\} < 0 \right)\right] = 0.
		\end{equation*}
		Similarly, $ \E\left[\mathbf{P}^*\left(\sup_{s \le t} \{ \eta_s - \eta_{-1}\} > 0 \right)\right] = 0$.
	\end{proof}
	
	As $\tilde{X}^\varepsilon$ is fundamentally a time-changed Brownian motion with drift, it is not hard to show that, with probability one, $\tilde{X}^\varepsilon$ will take a negative value on any open neighbourhood of its first hitting time of zero. This property is preserved by weak convergence for almost every realisation of $\mathbf{P}^*$. Furthermore, as the Lebesgue-Stieltjes integral $\int_0^t \alpha(s)\diff\mathfrak{L}^\varepsilon_s$ takes non-negative values, by weak convergence we expect that for $\operatorname{Law}(\mathbf{P}^*)$-almost every measure $\mu$, $\mu$-almost every path $\eta \in \DR$ will have only downward jumps.
	
	\begin{lemma}[Strong crossing property]\label{lem:sec2: LEMMA STATING THE STRONG CROSSING PROPERTY OF THE LIMITING RANDOM MEASURE P*}
		For any $h > 0$, we have
		\begin{align}
			&\E\left[\mathbf{P}^*\left(\inf_{s \in (\tnot,(\tnot + h)\wedge \Bar{T})}\{\eta_s - \eta_{\tnot}\} \ge 0,\, \tnot < \Bar{T}\right)\right] = 0,\\ \label{eq:sec2: EQUATION STATING THE STRONG CROSSING PROPERTY OF LIMITING MEASURE P* IN lem:sec2: LEMMA STATING THE STRONG CROSSING PROPERTY OF THE LIMITING RANDOM MEASURE P*}
			&\E\left[\mathbf{P}^*\left(\eta\,:\, \Delta \eta_t \le 0 \quad \forall\, t \le \Bar{T}\right)\right] = 1.
		\end{align}
	\end{lemma}
	
	\begin{proof}
		As with the space of \cadlag functions, we shall employ the shorthand notation $\mathcal{C}_\R$ for this proof to denote $\mathcal{C}([-1,\,\Bar{T}]\,,\,\R)$. Now, as $\sigma$ is non-degenerate and bounded by assumption, by the Kolmogorov-Chentsov Tightness Criterion, \citep{kunita1986tightness} and \citep{chentsov1956weak}, we have that $(\tilde{\mathcal{Y}}^\varepsilon)_{\varepsilon >0}$ is tight. Additionally, we define the random variable $\tilde{Z}^\varepsilon \coloneqq \langle\tilde{\mathbf{P}}^\varepsilon,\,\sup_{u \le \Bar{T}}\nnorm{\eta_u}\rangle$. By definition of $\tilde{\mathbf{P}}^\varepsilon$, $\E[\tilde{Z}^\varepsilon] = \E[\sup_{u \le \Bar{T}}\nnorm{\tilde{X}^\varepsilon_u}]$. Therefore, $\E[\tilde{Z}^\varepsilon]$ is uniformly bounded by Proposition \ref{prop: GRONWALL TYPE UPPERBOUND ON THE SUP OF THE MCKEAN VLASOV EQUATION WITH CONVOLUTION IN THE BANKING MODEL SIMPLIFIED AND INDEXED GENERALISED} and hence $\{\tilde{Z}^\varepsilon\}_{\varepsilon > 0}$ is tight on $\R$.
		
		As marginal tightness implies joint tightness, we have $\prob_{x,y,z}^{\varepsilon} \coloneqq \operatorname{Law}(\tilde{X}^\varepsilon,\,\tilde{\mathcal{Y}}^\varepsilon,\tilde{Z}^\varepsilon)$ is tight in $\mathcal{P}(\DR \times \mathcal{C}_\R \times \R)$ by Proposition \ref{prop: TIGHTHNESS OF THE EMPIRICAL MEASURES IN THE SIMPLIFIED BANKING MODEL WITH MOLLIFICATIONS}. Given a suitable subsequence, also denoted by $(\varepsilon_n)_{n \ge 1}$ for simplicity, we have $\tilde{\mathbf{P}}^{\varepsilon_n} \implies \mathbf{P}^*$ and $\prob_{x,y,z}^{\varepsilon_n} \implies \prob_{x,y,z}^*$. Here $\prob_x^*$ and $\prob_y^*$ are used to denote the first and second marginal respectively.
		
		Intuitively, $\E[\mathbf{P}^*(\cdot)]$ and $\prob_x^*$ should have the same law as we are averaging over the stochasticity inherited by the common noise. By definition of $\mathbf{P}^{\varepsilon}$ and $\prob_{x,y,z}^{\varepsilon}$, for any continuous bounded function $f \,:\, \DR \to \R$, we have that
		\begin{equation*}
			\E\left[\langle\mathbf{P}^{\varepsilon_n},\,f\rangle\right] = \langle\prob_{x,y,z}^{\varepsilon_n},\,f\rangle.
		\end{equation*}
		As $\DR$ is a Polish space, by a Monotone Class Theorem argument and Dynkin's Lemma, we have
		\begin{equation}\label{eq:sec2 FIRST EQUATION IN lem:sec2: LEMMA STATING THE STRONG CROSSING PROPERTY OF THE LIMITING RANDOM MEASURE P*}
			\E\left[\mathbf{P}^*(A)\right] = \prob_x^*(A) = \prob_{x,y,z}^*(A\times \mathcal{C}_\R \times \R) \qquad \forall\, A \in \mathcal{B}(\DR).
		\end{equation}
		
		Define the canonical processes $X^*$, $Y^*$ and $Z^*$ on $(\DR,M_1) \times (\mathcal{C}_\R,\norm{\cdot}_\infty) \times (\R,\,\nnorm{\cdot})$, where for $(\eta,\omega,z) \in \DR \times \mathcal{C}_\R \times \R,\,X^*(\eta,\omega,z) = \eta$, $Y^*(\eta,\omega,z) = \omega$ and $Z^*(\eta,\omega,z) = z$. By considering the parametric representations, the map $\eta \mapsto \sup_{u \le \Bar{T}} \nnorm{\eta_u}$ is $M_1$-continuous for any $\eta \in \DR$. Hence, by the linear growth condition on $b$, the Continuous Mapping Theorem, and the Portmanteau Theorem, for any $s,\,t \in \mathbb{T}$ with $s < t$ and $ \gamma > 0$,  
		\begin{equation}\label{eq: SECOND EQUATION IN CROSSING PROPERTY LEMMA}
			\begin{aligned}
				&\prob_{x,y,z}^*\left(X^*_t - X^*_s \le Y_t^* - Y_s^* + C_b(t - s)(1 + \sup_{u \le \Bar{T}}\nnorm{X_u^*} + Z^*) + \gamma\right)\\
				&\qquad \ge \limsup_{n \to \infty} \prob\left[\tilde{X}^{\varepsilon_n}_t - \tilde{X}^{\varepsilon_n}_s \le \mathcal{\tilde{Y}}_t^{\varepsilon_n} - \mathcal{\tilde{Y}}_s^{\varepsilon_n} + C_b(t - s)(1 + \sup_{u \le \Bar{T}}\nnorm{\tilde{X}_u^{\varepsilon_n}} + \tilde{Z}^{\varepsilon_n}) + \gamma\right]\\
				&\qquad = 1.
			\end{aligned}
		\end{equation}
		The last equality follows from the fact that for any $\varepsilon_n$
		\begin{align*}
			\tilde{X}^{\varepsilon_n}_t - \tilde{X}^{\varepsilon_n}_s &= \mathcal{\tilde{Y}}_t^{\varepsilon_n} - \mathcal{\tilde{Y}}_s^{\varepsilon_n} + \int^{(t \vee 0)\wedge T}_{(s \vee 0)\wedge T} b(u,X_u^{\varepsilon_n},\bm{\nu}_u^{\varepsilon_n})\1_{[0,T]}(u) \diff u - \int^{(t \vee 0)\wedge T}_{(s \vee 0)\wedge T}\alpha(u) \diff\mathfrak{L}_u^{\varepsilon_n}\\
			&\le \mathcal{\tilde{Y}}_t^{\varepsilon_n} - \mathcal{\tilde{Y}}_s^{\varepsilon_n} + C_b(t - s)\left(1 + \sup_{u \le \Bar{T}}\nnorm{\tilde{X}_u^{\varepsilon_n}}+\tilde{Z}^{\varepsilon_n}\right).
		\end{align*}
		Sending $\gamma \to 0$ countably and employing the right continuity of $X^*$ and $Y^*$, we deduce
		\begin{equation*}
			X^*_t - X^*_s \le Y_t^* - Y_s^* + C_b(t - s)(1 + \sup_{u \le \Bar{T}}\nnorm{X_u^*} + Z^*) \qquad \forall\, s < t \quad \prob^*_{x,y,z}\text{-a.s.}
		\end{equation*}
		Furthermore, $\Delta X_t^* \le 0$ for all $t$ $\prob_{x,y,z}^*$-almost surely. By Lemma \ref{app:lem: Y* IS  A CONTINUOUS LOCAL MARTINGALE}, $Y^*$ is a continuous local martingale with respect to the filtration generated by $(X^*, Y^*)$. It is clear that $\tau_0(X^*)$ is a stopping time with respect to the filtration generated by $(X^*, Y^*)$. So the claim follows from Lemma \ref{app:lem: GENERAL STRONG CROSSSING PROPERTY RESULT} if $\tau_0(X^*) \ge 0$ and $\E[\sup_{u \le \Bar{T}}\nnorm{X_u^*} + Z^*] < \infty$. For the former condition, it is sufficient to show
		\begin{equation*}
			\prob_x^*\left(\inf_{s < 0} \eta_s \le 0\right) = 0.
		\end{equation*}
		As $\prob^{\varepsilon_n}_x \implies \prob^*_x$, then by Skorokhod's Representation Theorem, there exist $(Z^n)$ and $Z$ on a common probability space such that $\operatorname{Law}(Z^n) = \prob^{\varepsilon_n}_x$, $\operatorname{Law}(Z) = \prob^{*}_x$, and $Z^n \to Z$ almost surely in $(\DR,M_1)$. By the Portmanteau Theorem, for any $\gamma > 0$
		\begin{equation*}
			\prob\left[Z_{-1} < \gamma \right] \le \liminf_{n \to \infty} \prob\left[Z_{-1}^n < \gamma\right] = \prob\left[X_{0-} < \gamma\right] = O(\gamma^{1/2}),
		\end{equation*}
		as $X_{0-}$ has an $L^2$-density by Assumption \ref{ass: MODIFIED AND SIMPLIED FROM SOJMARK SPDE PAPER ASSUMPTIONS II} (\ref{ass: MODIFIED AND SIMPLIED FROM SOJMARK SPDE PAPER ASSUMPTIONS THREE}). So $\prob^*_x(\eta_{-1} \le 0) = \prob\left[Z_{-1} \le 0 \right] = 0$. By Lemma \ref{lem: LEMMA STATING THAT P* IS SUPPORTED ON PATHS THAT ARE CONSTANT BEFORE TIME 0} and \eqref{eq:sec2 FIRST EQUATION IN lem:sec2: LEMMA STATING THE STRONG CROSSING PROPERTY OF THE LIMITING RANDOM MEASURE P*}, we have
		\begin{equation*}
			1 = \E\left[\mathbf{P}^*\left(\sup_{s < 0}\nnnorm{\eta_s - \eta_{-1}} = 0\right)\right] = \prob^*_x\left(\sup_{s < 0}\nnnorm{\eta_s - \eta_{-1}} = 0\right).
		\end{equation*}
		Therefore, $X^*_s > 0$ for every $s \in [-1,0)$ $\prob^*_{x,y,z}$-almost surely. Hence, $\tau_0(X^*) \ge 0$ almost surely. Furthermore, as $\eta \mapsto \sup_{u \le \Bar{T}} \nnorm{\eta_u}$ is an $M_1$-continuous map, $\E[\sup_{u \le \Bar{T}}\nnorm{X_u^*} + Z^*] < \infty$ follows from a simple application of the Continuous Mapping Theorem and Proposition \ref{prop: GRONWALL TYPE UPPERBOUND ON THE SUP OF THE MCKEAN VLASOV EQUATION WITH CONVOLUTION IN THE BANKING MODEL SIMPLIFIED AND INDEXED GENERALISED}. Therefore, we deduce,
		{\small\begin{align*}
				&\E\left[\mathbf{P}^*\left(\inf_{s \in (\tnot,(\tnot + h)\wedge \Bar{T})}\{\eta_s - \eta_{\tnot}\} \ge 0,\, \tnot < \Bar{T}\right)\right] \\
				&\qquad \le \prob^*_{x,y,z}\left(\inf_{s \in (\tau_0^*,(\tau_0^* + h)\wedge \Bar{T})} \left\{Y^*_s - Y^*_{\tau_0^*} + C_b(s - {\tau_0^*})(1 + \sup_{u \le \Bar{T}}\nnorm{X_u^*} + Z^*) \right\}\ge 0,\, \tau_0^* < \Bar{T}\right)\\
				&\qquad = 0,
		\end{align*}}
		where $\tau_0^* = \tau_0(X^*)$ and the final equality is due to Lemma \ref{app:lem: GENERAL STRONG CROSSSING PROPERTY RESULT}.
	\end{proof}
	
	Now we have all the ingredients to show that $\tau_0$ is an $M_1$-continuous map.
	
	\begin{corollary}[Hitting time continuity]\label{cor:sec2: LEMMA THAT THE HITTING TIME IS A M1 CONTINUOUS MAP}
		For $\operatorname{Law}(\mathbf{P}^*)$-almost every measure $\mu$, we have that the hitting time map $\tau_0\,:\, \DR \to \R$ is continuous in the $M_1$-topology for $\mu$-almost every $\eta \in \DR$.
	\end{corollary}
	
	\begin{proof}
		By Lemma \ref{lem:sec2: LEMMA STATING THE STRONG CROSSING PROPERTY OF THE LIMITING RANDOM MEASURE P*}, for $\operatorname{Law}(\mathbf{P}^*)$-almost every measure $\mu$, $\mu$-almost every path $\eta \in \DR$ will have only downward jumps and one of the following conditions holds:
		\begin{enumerate}
			\item\label{enum: FIRST ITEM IN CONTINUITY PROOF} $\tnot < \Bar{T}$ and $\eta$ takes a negative value on any neighbourhood of $\tnot$,
			\item\label{enum: SECOND ITEM IN CONTINUITY PROOF} $\tnot = \Bar{T}$ and $\inf_{s \le \Bar{T}} \eta_s > 0$,
			\item\label{enum: THIRD ITEM IN CONTINUITY PROOF} $\tnot = \Bar{T}$ and $\eta_{\Bar{T}} = 0$.
		\end{enumerate}
		
		If \ref{enum: FIRST ITEM IN CONTINUITY PROOF} holds, then by Lemma \ref{app:lem: CONVERGENCE OF THE STOPPING TIME FOR PARTICLES THAT EXHIBIT THE CROSSING PROPERTY}, $\tau_0$ is $M_1$-continuous at $\eta$. If \ref{enum: SECOND ITEM IN CONTINUITY PROOF} holds--$\tnot = \Bar{T}$ and $\inf_{s \le \Bar{T}} \eta_s > 0$--then for any approximating sequence $(\eta^n)_{n \ge 1} \subset \DR$ in the $M_1$-topology, we must have $\inf_{s \le \Bar{T}} \eta_s^n > 0$ eventually as the parametric representations get arbitrarily close in the uniform topology. Therefore, as $\inf_{s \le \Bar{T}} \eta_s^n > 0$ eventually, by definition $\tau_0(\eta^n) = \Bar{T}$ eventually. Therefore, $\tnot$ is $M_1$-continuous at $\eta$. If \ref{enum: THIRD ITEM IN CONTINUITY PROOF} holds, when $\tnot = \Bar{T}$ and $\eta_{\Bar{T}} = 0$, then for any $\gamma >0$, with $\Bar{T} -\gamma$ being a continuity point, we must have $\inf_{s \le \Bar{T} - \gamma} \eta_s > 0$ because $\eta$ only jumps downwards. So for any approximating sequence $(\eta^n)_{n \ge 1} \subset \DR$ in the $M_1$-topology, eventually $\inf_{s \le \Bar{T} - \gamma}{\eta^n_s}>0$. Hence, eventually $\tau_0(\eta^n) > \Bar{T} - \gamma$. As $\gamma > 0$ can be made arbitrarily close to zero, by definition, $\lim_{n \to \infty} \tau_0(\eta^n) = \tnot = \Bar{T}$. Therefore, $\tau_0$ is $M_1$-continuous at $\eta$.
	\end{proof}
	
	With the result stating that the hitting time is an $M_1$-continuous map, weak convergence of the loss function follows immediately.
	
	\begin{lemma}[Continuity of conditional feedback]\label{lem:sec2: LEMMA STATING THAT THE CONDITION FEEDBACK IS A CONTINUOUS MAP HENCE WE HAVE WEAK CONVERGENCE OF FEEDBACK ESSENTIALLY}
		For $\operatorname{Law}(\mathbf{P}^*)$-almost every measure $\mu \in \mathcal{P}(\DR)$, the map $\mu \mapsto \mu(\tnot \le t)$ is continuous with respect to $\tmwk$ for all $t \in \mathbb{T}^{\mu} \cap [0,\,\Bar{T})$. $\mathbb{T}^{\mu}$ is the set of continuity points of $t \mapsto \mu(\tnot \le t)$.    
	\end{lemma}
	
	\begin{proof}
		By Corollary \ref{cor:sec2: LEMMA THAT THE HITTING TIME IS A M1 CONTINUOUS MAP}, for $\operatorname{Law}(\mathbf{P}^*)$-almost every measure $\mu$, $\tau_0$ is $M_1$-continuous for $\mu$-almost every $\eta$. We fix such a $\mu$ and consider a sequence $(\mu^n)_{n \ge 1}$ converging to $\mu$ in $\mathcal{P}(\DR)$. By Skorokhod's Representation Theorem, 
		\begin{equation*}
			\mu^n(\tnot \le t) = \E\left[\ind_{\tau_0(Z^n) \le t}\right] \qquad \text{and}\qquad \mu(\tnot \le t) = \E\left[\ind_{\tau_0(Z) \le t}\right],
		\end{equation*}
		where $\tau_0$ is continuous for almost all paths $Z$ and $Z^n \to Z$ almost surely in $(\DR,M_1)$. Now, for any $t \in \mathbb{T}^\mu\coloneqq \{t \in [-1,\,\Bar{T}]\,:\, \mu(\tnot = t) = 0\}$, by the Monotone Convergence Theorem,
		\begin{equation}\label{eq:sec2 FIRST EQUATION IN lem:sec2: LEMMA STATING THAT THE CONDITION FEEDBACK IS A CONTINUOUS MAP HENCE WE HAVE WEAK CONVERGENCE OF FEEDBACK ESSENTIALLY}
			\prob\left[\tau_0(Z) = t\right] = \mu(\tnot \le t) - \lim_{s \uparrow t} \mu(\tnot \le s) = 0.
		\end{equation}
		Therefore, employing the continuity of $\tau_0$ and \eqref{eq:sec2 FIRST EQUATION IN lem:sec2: LEMMA STATING THAT THE CONDITION FEEDBACK IS A CONTINUOUS MAP HENCE WE HAVE WEAK CONVERGENCE OF FEEDBACK ESSENTIALLY}, we have
		\begin{equation*}
			\E\left[\ind_{\tau_0(Z^n) \le t}\right] \to \E\left[\ind_{\tau_0(Z) \le t}\right],
		\end{equation*}
		by the Dominated Convergence Theorem. So, we conclude that
		\begin{equation*}
			\mu^n(\tnot \le t) \to \mu(\tnot \le t) \qquad \forall \; t \in \mathbb{T}^\mu.
		\end{equation*}
	\end{proof}
	
	Furthermore, we have weak convergence of the mollified loss to the singular loss.
	
	\begin{corollary}[Convergence of delayed loss]\label{cor:sec2: COROLLARY SHOWING THAT WE HAVE CONVERGENCE OF THE DELAYED / SMOOTHENED LOSS TO THE LIMITING LOSS}
		
		For $\operatorname{Law}(\mathbf{P}^*)$-almost every measure $\mu$, $\int_0^t \kappa^{\varepsilon_n}(t-s)\mu^n(\tnot \le s) \diff s$ converges to $\mu(\tnot \le t)$ for any $t \in \mathbb{T}^\mu$ and any sequence $(\mu^n)_{n \ge 1}$ that converges to $\mu$ in $(\mathcal{P}(\DR),\tmwk)$.
	\end{corollary}
	
	\begin{proof}
		By Lemma \ref{lem:sec2: LEMMA STATING THAT THE CONDITION FEEDBACK IS A CONTINUOUS MAP HENCE WE HAVE WEAK CONVERGENCE OF FEEDBACK ESSENTIALLY}, $\mu^n(\tau_0 \le t)$ converges to $\mu(\tau_0 \le t)$ for any $t \in \mathbb{T}^\mu$ when $\tau_0$ is an $M_1$-continuous map $\mu$-almost surely. Such measures $\mu$ have full $\operatorname{Law}(\mathbf{P}^*)$-measure by Corollary \ref{cor:sec2: LEMMA THAT THE HITTING TIME IS A M1 CONTINUOUS MAP}. Furthermore, for every such $\mu$,
		\begin{equation}\label{eq:sec2 FIRST EQUATION IN cor:sec2: COROLLARY SHOWING THAT WE HAVE CONVERGENCE OF THE DELAYED / SMOOTHENED LOSS TO THE LIMITING LOSS}
			\left( s \mapsto \mu^n(\tnot \le s)\ind_{[0,\,t]}(s)\right) \xrightarrow[]{n \to \infty} \left( s \mapsto \mu(\tnot \le s)\ind_{[0,\,t]}(s)\right)
		\end{equation}
		in the $M_1$-topology as functions from $[-1,\,t] \to \R$, as the functions are non-decreasing, \citep[Corollary 12.5.1]{whitt2002stochastic}. Now, for any $t  \in \mathbb{T}^\mu \cap (0,\Bar{T}]$,
		{\footnotesize\begin{align*}
				\nnnorm{\int_0^t \kappa^{\varepsilon_n}(s)\mu^n(\tnot \le t -s) \diff s - \mu(\tnot \le t)} &\le \nnnorm{\int_0^t \kappa^{\varepsilon_n}(s)(\mu^n(\tnot \le t -s) - \mu(\tnot \le t - s)) \diff s}\\
				&\quad + \nnnorm{\int_0^t \kappa^{\varepsilon_n}(s)(\mu(\tnot \le t -s) - \mu(\tnot \le t)) \diff s}\\
				&\quad + \nnnorm{\int_0^t \kappa^{\varepsilon_n}(s)\diff s - 1} \mu(\tnot \le t)\\
				&= I + II + III.
		\end{align*}}
		For any $\delta > 0$, we observe
		\begin{align*}
			I &\le \sup_{t - \delta \le s \le t}\nnnorm{\mu^n(\tnot \le s) - \mu(\tnot \le s)}\int_0^\delta \kappa^{\varepsilon_n}(s)\diff s + \int_\delta^\infty \kappa^{\varepsilon_n}(s)\diff s, \\
			II &\le \sup_{t - \delta \le s \le t}\nnnorm{\mu(\tnot \le s) - \mu(\tnot \le t)}\int_0^\delta \kappa^{\varepsilon_n}(s)\diff s + \int_\delta^\infty \kappa^{\varepsilon_n}(s)\diff s, \\
			III &\le  \int_t^\infty \kappa^{\varepsilon_n}(s)\diff s.
		\end{align*}
		As $M_1$-convergence implies local uniform convergence at continuity points, \citep[Theorem~12.5.1]{whitt2002stochastic}, and $t$ is a continuity point, by setting $\delta = \varepsilon_n^{1/2}$ and sending $n \to \infty$, we have that $I,\,II,$ and $III$ all go to zero.
	\end{proof}
	
	
	
	\subsection{Martingale arguments and convergence}
	
	As marginal tightness implies joint tightness, $\{(\tilde{\mathbf{P}}^\varepsilon,\,W^0,\,W)\}_{\varepsilon > 0}$ is tight in $(\mathcal{P}(\DR),\tmwk)\times(\mathcal{C}_\R,\norm{\cdot}_\infty)\times(\mathcal{C}_\R,\norm{\cdot}_\infty)$ where $(\mathcal{C}_\R,\norm{\cdot}_\infty)$ is shorthand notation for $(\mathcal{C}([0,\,T]\, ,\,\R),\norm{\cdot}_\infty)$, the space of continuous functions from $[0,\,T]$ to $\R$ endowed with the topology of uniform convergence. From now on, we fix a weak limit point $(\mathbf{P}^*,\,W^0,\,W)$ along a subsequence $(\varepsilon_n)_{n \ge 1}$ for which $\varepsilon_n$ converges to zero. Although we have fixed a limit point, all the following results will hold for any limit point.
	
	Let $\prob^n \coloneqq \operatorname{Law}(\tilde{\mathbf{P}}^{\varepsilon_n},\,W^0,\,W)$ and $\prob^*_{\mu,\omega^0,\omega}\coloneqq\operatorname{Law}(\mathbf{P}^*,\,W^0,\,W)$. So $\prob^n \implies \prob^*_{\mu,\omega^0,\omega}$. For completeness, we will define the probability space $(\Omega^*,\,\mathcal{F}^*,\,\prob^*_{\mu,\omega^0,\omega})$ where $\Omega^* = \mathcal{P}(\DR)\times\mathcal{C}_\R \times \mathcal{C}_\R$ and $\mathcal{F}^*$ is the corresponding Borel $\sigma$-algebra. Define the random variables $\mathbf{P}^*$, $W^0$, and $W$ on $\Omega^*$ such that for any tuple $(\mu,\omega^0,\omega)$,
	\begin{equation*}
		\mathbf{P}^*(\mu,\omega^0,\omega) = \mu, \qquad W^0(\mu,\omega^0,\omega) = \omega^0, \qquad \textnormal{and} \qquad W(\mu,\omega^0,\omega) = \omega.
	\end{equation*}
	Hence, the joint law of $(\mathbf{P}^*,W^0,W)$ is $\prob^*_{\mu,\omega^0,\omega}$ and $\mathcal{F}^* = \sigma(\mathbf{P}^*,\,W^0,\,W)$. We also define the limiting loss function $L^* \coloneqq \mathbf{P}^*(\tnot \le \cdot)$ and the co-countable set of times
	\begin{equation}
		\mathbb{T} \coloneqq \left\{t \in [-1,\,\Bar{T}]\,:\, \prob^*_{\mu,\omega^0,\omega}(\eta_t = \eta_{t-}) = 1,\,\prob^*_{\mu,\omega^0,\omega}(L_t^* = L_{t-}^*) = 1\right\}.
	\end{equation}
	
	Looking at the approximating system, we know $(\tilde{\mathbf{P}}^\varepsilon,W^0) \perp W$ for any $\varepsilon > 0$. Even though $\mathbf{P}^*$ is the weak limit of $W^0$-measurable random variables, weak convergence does not allow us to guarantee that limit points will be $W^0$-measurable. Regardless, we may exploit the independence from the approximating system to deduce the independence of $(\mathbf{P}^*,\, W^0)$ and $W$ in the limit. To fix the notation, let $\prob^*_{\mu,\omega^0}$ denote the projection of the measure $\prob^*_{\mu,\omega^0,\omega}$ onto its first two coordinates and $\prob^*_{\omega}$ denote the projection onto its final coordinate. Then we intuitively expect $\prob^*_{\mu,\omega^0,\omega} = \prob^*_{\mu,\omega^0} \otimes \prob^*_{\omega}$.
	
	\begin{lemma}[Independence from idiosyncratic noise]\label{lem:sec2: LEMMA STATING THAT THE LIMITING RANDOM MEASURE AND THE COMMON NOISE IS INDEPENDENT OF THE IDIOSYNCRATIC NOISE}
		Let $\mathbf{P}^*$, $W^0$, and $W$ be the canonical random variables on the probability space $(\Omega^*,\,\mathcal{F}^*,\,\prob^*_{\mu,\omega^0,\omega})$ defined above. Then, $(\mathbf{P}^*,\,W^0)$ is independent of $W$.
	\end{lemma}
	
	\begin{proof}
		As $(\mathcal{P}(\DR),\tmwk)$ and $(\mathcal{C}_\R,\norm{\cdot}_\infty)$ are Polish spaces, it is sufficient to show that for any $f \in \mathcal{C}_b(\mathcal{P}(\DR))$ and $g,h \in \mathcal{C}_b(\mathcal{C}_\R)$,
		\begin{equation}\label{eq:sec2: FIRST EQUATION IN lem:sec2: LEMMA STATING THAT THE LIMITING RANDOM MEASURE AND THE COMMON NOISE IS INDEPENDENT OF THE IDIOSYNCRATIC NOISE}
			\left\langle\prob^*_{\mu,\omega^0,\omega},\,f \otimes g \otimes h\right\rangle = \left\langle\prob^*_{\mu,\omega^0},\,f \otimes g\right\rangle\left\langle\prob^*_{\omega},\,h\right\rangle.
		\end{equation}
		The result follows from the Dominated Convergence Theorem and Dynkin's Lemma. Now, \eqref{eq:sec2: FIRST EQUATION IN lem:sec2: LEMMA STATING THAT THE LIMITING RANDOM MEASURE AND THE COMMON NOISE IS INDEPENDENT OF THE IDIOSYNCRATIC NOISE} follows readily by weak convergence and the Portmanteau Theorem as
		\begin{align*}
			\left\langle\prob^*_{\mu,\omega^0,\omega},\,f \otimes g \otimes h\right\rangle &= \lim_{n \to \infty} \left\langle\prob^n,\,f \otimes g \otimes h\right\rangle\\
			&= \lim_{n \to \infty} \left\langle\prob^n,\,f \otimes g\right\rangle \left\langle\prob^n,\,h\right\rangle\\
			&=\left\langle\prob^*_{\mu,\omega^0},\,f \otimes g\right\rangle\left\langle\prob^*_{\omega},\,h\right\rangle.
		\end{align*}
		The equality in the second line follows from the independence of $(\tilde{\mathbf{P}}^\varepsilon,\, W^0)$ from $W$.
	\end{proof}
	
	We shall use $\prob^*_{\mu,\omega^0,\omega}$ to construct a probability space where we can define a process that will solve \eqref{eq: SINGULAR MCKEAN VLASOV EQUATION IN THE SYSTEMIC RISK MODEL WITH COMMON NOISE GENERALISED} in the sense of Definition \ref{def: DEFINITION OF RELAXED SOLUTIONS GENERALISED}. Prior to that, we need to define the map employed in the martingale arguments that follow. This allows us to deduce that the process we construct will be of the correct form. For any $\varepsilon > 0$, we define the following functionals $\mathcal{M},\,\mathcal{M}^\varepsilon :\mathcal{P}(\DR) \times \DR \to \DR$: 
	\begin{align}
		\mathcal{M}^\varepsilon(\mu,\eta) &= \eta - \eta_{-1} - \int_0^{\cdot}b(s,\eta_s,\nu_s^\mu)\diff s + \int_0^\cdot \alpha(s) \diff \mathfrak{L}_s^{\mu,\varepsilon} \label{eq: MARTINGALE ARGUMENTS DEFINITION OF MATHCAL M SUPERSCRIPT M FOR THE BANKING MODELS GENERALISED},\\
		\mathcal{M}(\mu,\eta) &= \eta - \eta_{-1} - \int_0^{\cdot}b(s,\eta_s,\nu_s^\mu)\diff s + \int_{[0,\cdot]} \alpha(s) \diff {L}_s^{\mu}, \label{eq: MARTINGALE ARGUMENTS DEFINITION OF MATHCAL M FOR THE BANKING MODELS GENERALISED}
	\end{align}
	where for any $\mu \in \mathcal{P}(\DR)$, $$\nu_t^\mu \coloneqq \mu(\eta_t \in \cdot, \tnot > t),\qquad L_t^{\mu}\coloneqq \mu(\tnot \le t),\qquad \mathfrak{L}_t^{\mu,\varepsilon} = \int_0^t \kappa^{\varepsilon}(t - s) L_s^{\mu} \diff s,$$ and $b$ satisfies Assumption \ref{ass: MODIFIED AND SIMPLIED FROM SOJMARK SPDE PAPER ASSUMPTIONS II}. For any $s_0,\,t_0 \in \mathbb{T}^{\mu} \cap [0,\, T)$ with $s_0 < t_0$ and $\{s_i\}_{i = 1}^k \subset [0,s_0] \cap \mathbb{T}$, we define the function
	\begin{equation}\label{eq: MARTINGALE ARGUMENTS DEFINITION OF F FOR THE BANKING MODELS INDEXED AND SMOOTHED GENERALISED}
		F\;:\; D_\R \to \R,\quad \eta \mapsto (\eta_{t_0} - \eta_{s_0})\prod_{i = 1}^k f_i(\eta_{s_i}),
	\end{equation}
	for arbitrary $f_i \in \mathcal{C}_b(\R)$. We define the functionals:
	\begin{equation}\label{eq: FUNCTIONAL EMPLOYED IN THE FUNCTIOAL CONTINUITY I ARGUMENTS GENERALISED}
		\begin{cases}
			\Psi^\varepsilon(\mu) = \left\langle \mu,\eta \mapsto F\left(\mathcal{M}^\varepsilon(\mu,\eta)\right)\right\rangle,\\
			\Upsilon^\varepsilon(\mu) = \left\langle \mu,\eta \mapsto F\left((\mathcal{M}^\varepsilon(\mu,\eta))^2 - \int_0^{\cdot} \sigma(s,\eta_s)^2\diff s\right)\right\rangle,\\
			\Theta^\varepsilon(\mu,\omega) = \left\langle \mu,\eta \mapsto F\left(\mathcal{M}^\varepsilon(\mu,\eta)\times \omega - \int_0^\cdot\sigma(s,\eta_s)\sqrt{1 - \rho(s,\nu^\mu_s)^2} \diff s\right)\right\rangle,\\
			\Theta^{0,\varepsilon}(\mu,\omega^0) = \left\langle \mu,\eta \mapsto F\left(\mathcal{M}^\varepsilon(\mu,\eta)\times \omega^0 - \int_0^{\cdot} \sigma(s,\eta_s)\rho(s,\nu^\mu_s) \diff s\right)\right\rangle.
		\end{cases}
	\end{equation}
	\begin{minipage}{\columnwidth}Lastly, the corresponding functionals without the mollification, denoted by $\Psi(\mu),\,\Upsilon(\mu),$ $\Theta(\mu,\omega)$ and $\Theta^{0}(\mu,\omega^0)$, are defined in exactly the same way as $\Psi^\varepsilon(\mu),\,\Upsilon^\varepsilon(\mu),\,\Theta^\varepsilon(\mu,\omega)$, and $\Theta^{0,\varepsilon}(\mu,\omega^0)$ with $\mathcal{M}^\varepsilon$ replaced by $\mathcal{M}$.\end{minipage}
	\begin{remark}[Measurability of measure flows]
		In \eqref{eq: MARTINGALE ARGUMENTS DEFINITION OF MATHCAL M SUPERSCRIPT M FOR THE BANKING MODELS GENERALISED} and \eqref{eq: MARTINGALE ARGUMENTS DEFINITION OF MATHCAL M FOR THE BANKING MODELS GENERALISED}, we are taking a fixed measure, $\mu$, and computing the integral with respect to the measure flow $t \mapsto \nu^\mu_t$. The measurability of the functions $b$ and $\sigma$ is sufficient for this integral to be well-defined.
	\end{remark}
	Using Corollary \ref{cor:sec2: COROLLARY SHOWING THAT WE HAVE CONVERGENCE OF THE DELAYED / SMOOTHENED LOSS TO THE LIMITING LOSS}, we have the following proposition.
	
	\begin{proposition}[Functional Continuity I Generalised]\label{prop: FUNCTIONAL CONTINUITY I GENERALISED}
		For $\prob^*_{\mu,\omega^0,\omega}$-almost every measure $\mu$, we have that $\Psi^{\varepsilon_n}(\mu^n)$, $\Upsilon^{\varepsilon_n}(\mu^n)$, $\Theta^{\varepsilon_n}(\mu^n,\omega^n)$, and $\Theta^{0,\varepsilon_n}(\mu^n,\omega^{0,n})$ converge to $\Psi(\mu)$, $\Upsilon(\mu)$, $\Theta(\mu,\omega)$, and $\Theta^{0}(\mu,\omega^0)$ respectively, whenever $(\mu^n,\omega^{0,n},\omega^n) \to (\mu,\,\omega^0,\,\omega)$ in $(\mathcal{P}(\DR),\tmwk)\times(\mathcal{C}_\R,\norm{\cdot}_\infty)\times(\mathcal{C}_\R,\norm{\cdot}_\infty)$, along a sequence for which $\sup_{n \ge 1} \langle \mu^n,\sup_{s \le \bar{T}} \nnorm{\tilde{\eta}_s}^p \rangle$ is bounded for some $p > 2$ and $\varepsilon_n$ that converges to zero.
	\end{proposition}
	
	\begin{proof}
		By Lemma \ref{lem:sec2: LEMMA STATING THE STRONG CROSSING PROPERTY OF THE LIMITING RANDOM MEASURE P*} and the definition of $\mathbb{T}$, we have a set of $\mu$'s that have full $\prob^*_{\mu,\omega^0,\omega}$ measure, such that
		\begin{equation*}
			\mu\left(\inf_{s \in (\tnot,\,(\tnot + h) \wedge \Bar{T})} \left\{\eta_s - \eta_{\tnot}\right\} \ge 0,\, \tnot < \Bar{T}\right) = 0
		\end{equation*}
		for any $h > 0$, $\mu(\eta_{s_i} = \eta_{s_i-}) = 1$, $\mu(\eta_{t_0} = \eta_{t_0-}) = 1$, and $\mu(\tnot = t_0) = 0$. First, we shall show that $\Psi^{\varepsilon_n}(\mu^n)$ converges to $\Psi(\mu)$. By Corollary \ref{cor:sec2: COROLLARY SHOWING THAT WE HAVE CONVERGENCE OF THE DELAYED / SMOOTHENED LOSS TO THE LIMITING LOSS}, $\mathfrak{L}_t^{\mu^n,\varepsilon_n}$ converges to $L_t^\mu$. It is well-known that for any Borel measurable functions $f$ and $g$ of finite variation, we have for any $t > 0$
		\begin{equation*}
			f_tg_t = f(0)g(0) + \int_{(0,t]} f_{s-} \diff g_s + \int_{(0,t]} g_{s-} \diff f_s + \sum_{s \le t} \Delta f_s \Delta g_s.
		\end{equation*}
		This, together with the continuous differentiability of $\alpha$ implies
		\begin{equation*}
			\int_0^t \alpha(s) \diff \mathfrak{L}_s^{\mu,\varepsilon_n} = \alpha(t)\mathfrak{L}_t^{\mu,\varepsilon_n} - \int_0^t \mathfrak{L}_s^{\mu,\varepsilon_n} \alpha^\prime(s) \diff s \xrightarrow[]{} \alpha(t) {L}_t^{\mu} - \int_0^t {L}_s^{\mu} \alpha^\prime(s) \diff s = \int_{[0,t]} \alpha(s) \diff {L}_s^{\mu}.
		\end{equation*}
		
		As $\mu^n \implies \mu$, by Skorokhod's Representation Theorem, there exists a $(Z^n)_{n \ge 1}$ and $Z$ defined on a common probability space such that $\operatorname{Law}(Z^n) = \mu^n,\,\operatorname{Law}(Z) = \mu$ and $Z^n \to Z$ almost surely in $(\DR,M_1)$. Hence,
		\begin{equation*}
			\Psi^{\varepsilon_n}(\mu^n) = \E \left[F(\mathcal{M}^{\varepsilon_n}(\mu^n,Z^n))\right] \qquad \text{and} \qquad \Psi^{}(\mu) = \E \left[F(\mathcal{M}(\mu,Z))\right].
		\end{equation*}
		By Lemma \ref{app: lem: FUNCTIONAL CONTINUITY III},
		\begin{equation}\label{eq: FIRST EQUATION IN FUNCTIONAL CONTINUITY GENERALISED EQUATION THAT SHOWS CONVERGENCE OF THE INTEGRAL WITH B}
			\int_0^t b(s,Z_s^n,\nu_s^{\mu^n}) \diff s \to \int_0^t b(s,Z_s,\nu_s^{\mu}) \diff s
		\end{equation}
		almost surely for any $t \ge 0$. Since $\mathbb{T}^{\mu}$ contains all the almost sure continuity points of $Z$, by the properties of $M_1$-convergence and \eqref{eq: FIRST EQUATION IN FUNCTIONAL CONTINUITY GENERALISED EQUATION THAT SHOWS CONVERGENCE OF THE INTEGRAL WITH B}, we have
		\begin{equation*}
			Z_t^n - Z_{-1}^n - \int_0^t b(s,Z_s^n,\nu_s^{\mu^n}) \diff s \to Z_t - Z_{-1} - \int_0^t b(s,Z_s,\nu_s^{\mu}) \diff s
		\end{equation*}
		almost surely for any $t \in \{t_0,s_0,\ldots,s_k\}$. Hence, we deduce that $F(\mathcal{M}^{\varepsilon_n}(\mu^n,Z^n))$ converges almost surely to $F(\mathcal{M}(\mu,Z))$ in $\R$. Lastly, we observe
		\begin{equation}\label{eq: SECOND EQUATION IN FUNCTIONAL CONTINUITY GENERALISED LEMMA THAT SHOWS THE UNIFORM P TH MOMENT BOUNDS}
			\left\langle\mu^n,\,\nnorm{\mathcal{M}^{\varepsilon_n}_t(\mu^n,\cdot)}^p\right\rangle \le C \left(\left\langle\mu^n,\,\sup_{s \le \bar{T}}\nnnorm{\tilde{\eta_s}}^p\right\rangle + 1\right),
		\end{equation}
		for some constant $C$ that depends on $p$ and $b$ only but is uniform in $n$. Therefore, $F(\mathcal{M}^{\varepsilon_n}(\mu^n,Z^n))$ is uniformly $L^p$-bounded as
		\begin{equation*}
			\nnnorm{F(\mathcal{M}^{\varepsilon_n}(\mu^n,Z^n))}^p \le C\left(\nnnorm{\mathcal{M}^{\varepsilon_n}_{t_0}(\mu^n,Z^n)}^p + \nnnorm{\mathcal{M}^{\varepsilon_n}_{s_0}(\mu^n,Z^n)}^p\right),
		\end{equation*}
		and $\E[\nnorm{\mathcal{M}^{\varepsilon_n}_{t_0}(\mu^n, Z^n)}^p] = \langle \mu^n,\, \nnorm{\mathcal{M}^{\varepsilon_n}_{t_0}(\mu^n,\cdot)}^p\rangle$ where the latter is uniformly bounded in $n$ for some $p > 2$ by \eqref{eq: SECOND EQUATION IN FUNCTIONAL CONTINUITY GENERALISED LEMMA THAT SHOWS THE UNIFORM P TH MOMENT BOUNDS} and assumption. Therefore, by Vitali's Convergence Theorem, it follows that $\Psi^{\varepsilon_n}(\mu^n)$ converges to $\Psi(\mu)$.
		
		The convergence of $\Upsilon^{\varepsilon_n}(\mu^n)$, $\Theta^{\varepsilon_n}(\mu^n,\omega^n)$, and $\Theta^{0,\varepsilon_n}(\mu^n,\omega^{0,n})$, to $\Upsilon(\mu)$, $\Theta(\mu,\omega)$ and $\Theta^{0}(\mu,\omega^0)$ respectively follows from similar arguments. As $\sigma$ and $\rho$ are totally bounded by Assumption \ref{ass: MODIFIED AND SIMPLIED FROM SOJMARK SPDE PAPER ASSUMPTIONS II} \eqref{ass: MODIFIED AND SIMPLIED FROM SOJMARK SPDE PAPER ASSUMPTIONS II TWO} and \eqref{ass: MODIFIED AND SIMPLIED FROM SOJMARK SPDE PAPER ASSUMPTIONS II FOUR}, $\Upsilon^{\varepsilon_n}(\mu^n)$, $\Theta^{\varepsilon_n}(\mu^n,\omega^n)$, and $\Theta^{0,\varepsilon_n}(\mu^n,\omega^{0,n})$ are $L^p$-bounded uniformly in $n$. The continuity of $\sigma$ and the almost sure convergence of $Z^n$ to $Z$ in the $M_1$-topology ensures that
		\begin{equation*}
			\int_0^t \sigma(s,Z_s^n)^2 \diff s \to \int_0^t \sigma(s,Z_s)^2 \diff s
		\end{equation*}
		almost surely for all $t \ge 0$. Lastly, by the boundedness of $\sigma$ and $\rho$, a straightforward computation shows that
		{\small\begin{align*}
				\nnnorm{\sigma(t,x)\rho(t,\mu) - \sigma(t,x)\rho(t,\tilde{\mu})} &\le C \left(1 + \langle\mu,\,\nnnorm{\,\cdot\,}\rangle\right)d_1(\mu,\,\tilde{\mu})\le C \left(1 + \langle\mu,\,\nnnorm{\,\cdot\,}\rangle\right)d_0(\mu,\,\tilde{\mu}),\\
				\nnnorm{\sigma(t,x)\sqrt{1 - \rho(t,\mu)^2} - \sigma(t,x)\sqrt{1 - \rho(t,\tilde{\mu})^2}} &\le C \left(1 + \langle\mu,\,\nnnorm{\,\cdot\,}\rangle\right)d_1(\mu,\,\tilde{\mu})\le C \left(1 + \langle\mu,\,\nnnorm{\,\cdot\,}\rangle\right)d_0(\mu,\,\tilde{\mu}).
		\end{align*}}
		Therefore, the functions $(t,x,\mu) \mapsto \sigma(t,x)\rho(t,\mu)$ and $(t,x,\mu) \mapsto \sigma(t,x)\sqrt{1 - \rho(t,\mu)^2}$ satisfy Assumption \ref{ass: MODIFIED AND SIMPLIED FROM SOJMARK SPDE PAPER ASSUMPTIONS II} (\ref{ass: MODIFIED AND SIMPLIED FROM SOJMARK SPDE PAPER ASSUMPTIONS II ONE}). Now we may apply Lemma \ref{app: lem: FUNCTIONAL CONTINUITY III} and conclude
		\begin{align*}
			\int_0^t \sigma(s,Z_s^n)\sqrt{1 - \rho(s,\nu_s^{\mu^n})^2} \diff s &\xrightarrow[]{} \int_0^t \sigma(s,Z_s)\sqrt{1 - \rho(s,\nu_s^{\mu})^2} \diff s,\\
			\int_0^t \sigma(s,Z_s^n)\rho(s,\nu_s^{\mu^n}) \diff s &\xrightarrow[]{} \int_0^t \sigma(s,Z_s)\rho(s,\nu_s^{\mu}) \diff s,
		\end{align*}
		almost surely for all $t \ge 0$.
	\end{proof}
	
	The remainder of this section aims to show that the conditional law of $\{\tilde{X}^{\varepsilon_n}\}$ converges weakly to a random variable ${X}$ which will have the dynamics defined in \eqref{eq: SINGULAR RELAXED MCKEAN VLASOV EQUATION IN THE SYSTEMIC RISK MODEL WITH COMMON NOISE}. This is achieved in the following two steps:
	\begin{enumerate}
		\item First, we construct a probability space $(\Bar{\Omega},\,\Bar{\mathcal{F}},\,\Bar{\prob})$ such that $\mathcal{M}_\cdot$, $\mathcal{M}_\cdot^2 - \int_0^\cdot \sigma(s,\eta_s)^2 \diff s$, $\mathcal{M}_\cdot \times W - \int_0^\cdot \sigma(s,\eta_s)\sqrt{1 - \rho(s,\nu_s^\mu)^2} \diff s$, and $\mathcal{M}_\cdot \times W^0 - \int_0^\cdot \sigma(s,\eta_s)\rho(s,\nu_s^\mu) \diff s$, defined as in \eqref{eq: MARTINGALE ARGUMENTS DEFINITION OF MATHCAL M FOR THE BANKING MODELS GENERALISED}, are continuous martingales.
		\item Secondly, we construct a stochastic process ${X}$ on $(\Bar{\Omega},\,\Bar{\mathcal{F}},\,\Bar{\prob})$ such that the tuple $({X}, W, W^0, \mathbf{P}^*)$ is a solution to \eqref{eq: SINGULAR RELAXED MCKEAN VLASOV EQUATION IN THE SYSTEMIC RISK MODEL WITH COMMON NOISE GENERALISED} in the sense of Definition \ref{def: DEFINITION OF RELAXED SOLUTIONS GENERALISED}.
	\end{enumerate}
	
	To this end, we now proceed to show the above two claims. We begin by defining the probability space $(\Bar{\Omega},\,\Bar{\mathcal{F}},\,\Bar{\prob})$ where $\Bar{\Omega} = \Omega^* \times \DR = \mathcal{P}(\DR)\times \mathcal{C}_\R \times \mathcal{C}_\R \times \DR$ and $\Bar{\mathcal{F}}$ is the Borel $\sigma$-algebra $\mathcal{B}(\Bar{\Omega})$. We define the probability measure
	\begin{equation}\label{eq:sec2: EQUATION THAT DEFINES THE PROBABILITY MEASURE WE HAVE CONSTRUCTED}
		\Bar{\prob}(A) \coloneqq \int_{\mathcal{P}(\DR)\times \mathcal{C}_\R \times \mathcal{C}_\R} \mu\left(\left\{ \eta\,:\, (\mu, \omega^0,\omega,\eta) \in A\right\}\right) \diff \prob^*_{\mu,\omega^0,\omega}(\mu,\omega^0,\omega),
	\end{equation}
	for any $A \in \mathcal{B}(\Bar{\Omega})$. Observe by construction, for any $A \in \mathcal{B}(\Bar{\Omega})$,
	\begin{equation*}
		\Bar{\prob}(A) = \E^*\left[\left\langle\mathbf{P}^*,\, \ind_{A}(\mathbf{P}^*,\,W^0,\,W,\,\cdot)\right\rangle\right]. 
	\end{equation*}
	Furthermore, under $\Bar{\prob}$, $W^0$ and $W$ are still Brownian motions, and $(\mathbf{P}^*,\,W^0)$ is independent of $W$. This is immediate, as for any $A \in \mathcal{B}(\mathcal{P}(\DR)\times \mathcal{C}_\R)$ and $B \in \mathcal{B}(\mathcal{C}_\R)$,
	\begin{equation*}
		\Bar{\prob}\left[(\mathbf{P}^*,\, W^0) \in A,\, W \in B\right] =\Bar{\prob}\left(A\times B \times \DR\right) = \prob^*_{\mu,\omega^0,\omega}(A\times B). 
	\end{equation*}
	Given these ingredients, we may now show our first claim.
	
	\begin{proposition}\label{prop:sec2:PROPOSITION SHOWING THAT WE HAVE A CONTINUOUS LOCAL MARTINGALE WITH THE CORRECT COVARIATION}
		Let $\mathcal{M}$ be given as in \eqref{eq: MARTINGALE ARGUMENTS DEFINITION OF MATHCAL M FOR THE BANKING MODELS GENERALISED}. Then $\mathcal{M}_\cdot$, $\mathcal{M}_\cdot^2 - \int_0^\cdot \sigma(s,\pi_s(\cdot))^2 \diff s$,
		$\mathcal{M}_\cdot \times W - \int_0^\cdot \sigma(s,\pi_s(\cdot))\sqrt{1 - \rho(s,\nu_s^{\cdot})^2} \diff s$, and $\mathcal{M}_\cdot \times W^0 - \int_0^\cdot \sigma(s,\pi_s(\cdot))\rho(s,\nu_s^\cdot) \diff s$ are all continuous martingales on $(\Bar{\Omega},\,\Bar{\mathcal{F}},\, \Bar{\prob})$, where 
		\begin{equation*}
			\pi_s\,:\, \DR \to \R, \, \pi_s(\eta) = \eta_s \qquad \text{and} \qquad \nu_s^\cdot \,:\, \mathcal{P}(\DR) \to \MRR,\, \nu_s^\cdot(\mu) = \nu_s^\mu.  
		\end{equation*}
	\end{proposition}
	
	\begin{proof}
		If $\mathcal{M}$ is continuous, then the continuity of the other processes follows from the continuity of $\mathcal{M}$ and the continuity of integration. For simplicity, we shall use $\mathcal{N}$ to denote any one of $\mathcal{M}_\cdot,$ $ \mathcal{M}_\cdot \times W - \int_0^\cdot \sigma(s,\pi_s(\cdot))\sqrt{1 - \rho(s,\nu_s^{\cdot})^2} \diff s$, $\mathcal{M}_\cdot \times W^0 - \int_0^\cdot \sigma(s,\pi_s(\cdot))\rho(s,\nu_s^\cdot) \diff s$ or $\mathcal{M}_\cdot^2 - \int_0^\cdot \sigma(s,\pi_s(\cdot))^2 \diff s$. Hence to show that $\mathcal{N}$ is a martingale, it is sufficient by a Monotone Class argument that
		\begin{equation}\label{eq:sec2 FIRST EQUATION IN prop:sec2:PROPOSITION SHOWING THAT WE HAVE A CONTINUOUS LOCAL MARTINGALE WITH THE CORRECT COVARIATION}
			\Bar{\E}\left[F(\mathcal{N})\right] = 0,
		\end{equation}
		
		To begin, recall that $\prob^{\varepsilon_n} \implies \prob^*$, where $\prob^{\varepsilon_n} = \operatorname{Law}(\tilde{\mathbf{P}}^{\varepsilon_n},W^0,W)$ and $\prob^* = \operatorname{Law}(\mathbf{P}^*,W^0,W)$. By Skorokhod's Representation Theorem, we may find $\{(\mathbf{Q}^n,B^n,\tilde{B}^n)\}_{n \ge 1}$ and $(\mathbf{Q}^*,B^*,\tilde{B}^*)$ defined on a common probability space such that $\prob^{\varepsilon_n} = \operatorname{Law}(\mathbf{Q}^n,\,B^n,\,\tilde{B}^n)$, $\prob^* = \operatorname{Law}(\mathbf{Q}^*,\,B^*,\,\tilde{B}^*)$ and $(\mathbf{Q}^n,\,B^n,\,\tilde{B}^n) \xrightarrow[]{} (\mathbf{Q}^*,\,B^*,\,\tilde{B}^*)$ almost surely in $(\mathcal{P}(\DR),\tmwk)\times(\mathcal{C}_\R,\norm{\cdot}_\infty)\times(\mathcal{C}_\R,\norm{\cdot}_\infty)$.
		By the definition of $\tilde{\mathbf{P}}^{\varepsilon_n}$, for any $p > 1$, we have:
		\begin{equation}\label{eq:sec2 SECOND
				EQUATION IN prop:sec2:PROPOSITION SHOWING THAT WE HAVE A CONTINUOUS LOCAL MARTINGALE WITH THE CORRECT COVARIATION}
			\E\left[\left\langle \tilde{\mathbf{P}}^{\varepsilon_n},\, \sup_{s \le \bar{T}}\nnnorm{\eta_s}^p\right\rangle\right] = \E \left[\sup_{s \le \bar{T}}\nnnorm{\tilde{X}^{\varepsilon_n}}^p\right] \le C,
		\end{equation}
		where the constant $C$ is derived from Proposition \ref{prop: GRONWALL TYPE UPPERBOUND ON THE SUP OF THE MCKEAN VLASOV EQUATION WITH CONVOLUTION IN THE BANKING MODEL SIMPLIFIED AND INDEXED GENERALISED}. As the map $\eta \mapsto \sup_{s \le \bar{T}} \nnnorm{\eta_s}$ is $M_1$-continuous, the Portmanteau Theorem yields:
		\begin{equation}\label{eq:sec2 SECOND
				EQUATION WITH A STAR IN prop:sec2:PROPOSITION SHOWING THAT WE HAVE A CONTINUOUS LOCAL MARTINGALE WITH THE CORRECT COVARIATION}
			\E\left[\left\langle {\mathbf{P}}^{*},\, \sup_{s \le \bar{T}}\nnnorm{\eta_s}^p\right\rangle\right] \le C.
		\end{equation}
		
		Now, for any $K > 0$, we can write
		\begin{align*}
			\E [ \nnorm{\langle {\mathbf{Q}}^{n},\, \sup_{s \le \bar{T}}\nnorm{\eta_s}^p\rangle - \langle {\mathbf{Q}}^{*},\, \sup_{s \le \bar{T}}\nnorm{\eta_s}^p\rangle}] \le&
			\E [ \nnorm{\langle {\mathbf{Q}}^{n},\, \sup_{s \le \bar{T}}\nnorm{\eta_s}^p \wedge K\rangle - \langle {\mathbf{Q}}^{*},\, \sup_{s \le \bar{T}}\nnorm{\eta_s}^p\wedge K\rangle}] \\
			&+ \E [ \nnorm{\langle {\mathbf{Q}}^{n},\, (\sup_{s \le \bar{T}}\nnorm{\eta_s}^p - K)\ind_{\sup_{s \le \bar{T}}\nnorm{\eta_s}^p > K}\rangle}] \\
			&+ \E [ \nnorm{\langle {\mathbf{Q}}^{*},\, (\sup_{s \le \bar{T}}\nnorm{\eta_s}^p - K)\ind_{\sup_{s \le \bar{T}}\nnorm{\eta_s}^p > K}\rangle}].
		\end{align*}
		By equation \eqref{eq:sec2 SECOND
			EQUATION IN prop:sec2:PROPOSITION SHOWING THAT WE HAVE A CONTINUOUS LOCAL MARTINGALE WITH THE CORRECT COVARIATION}, and a combination of H\"older's and Markov's inequalities, the second term on the right-hand side above is $o(1)$ as $K \to \infty$ uniformly in $n$. Similarly, by \eqref{eq:sec2 SECOND
			EQUATION WITH A STAR IN prop:sec2:PROPOSITION SHOWING THAT WE HAVE A CONTINUOUS LOCAL MARTINGALE WITH THE CORRECT COVARIATION}, the third term on the right-hand side above is $o(1)$ as $K \to \infty$. For the first term, we recall that $\mathbf{Q}^n \to \mathbf{Q}^*$ almost surely in $(\mathcal{P}(\DR),\tmwk)$. Hence, as $\eta \mapsto \sup_{s \le \bar{T}}\nnorm{\eta_s}^p \wedge K$ is an $M_1$-continuous and bounded function, we have by the almost sure convergence and the Portmanteau Theorem that $\langle {\mathbf{Q}}^{n},\, \sup_{s \le \bar{T}}\nnorm{\eta_s}^p \wedge K\rangle$ converges almost surely towards $\langle {\mathbf{Q}}^{*},\, \sup_{s \le \bar{T}}\nnorm{\eta_s}^p\wedge K\rangle$. Therefore, by the Dominated Convergence Theorem, the first term on the right-hand side above is $o(1)$ as $n \to \infty$. Therefore, we can make the right-hand side arbitrarily small by first fixing $K$ sufficiently large, then choosing $n$ sufficiently large. Hence, we have
		\begin{equation*}
			\E [ \nnorm{\langle {\mathbf{Q}}^{n},\, \sup_{s \le \bar{T}}\nnorm{\eta_s}^p\rangle - \langle {\mathbf{Q}}^{*},\, \sup_{s \le \bar{T}}\nnorm{\eta_s}^p\rangle}] \to 0.
		\end{equation*} 
		Now, fixing a $p > 2$, as we have $L^1$-convergence, we can find a subsequence (also denoted by $n$ for simplicity) such that $\langle {\mathbf{Q}}^{n},\, \sup_{s \le \bar{T}}\nnorm{\eta_s}^p\rangle$ converges almost surely to $\langle {\mathbf{Q}}^{*},\, \sup_{s \le \bar{T}}\nnorm{\eta_s}^p\rangle$. By \eqref{eq:sec2 SECOND
			EQUATION WITH A STAR IN prop:sec2:PROPOSITION SHOWING THAT WE HAVE A CONTINUOUS LOCAL MARTINGALE WITH THE CORRECT COVARIATION},
		$\langle {\mathbf{Q}}^{*},\, \sup_{s \le \bar{T}}\nnorm{\eta_s}^p\rangle < +\infty$ almost surely. Therefore, restricting ourselves to this subsequence going forward, we have
		\begin{equation*}
			\sup_{n \ge 1} \left\langle \mathbf{Q}^{n},\, \sup_{s \le \bar{T}}\nnnorm{\eta_s}^p\right\rangle < \infty
		\end{equation*}
		almost surely.
		
		By definition of $\mathcal{M}$, for the same $p > 2$ fixed above
		
		\color{black}
		\begin{equation}\label{eq:sec2 SECOND
				EQUATION IN prop:sec2:PROPOSITION SHOWING THAT WE HAVE A CONTINUOUS LOCAL MARTINGALE WITH THE CORRECT COVARIATION TWO}
			\E\left[\left\langle \tilde{\mathbf{P}}^{\varepsilon_n},\, \nnnorm{\mathcal{M}_t(\tilde{\mathbf{P}}^{\varepsilon_n},\cdot)}^p\right\rangle\right] = \E \left[\nnnorm{\int_{0}^t\sigma(s)\sqrt{1 -\rho(s)^2} \diff W_s + \int_0^t \sigma(s)\rho(s \diff W_s^0}^p\right] \le Ct^p,
		\end{equation}
		where the constant $C$ depends on the constant from applying Burkholder--Davis--Gundy, $p$, and the bounds on $\sigma$ but is independent of $\varepsilon$. Hence, $\E[\langle \mathbf{Q}^{n},\, \nnnorm{\mathcal{M}_t(\mathbf{Q}^{n},\cdot)}^p \rangle] < \infty$ uniformly in $n$.
		
		Employing Proposition \ref{prop:sec2:PROPOSITION SHOWING THAT WE HAVE A CONTINUOUS LOCAL MARTINGALE WITH THE CORRECT COVARIATION} and Vitali's Convergence Theorem, we have
		\begin{equation*}
			\Bar{\E}\left[F(\mathcal{N})\right] = \E^*\left[\left\langle\mathbf{P}^*,\, F(\mathcal{N}(\mathbf{P}^*,\,W^0,\,W,\,\cdot))\right\rangle\right] = \lim_{n \to \infty} \E\left[\left\langle\mathbf{P}^{\varepsilon_n},\, F(\mathcal{N}^{\varepsilon_n}(\mathbf{P}^{\varepsilon_n},\,W^0,\,W,\,\cdot))\right\rangle\right],
		\end{equation*}
		where $\mathcal{N}^{\varepsilon_n}$ represents one of $\left(\mathcal{M}^{\varepsilon_n}_\cdot\right)^2 - \int_0^\cdot \sigma(s,\pi_s(\cdot))^2 \diff s$, $\mathcal{M}^{\varepsilon_n}_\cdot \times W^0 - \int_0^\cdot \sigma(s,\pi_s(\cdot))\rho(s,\nu_s^\cdot) \diff s$,  $\mathcal{M}^{\varepsilon_n}_\cdot \times W - \int_0^\cdot \sigma(s,\pi_s(\cdot))\sqrt{1 - \rho(s,\nu_s^\cdot)^2} \diff s$, or $\mathcal{M}^{\varepsilon_n}_\cdot$ depending on $\mathcal{N}$. Recall for arbitrary $f_i \in \mathcal{C}_b(\R)$, $F(\eta) = (\eta_{t_0} - \eta_{s_0})\prod_{i = 1}^kf_i(\eta_{s_i})$. So
		
		\begin{equation}\label{eq:sec2 THIRD
				EQUATION IN prop:sec2:PROPOSITION SHOWING THAT WE HAVE A CONTINUOUS LOCAL MARTINGALE WITH THE CORRECT COVARIATION}
			\E\left[\left\langle\mathbf{P}^{\varepsilon_n},\, F(\mathcal{N}^{\varepsilon_n}(\mathbf{P}^{\varepsilon_n},\,W^0,\,W,\,\cdot))\right\rangle\right] = \E\left[\left(\tilde{\mathcal{N}}^{\varepsilon_n}_{t_0} - \tilde{\mathcal{N}}^{\varepsilon_n}_{s_0}\right)\prod_{i = 1}^kf_i(\tilde{\mathcal{N}}^{\varepsilon_n}_{s_i})\right],
		\end{equation}
		where $\tilde{\mathcal{N}}^{\varepsilon_n}$ is one of $\tilde{\mathcal{Y}}^{\varepsilon_n},$ $(\tilde{\mathcal{Y}}^{\varepsilon_n})^2 - \int_0^\cdot \sigma(s,\tilde{X}^{\varepsilon_n})^2 \diff s,$ $ \tilde{\mathcal{Y}}^{\varepsilon_n} \times W - \int_0^\cdot \sigma(s,\tilde{X}^{\varepsilon_n})\sqrt{1 - \rho(s,\bm{\nu}_s^{\varepsilon_n})^2} \diff s$, or $\tilde{\mathcal{Y}}^{\varepsilon_n} \times W^0 - \int_0^\cdot \sigma(s,\tilde{X}^{\varepsilon_n})\rho(s,\bm{\nu}_s^{\varepsilon_n}) \diff s$ depending on the choice of $\mathcal{N}$. By the boundness assumption on $\sigma$, Assumption \ref{ass: MODIFIED AND SIMPLIED FROM SOJMARK SPDE PAPER ASSUMPTIONS II} \eqref{ass: MODIFIED AND SIMPLIED FROM SOJMARK SPDE PAPER ASSUMPTIONS II TWO}, $\tilde{\mathcal{N}}^{\varepsilon_n}$ is a martingale. As $s_1 \le \ldots \le s_k \le s_0 < t_0$, we have \eqref{eq:sec2 THIRD EQUATION IN prop:sec2:PROPOSITION SHOWING THAT WE HAVE A CONTINUOUS LOCAL MARTINGALE WITH THE CORRECT COVARIATION} equals zero by the tower property. Hence, we have shown \eqref{eq:sec2 FIRST EQUATION IN prop:sec2:PROPOSITION SHOWING THAT WE HAVE A CONTINUOUS LOCAL MARTINGALE WITH THE CORRECT COVARIATION}.
		
		Lastly, to see the continuity of $\mathcal{M}$, define the function
		\begin{equation*}
			\tilde{F}\;:\; D_\R \to \R,\quad \eta \mapsto  \nnnorm{\eta_t - \eta_s}^4,
		\end{equation*}
		for $s,\, t \in \mathbb{T} \cap [0,T)$. As before, define the functionals
		\begin{equation*}
			\tilde{\Psi}^\varepsilon(\mu) = \left\langle\mu,\,\tilde{F}(\mathcal{M^\varepsilon(\mu,\,\cdot)})\right \rangle, \qquad \tilde{\Psi}(\mu) = \left\langle\mu,\,\tilde{F}(\mathcal{M(\mu,\,\cdot)})\right \rangle.
		\end{equation*}
		Following the same proof as in Proposition \ref{prop: FUNCTIONAL CONTINUITY I GENERALISED}, we have that for $\prob^*_{\mu,\omega^0,\omega}$-almost every measure $\mu$, $\tilde{\Psi}^{\varepsilon_n}(\mu^n)$ converges to $\tilde{\Psi}(\mu)$ whenever $\mu^n \to \mu$ in $(\mathcal{P}(\DR),\tmwk)$ along a sequence for which \[\sup_{n \ge 1}\langle\mu^n,\,\sup_{s \le \bar{T}}\nnnorm{\eta_s}^p\rangle < \infty\] for some $p > 4$. We have finite moments for any $p > 1$, by  \eqref{eq:sec2 SECOND EQUATION IN prop:sec2:PROPOSITION SHOWING THAT WE HAVE A CONTINUOUS LOCAL MARTINGALE WITH THE CORRECT COVARIATION}. Therefore, by functional continuity and Vitali's convergence theorem, for any $s,t \in \mathbb{T}\cap[0,\, T)$ we have
		\begin{align*}
			\Bar{\E}\nnnorm{\mathcal{M}_t - \mathcal{M}_s}^4 &= \E^*\left[\left\langle\mathbf{P}^*,\, \nnnorm{\mathcal{M}_t(\mathbf{P}^*,\,\cdot) - \mathcal{M}_s(\mathbf{P}^*,\,\cdot)}^4 \right \rangle\right]\\
			&= \lim_{n \to \infty} \E\left[\left\langle\tilde{\mathbf{P}}^{\varepsilon_n},\, \nnnorm{\mathcal{M}^{\varepsilon_n}_t(\tilde{\mathbf{P}}^{\varepsilon_n},\,\cdot) - \mathcal{M}^{\varepsilon_n}_s(\tilde{\mathbf{P}}^{\varepsilon_n},\,\cdot)}^4 \right\rangle\right].
		\end{align*}
		By the definition of $\tilde{\mathbf{P}}^{\varepsilon_n}$ and the Burkholder--Davis--Gundy inequality,
		\begin{equation*}
			\E\left[\left\langle\tilde{\mathbf{P}}^{\varepsilon_n},\, \nnnorm{\mathcal{M}^{\varepsilon_n}_t(\mathbf{P}^{\varepsilon_n},\,\cdot) - \mathcal{M}^{\varepsilon_n}_s(\mathbf{P}^{\varepsilon_n},\,\cdot)}^4 \right\rangle\right] = \E \nnnorm{\tilde{\mathcal{Y}}^{\varepsilon_n}_t - \tilde{\mathcal{Y}}^{\varepsilon_n}_s}^4 \le C\nnnorm{t -s}^2,
		\end{equation*}
		where the constant $C$ is uniform in $n$. As $\mathbb{T}$ is dense, by Kolmogorov's Criterion, there exists a continuous process that is a modification of $\mathcal{M}$. Since $\mathcal{M}$ is right continuous and $\mathbb{T}$ is dense, these two processes are indistinguishable. Hence, $\mathcal{M}$ has a continuous version.
	\end{proof}
	
	Now, we have all the ingredients to prove Theorem \ref{thm:sec2: THE MAIN EXISTENCE AND CONVERGENCE THEOREM OF SOLUTIONS GENERALISED FOR SECTION 2}.
	
	\begin{proof}[Proof of Theorem \ref{thm:sec2: THE MAIN EXISTENCE AND CONVERGENCE THEOREM OF SOLUTIONS GENERALISED FOR SECTION 2}]
		By Proposition \ref{prop: TIGHTHNESS OF THE EMPIRICAL MEASURES IN THE SIMPLIFIED BANKING MODEL WITH MOLLIFICATIONS}, $\{(\tilde{\mathbf{P}}^\varepsilon,\, W^0,\,W)\}_{\varepsilon > 0}$ is tight. By Prokhorov's Theorem, tightness on Polish spaces is equivalent to being sequentially precompact. Therefore, for any subsequence $\{(\mathbf{P}^{\varepsilon_n},\,W^0,\, W)\}_{n \ge 1}$, where $(\varepsilon_n)_{n \ge 1}$ is a positive sequence that converges to zero, we have a convergent sub-subsequence. Fix a limit point $(\mathbf{P}^*,\,W^0,\,W)$ of this subsequence. Having fixed $(\mathbf{P}^*,\,W^0,\,W)$, we define the probability space $(\Bar{\Omega},\,\Bar{\mathcal{F}},\,\Bar{\prob})$ exactly as in \eqref{eq:sec2: EQUATION THAT DEFINES THE PROBABILITY MEASURE WE HAVE CONSTRUCTED}. Now, define the \cadlag process $X$ by
		\begin{equation*}
			X\,:\, \Bar{\Omega} \to \DR,\qquad (\mu,\,\omega^0,\,\omega,\,\eta) \mapsto \eta.
		\end{equation*}
		Then, by the construction of $\Bar{\prob}$ and the fact that $\prob^*_{\mu,\omega^0,\omega} = \prob^*_{\mu,\omega^0} \times\prob^*_{\omega}$ by Lemma \ref{lem:sec2: LEMMA STATING THAT THE LIMITING RANDOM MEASURE AND THE COMMON NOISE IS INDEPENDENT OF THE IDIOSYNCRATIC NOISE}, for all $A \in \mathcal{B}(\DR),\, S \in \mathcal{B}(\mathcal{P}(\DR)\times \mathcal{C}_\R)$, we have
		\begin{equation*}
			\Bar{\prob}\left[ X \in A,\, (\mathbf{P}^*,\,W^0) \in S\right] = \int_S \mu(A) \diff \prob^*_{\mu,\omega^0}. 
		\end{equation*}
		Here, $\prob^*_{\mu,\omega^0,\omega} = \operatorname{Law}(\mathbf{P}^*,\,W^0,\,W),\, \prob^*_{\mu,\omega^0} = \operatorname{Law}(\mathbf{P}^*,\,W^0)$, and $\prob^*_{\omega} = \operatorname{Law}(W)$. Consequently,
		\begin{equation*}
			\Bar{\prob}\left[\left.X \in A\right|\mathbf{P}^*,\,W^0\right] = \mathbf{P}^*(A) \qquad \forall \quad A \in \mathcal{B}(\DR).
		\end{equation*}
		By Proposition \ref{prop:sec2:PROPOSITION SHOWING THAT WE HAVE A CONTINUOUS LOCAL MARTINGALE WITH THE CORRECT COVARIATION},
		\begin{equation*}
			\mathcal{M}_t = X_t - X_{-1} - \int_0^t b(s,\,X_s,\,\bm{\nu}_s^*) \diff s - \int_{[0,t]}\alpha(s) \diff \mathbf{P}^*(\tau_0(X) \le s)
		\end{equation*}
		is a continuous local martingale with
		\begin{gather*}
			\left\langle\mathcal{M}\right \rangle_t = \int_0^t \sigma(s,X_s)^2 \diff s, \quad \left\langle\mathcal{M},\,W \right \rangle_t = \int_0^t \sigma(s,X_s)\sqrt{1 - \rho(s,\bm{\nu}^*_s)^2} \diff s,\\
			\left\langle\mathcal{M},\,W^0\right \rangle_t = \int_0^t \sigma(s,X_s)\rho(s,\bm{\nu}_s^*) \diff s, 
		\end{gather*}
		where $\bm{\nu}_s^* \coloneqq \mathbf{P}^*(X_s \in \cdot,\,\tau_0(X) > s)$. As $W^0$ and $W$ are standard independent Brownian motions, by L\'evy's Characterisation Theorem we have that
		\begin{equation*}
			\mathcal{M}_t = \int_0^t \sigma(s,X_s)\left(\sqrt{1 - \rho(s,\bm{\nu}_s^*)^2} \diff W_s + \rho(s,\bm{\nu}_s^*) \diff W_s^0\right).
		\end{equation*}
		Now, as $-1 \in \mathbb{T}$, the map $\eta \mapsto \eta_{-1}$ is $\mu$-almost surely continuous for $\prob_{\mu,\omega^0,\omega}$-almost every measure $\mu$. A simple application of the Portmanteau Theorem shows that $X_{-1} \sim \nu_{0-}$. By Lemma \ref{lem: LEMMA STATING THAT P* IS SUPPORTED ON PATHS THAT ARE CONSTANT BEFORE TIME 0}, we deduce that $X_{0-} \sim \nu_{0-}$. The independence between $(\mathbf{P}^*,\, W^0)$ and $W$ follows from Lemma \ref{lem:sec2: LEMMA STATING THAT THE LIMITING RANDOM MEASURE AND THE COMMON NOISE IS INDEPENDENT OF THE IDIOSYNCRATIC NOISE}. A similar argument as employed in Lemma \ref{lem:sec2: LEMMA STATING THAT THE LIMITING RANDOM MEASURE AND THE COMMON NOISE IS INDEPENDENT OF THE IDIOSYNCRATIC NOISE} shows that $X_{0-} \perp (\mathbf{P}^*,\, W^0,\, W)$. Lastly, by Lemma \ref{app:lem: APPENDIX LEMMA STATING THE UPPER BOUND ON THE JUMPS OF THE LIMITING PROCESS},
		\begin{equation*}
			\Delta L_t^* \le \inf\{x \ge 0\,:\, \bm{\nu}_{t-}^*([0,\alpha(t) x]) < x\} \qquad \textnormal{a.s.}
		\end{equation*}
		for all $t \ge 0$.
	\end{proof}
	
	
	\section{Stronger mode of convergence}\label{sec: SECTION ON THE CONVERGENCE FROM STEFAN}
	
	One of the limitations of the method in Section \ref{sec: SECTION ON GENERALISED CONVERGENCE TO DELAYED SOLUTIONS} is that it fails to yield a strong solution. That is, $\mathbf{P}$ is not equal to $  \operatorname{Law}(X \mid W^0)$. This is due to the mode of convergence employed being weak. To the best of our knowledge, there are no results in the existing literature relating to the existence of strong physical solutions in the setting with common noise. By Remark 2.5 from \citep[]{MEANFIELDTHROUGHHITTINGTIMESNADOTOCHIY}, the existence of strong solutions in the setting when $b$, $\sigma$ and $\rho$ are functions of time only is shown; however, it remains unclear whether these solutions are physical or not.
	
	The work introduced in \citep{cuchiero2020propagation} provided an alternative framework to construct solutions to systems with simplified dynamics and without common noise. This is done by a fixed-point approach. Notably, the constructed solutions possess a minimality property, meaning that any alternative solution to the system will dominate the solution obtained in \citep{cuchiero2020propagation}. By utilising the mean-field limit of a perturbed finite particle system approximation, the authors deduce that minimal solutions are in fact physical.
	
	This section extends this work to the case with common noise. Provided more restrictive assumptions on the coefficients than those introduced in Assumption \ref{ass: MODIFIED AND SIMPLIED FROM SOJMARK SPDE PAPER ASSUMPTIONS II}, we provide an algorithm to construct minimal $W^0$-measurable solutions to the singular and smoothed system. Furthermore, we get almost sure convergence of the smoothed minimal system towards the singular minimal system. As a consequence, we are able to conclude that the minimal $W^0$-measurable solution is, in fact, physical. This provides an alternative method to show minimal solutions are physical in the setting of \citep{cuchiero2020propagation}.
	
	We fix a filtered probability space $(\Omega,\,\mathcal{F},\,(\mathcal{F}_t)_{t \ge 0},\,\prob)$ that satisfies the usual conditions and supports two independent Brownian motions. This differs from Section \ref{sec: SECTION ON GENERALISED CONVERGENCE TO DELAYED SOLUTIONS} as the filtered probability space may change as we change $\varepsilon$. The mode of convergence was weak in Section \ref{sec: SECTION ON GENERALISED CONVERGENCE TO DELAYED SOLUTIONS}, therefore the smoothed systems needed not be defined on the same probability space. In this section, to be able to show a stronger mode of convergence, we require that our probability space and our Brownian motions are fixed because our methods employ a comparison principle approach.
	
	We would like the loss process to be adapted and measurable with respect to the common noise. Hence, for measurability reasons, we define $\mathcal{F}^{W^0}$ as the $\sigma$-algebra generated by $W^0$ and augmented to contain all $\prob$-null sets. We define $\mathcal{F}_t^{W^0}$ to be the right continuous filtration generated by $W^0$ that contains all the information up to time $t$ and augmented to contain all $\prob$-null sets. To be precise, that is
	\begin{equation*}
		\mathcal{F}^{W^0}_t = \left(\bigcap_{s > t} \sigma(\{W_u^0\,:\, u \le s\})\right) \vee \sigma(\{N \in \mathcal{F}\,:\, \prob(N) = 0\}).
	\end{equation*}
	As Brownian motion is continuous and has independent increments, $W^0$ is still a standard Brownian motion under the filtration $(\mathcal{F}^{W^0}_t)_{t\ge 0}$.
	
	We now propose our alternative method of solution construction. We will be considering the equation
	\begin{equation}\label{eq: SINGULAR MCKEAN VLASOV EQUATION IN THE SYSTEMIC RISK MODEL WITH COMMON NOISE WITH SIMPLIFIED COEFFICIENTS IN THE MINIMALITY SECITON}
		\begin{cases}
			\diff{}X_t = b(t)\diff t + \sigma(t)\sqrt{1 - \rho(t)^2}\diff{}W_t + \sigma(t)\rho(t) \diff{}W_t^0 - \alpha\diff L_t,\\
			\hspace{0.423cm}\tau = \inf\{t > 0 \; : \; X_t \le 0\},\\
			\hspace{0.339cm}\mathbf{P} = \prob\left[\left.X \in \,\cdot\,\right|\mathcal{F}^{W^0} \right],\quad\bm{\nu}_t \vcentcolon= \prob\left[ \left.X_t \in \cdot, \tau > t\right|\mathcal{F}^{W^0}_t\right],\\
			\hspace{0.25cm}L_t = \prob\left[\left.\tau \le t\right|\mathcal{F}^{W^0}_t \right],
		\end{cases}
	\end{equation}
	where $\alpha > 0$ is a constant. The coefficients $b$, $\sigma$, and $\rho$ are a measurable maps from $\R$ into $\R$ satisfying Assumption \ref{ass: MODIFIED AND SIMPLIED FROM SOJMARK SPDE PAPER ASSUMPTIONS II}. The system starts at time $0-$ with initial condition $X_{0-}$ which is almost surely positive. We require no further assumptions on the initial condition. 
	
	Given any solution $(X,L)$ to \eqref{eq: SINGULAR MCKEAN VLASOV EQUATION IN THE SYSTEMIC RISK MODEL WITH COMMON NOISE WITH SIMPLIFIED COEFFICIENTS IN THE MINIMALITY SECITON}, we may view the paths of $L$ living in the space
	\begin{equation*}
		M \vcentcolon= \left\{\ell:\Bar{\R} \to [0,1]\,:\, \ell_{0-} = 0,\,\ell_\infty = 1,\,\ell \text{ increasing and \cadlag}\right\}.
	\end{equation*}
	$M$ is the space of cumulative density functions on the extended real line. We endow $M$ with the topology induced by the L\'evy-metric
	\begin{equation*}
		d_{L}(\ell^1,\ell^2) \vcentcolon= \inf \left\{\varepsilon > 0\,:\, \ell^1_{t+\varepsilon} + \varepsilon \ge \ell_t^2\ge \ell_{t - \varepsilon}^1 - \varepsilon,\, \forall t \ge 0\right\}.
	\end{equation*}
	The L\'evy-metric metricizes weak convergence, hence we are endowing $M$ with the topology of weak convergence as we can associate each $\ell$ with a distribution $\mu_\ell \in \mathcal{P}([0,\infty
	])$. Hence as $M$ is endowed with the topology of weak convergence, then we observe that $\ell^n \xrightarrow[]{} \ell$ in $M$ if and only if $\ell_t^n \xrightarrow[]{} \ell_t$ for all $t \in \mathbbm{T} \vcentcolon= \{t \ge 0\,:\, \ell_{t-} = \ell_{t}\}$. With this topology, $M$ is a compact Polish space. As in the previous section, we will let $\DR$ denote the space of \cadlag functions from $[-1,\infty)$ to $\R$ and we endow $\DR$ with the $M_1$-topology. As elements in $M$ are increasing, then convergence in $M$ is equivalent to convergence in $\DR$.
	\subsection{Properties of $\Gamma$ and existence of strong solutions}
	For any $W^0$-measureable process $\ell$ with values in $M$, we define the operator $\Gamma$ as
	\begin{equation*}
		\begin{cases}
			\diff{}X_t^\ell = b(t)\diff t + \sigma(t)\sqrt{1 - \rho(t)^2}\diff{}W_t + \sigma(t)\rho(t) \diff{}W_t^0 - \alpha\diff \ell_t,\\
			\hspace{0.328cm}\tau^\ell = \inf\{t > 0 \; : \; X_t^\ell \le 0\},\\
			\hspace{0.2cm}\Gamma[\ell]_t = \mathbf{P}\left[\left.\tau^\ell \le t\right|\mathcal{F}^{W^0}_t \right].
		\end{cases}
	\end{equation*}
	By the independence of increments of Brownian motion, $\mathbf{P}[\tau^\ell \le t\mid \mathcal{F}^{W^0}_t ] = \mathbf{P}[\tau^\ell \le t\mid\mathcal{F}^{W^0}]$. Therefore, we may always choose a version of $\mathbf{P}[\tau^\ell \le t\mid \mathcal{F}^{W^0}_t ]$ such that $ \Gamma[\ell]$ is a $W^0$-measurable process with \cadlag paths. By artificially setting $\Gamma[\ell]_{\infty} = 1$, $\Gamma[\ell]$ has paths in $M$. First, we observe that $\Gamma$ is a continuous operator.
	
	\begin{proposition}[Continuity of $\Gamma$]\label{prop: CONTINUITY OF GAMMA PROPOSITIION}
		Let $\ell^n$ and $\ell$ be a sequence of adapted $W^0$-measurable processes that take values in $M$ such that $\ell^n \xrightarrow[]{} \ell$ almost surely in $M$. Then $\Gamma[\ell^n] \xrightarrow[]{} \Gamma[\ell]$ almost surely in $M$.
	\end{proposition}
	
	\begin{proof}
		For simplicity, we shall denote $X^{\ell^n}$ by $X^n$ and $X^\ell$ by $X$. As done previously, we may artificially extend $X^n$ and $X$ to be \cadlag processes on $[-1,\,\infty)$ by
		\begin{equation*}
			\Tilde{X}^n \vcentcolon= \begin{cases}
				X_{0-}   \hspace{0.6875cm} t \in [-1,0),\\
				X_t^n \qquad \,\, t \ge 0,
			\end{cases}
			\qquad%
			\Tilde{X} \vcentcolon= \begin{cases}
				X_{0-}   \hspace{0.6875cm} t \in [-1,0),\\
				X_t \qquad \,\, t \ge 0,
			\end{cases}
		\end{equation*}
		By the coupling, $\Tilde{X}^n + \alpha\ell^n = \Tilde{X} +\alpha \ell$ for every $n$. Hence trivially $\Tilde{X}^n + \alpha \ell^n \xrightarrow[]{} \Tilde{X} + \alpha \ell$ in $\DR$. As convergence in $M$ is equivalent to convergence in the $M_1$-topology, $\ell^n \xrightarrow[]{} \ell$ almost surely in $\DR$. Addition is a $M_1$-continuous map for functions that have jumps of common sign, \citep[Theorem~12.7.3]{whitt2002stochastic}, therefore $\Tilde{X}^n \xrightarrow[]{} \Tilde{X}$ almost surely in $\DR$. It is clear that $\Delta \Tilde{X}_t \le 0$ for any $t \ge 0$ and
		\begin{equation*}
			\prob\left[\underset{s \in (\tau_0(\Tilde{X}),\,\tau_0(\Tilde{X}) + h)}{\inf}\left\{\Tilde{X}_s - \Tilde{X}_{\tau_0(\Tilde{X})}\right\} \ge 0\right] = 0
		\end{equation*}
		for any $h > 0$ by Lemma \ref{app:lem: GENERAL STRONG CROSSSING PROPERTY RESULT}. Hence, $\tau_0$ is an $M_1$-continuous map at almost every path of $\Tilde{X}$ by Lemma \ref{app:lem: CONVERGENCE OF THE STOPPING TIME FOR PARTICLES THAT EXHIBIT THE CROSSING PROPERTY}. By the Conditional Dominated Convergence Theorem, for any $t \in \mathbbm{T}^{\Gamma[\ell]} \vcentcolon= \{t \ge 0\,:\, \prob[\Gamma[\ell]_t = \Gamma[\ell]_{t-}] = 1\}$ we have
		\begin{equation}\label{eq: FIRST EQUATION IN THE PROOF OF THE CONTINIUTY OF GAMMA}
			\Gamma[\ell^n]_t = \E\left[\left.\ind_{\{\tau_0(\Tilde{X}^n)\le t\}}\right|\mathcal{F}^{W^0}\right] \longrightarrow \E\left[\left.\ind_{\{\tau_0(\Tilde{X})\le t\}}\right|\mathcal{F}^{W^0}\right] = \Gamma[\ell]_t
		\end{equation}
		almost surely. Now, we fix a $\mathbbm{D} \subset \mathbbm{T}^{\Gamma[\ell]}$ such that $\mathbbm{D}$ is countable and dense in $\R_+$. By \eqref{eq: FIRST EQUATION IN THE PROOF OF THE CONTINIUTY OF GAMMA}, we may find a $\Omega_0 \in \mathcal{F}^{W^0}$ of full measure such that if we fix $\omega \in \Omega_0$ then \eqref{eq: FIRST EQUATION IN THE PROOF OF THE CONTINIUTY OF GAMMA} holds at $\omega$ for all $t \in \mathbbm{D}$. Now we fix a $\gamma > 0$, $\omega \in \Omega_0$ and $t > 0$ such that $\Gamma[\ell]_t(\omega) =  \Gamma[\ell]_{t-}(\omega)$. By continuity, there is a $s_1,\,s_2 \in \mathbbm{D}$ such that $s_1 < t < s_2$ and
		\begin{equation}
			\nnnorm{\Gamma[\ell]_t(\omega) - \Gamma[\ell]_{s_1}(\omega)} + \nnnorm{\Gamma[\ell]_t(\omega) - \Gamma[\ell]_{s_2}(\omega)} < \gamma
		\end{equation}
		Therefore for by monotonicity of $\Gamma[\ell^n]$ and the above we have
		\begin{align*}
			\nnnorm{\Gamma[\ell]_t(\omega) - \Gamma[\ell^n]_t(\omega) } \le \nnnorm{\Gamma[\ell]_{s_2}(\omega) - \Gamma[\ell]_t(\omega) } + \nnnorm{\Gamma[\ell]_{s_2}(\omega) - \Gamma[\ell^n]_{s_2} (\omega) } & \\ + \nnnorm{\Gamma[\ell^n]_{s_1}(\omega) - \Gamma[\ell^n]_{s_2}(\omega) } & = O(\gamma)
		\end{align*}
		for all $n$ large. In the case when $t = 0$ is a continuity point, we set $s_1 = -1$. Hence we have convergence of $\Gamma[\ell^n](\omega)$ to $\Gamma[\ell](\omega)$ at the continuity points of $\Gamma[\ell](\omega)$. Therefore, by definition, $\Gamma[\ell^n](\omega)$ converges to $\Gamma[\ell](\omega)$ in $M$. As $\Omega_0$ is a set of full measure, the result follows.
	\end{proof}
	
	We observe that the map $\Gamma$ also preserves almost sure monotonicity of the input processes.
	
	\begin{lemma}[Monotonicity of $\Gamma$]\label{lem: MONOTONICITY OF THE OPERATOR GAMMA}
		Let $\ell^1$ and $\ell^2$ be $W^0$-measurable processes with paths in $M$ such that $\ell^1 \le \ell^2$ almost surely, then $\Gamma[\ell^1]\le \Gamma[\ell^2]$ almost surely.
	\end{lemma}
	
	\begin{proof}
		As $\ell^1 \le \ell^2$ almost surely, then we have $X^{\ell^1} \ge X^{\ell^2}$ almost surely. It follows that $\tau^{\ell^1} \le \tau^{\ell^2}$ almost surely. By monotonicity of conditional expectation,
		\begin{equation*}
			\Gamma[\ell^1]_t = \prob\left[\left.\tau^{\ell^1} \le t \right|\mathcal{F}^{W^0}_t\right] \le \prob\left[\left.\tau^{\ell^2} \le t \right|\mathcal{F}^{W^0}_t\right] = \Gamma[\ell^2]_t 
		\end{equation*}
		almost surely for any $t \ge 0$. As $\Gamma[\ell^1]$ and $\Gamma[\ell^2]$ are \cadlag, we deduce $\Gamma[\ell^1]_t \le \Gamma[\ell^2]_t$ for any $t \ge 0$ almost surely.  
	\end{proof}
	
	With these two results in hand, we have all the ingredients to construct $W^0$-measurable solutions to \eqref{eq: SINGULAR MCKEAN VLASOV EQUATION IN THE SYSTEMIC RISK MODEL WITH COMMON NOISE WITH SIMPLIFIED COEFFICIENTS IN THE MINIMALITY SECITON}.
	
	\begin{proposition}\label{prop: EXISTENCE OF SOLUTIONS TO THE SINGULAR CASE}
		There exists a \cadlag $W^0$-measurable process $\ubar{L}$ which solves \eqref{eq: SINGULAR MCKEAN VLASOV EQUATION IN THE SYSTEMIC RISK MODEL WITH COMMON NOISE WITH SIMPLIFIED COEFFICIENTS IN THE MINIMALITY SECITON} and for any other \cadlag $W^0$- measurable process $L$ which satisfies \eqref{eq: SINGULAR MCKEAN VLASOV EQUATION IN THE SYSTEMIC RISK MODEL WITH COMMON NOISE WITH SIMPLIFIED COEFFICIENTS IN THE MINIMALITY SECITON}, we have $\ubar{L} \le L$ almost surely.
	\end{proposition}
	
	\begin{proof}
		For any $n \ge 1$, we define inductively
		\begin{equation*}
			\begin{cases}
				\hspace{0.825cm}\diff{}X_t^n = b(t)\diff t + \sigma(t)\sqrt{1 - \rho(t)^2}\diff{}W_t + \sigma(t)\rho(t) \diff{}W_t^0 - \alpha\diff \Gamma^{n-1}[\ind_{\{\infty\}}],\\
				\hspace{1.145cm}\tau^n = \inf\{t > 0 \; : \; X_t^n \le 0\},\\
				\Gamma^n[\ind_{\{\infty\}}]_t = \mathbf{P}\left[\left.\tau^n \le t\right|\mathcal{F}^{W^0}_t \right],
			\end{cases}
		\end{equation*}
		with $\Gamma^0[\ind_{\{\infty\}}] = \ind_{\{\infty\}}$ and $\Gamma^n[\ind_{\{\infty\}}]$ is the application of $\Gamma$ $n$-times to the function $\ind_{\{\infty\}} \in M$. By Lemma \ref{lem: MONOTONICITY OF THE OPERATOR GAMMA}, $\Gamma^{n+1}[\ind_{\{\infty\}}] \ge \Gamma^n[\ind_{\{\infty\}}]$ almost surely for any $n \in \N$. As these processes are c\`adl\`ag, we deduce $\Gamma^{n+1}[\ind_{\{\infty\}}] \ge \Gamma^n[\ind_{\{\infty\}}]$ for any $n \in \N$ almost surely. Let $\Omega_0 \in \mathcal{F}^{W^0}$ denote the set of full measure where the monotonicity holds for every $n$ and we fix a $\mathbbm{D} \subset \R_+$ that is countable and dense. As $\Gamma^n[\ind_{\{\infty\}}]$ is increasing and bounded above, let
		\begin{equation*}
			\ell_t \vcentcolon= \lim_{n \xrightarrow[]{} \infty} \Gamma^n[\ind_{\{\infty\}}]_t\ind_{\Omega_0} \qquad \forall \; t \in \mathbbm{D}.
		\end{equation*}
		It is clear for any $t \in \mathbbm{D}$, $\ell_t$ is $\mathcal{F}^{W^0}_t$-measurable. Therefore we define
		\begin{equation*}
			\ubar{L}_t \vcentcolon= \underset{s \downarrow t,\,s \in \mathbbm{D}}{\lim} \ell_s \qquad \forall \; t \ge 0.
		\end{equation*}
		By construction, $\ubar{L}_t$ is a \cadlag $W^0$-measurable process with paths in $M$. A similar proof as that used in the end of Proposition \ref{prop: CONTINUITY OF GAMMA PROPOSITIION}, shows that $\Gamma^n[\ind_{\{\infty\}}] \xrightarrow[]{} \ubar{L}$ almost surely in $M$. Hence by Proposition \ref{prop: CONTINUITY OF GAMMA PROPOSITIION}, $\Gamma^{n + 1}[\ind_{\{\infty\}}] \xrightarrow[]{} \Gamma[{\ubar{L}}]$. As $\Gamma[{\ubar{L}}]$ and $\ubar{L}$ are \cadlag $W^0$-measurable processes that are limits of $\Gamma^n[\ind_{\{\infty\}}]$, we may conclude that $\Gamma[{\ubar{L}}] = \ubar{L}$ almost surely. Lastly, if $L$ is any \cadlag $W^0$-measurable process that solves \eqref{eq: SINGULAR MCKEAN VLASOV EQUATION IN THE SYSTEMIC RISK MODEL WITH COMMON NOISE WITH SIMPLIFIED COEFFICIENTS IN THE MINIMALITY SECITON}, then by Lemma \ref{lem: MONOTONICITY OF THE OPERATOR GAMMA} we have $\Gamma^n[\ind_{\{\infty\}}] \le L$ for all $n \in \N$ almost surely. Taking limit, we deduce $\ubar{L} \le L$ almost surely.  
	\end{proof}
	
	We now turn our attention to the smoothed version of \eqref{eq: SINGULAR MCKEAN VLASOV EQUATION IN THE SYSTEMIC RISK MODEL WITH COMMON NOISE WITH SIMPLIFIED COEFFICIENTS IN THE MINIMALITY SECITON}. We will work on the same filtered probability space $(\Omega,\,\mathcal{F},\,(\mathcal{F}_t)_{t \ge 0},\,\prob)$ as in \eqref{eq: SINGULAR MCKEAN VLASOV EQUATION IN THE SYSTEMIC RISK MODEL WITH COMMON NOISE WITH SIMPLIFIED COEFFICIENTS IN THE MINIMALITY SECITON} that satisfies the usual conditions and supports two independent Brownian motions. For an $\varepsilon > 0$, we consider the McKean--Vlasov problem
	\begin{equation}\label{eq: SMOOTHENED MCKEAN VLASOV EQUATION IN THE SYSTEMIC RISK MODEL WITH COMMON NOISE WITH SIMPLIFIED COEFFICIENTS IN THE MINIMALITY SECITON}
		\begin{cases}
			\diff{}X_t^\varepsilon = b(t)\diff t + \sigma(t)\sqrt{1 - \rho(t)^2}\diff{}W_t + \sigma(t)\rho(t) \diff{}W_t^0 - \alpha\diff \mathfrak{L}^\varepsilon_t,\\
			\hspace{0.328cm}\tau^\varepsilon = \inf\{t > 0 \; : \; X_t^\varepsilon \le 0\},\\
			\hspace{0.247cm}\mathbf{P}^\varepsilon = \prob\left[\left.X^\varepsilon \in \,\cdot\,\right|\mathcal{F}^{W^0} \right],\quad\bm{\nu}_t^\varepsilon \vcentcolon= \prob\left[ \left.X_t^\varepsilon \in \cdot, \tau^\varepsilon > t\right|\mathcal{F}^{W^0}_t\right],\\
			\hspace{0.29cm}L_t^\varepsilon = \prob^\varepsilon\left[\left.\tau^\varepsilon \le t \right|\mathcal{F}^{W^0}_t\right],\quad\mathfrak{L}^\varepsilon_t = \int_0^t \kappa^{\varepsilon}(t-s)L_s^\varepsilon\diff s,
		\end{cases}
	\end{equation}
	where $\alpha > 0$ is a constant. The coefficients $b$, $\sigma$, $\rho$ and $\kappa$ are a measurable maps from $\R$ into $\R$ satisfying Assumption \ref{ass: MODIFIED AND SIMPLIED FROM SOJMARK SPDE PAPER ASSUMPTIONS II}. The system starts at time $0-$ with the same initial condition, $X_{0-}$, as in \eqref{eq: SINGULAR MCKEAN VLASOV EQUATION IN THE SYSTEMIC RISK MODEL WITH COMMON NOISE WITH SIMPLIFIED COEFFICIENTS IN THE MINIMALITY SECITON}. As the assumptions on $X_{0-}$ is more general than those imposed in Section \ref{sec: SECTION ON GENERALISED CONVERGENCE TO DELAYED SOLUTIONS}, we may not apply Theorem \ref{thm: EXISTENCE AND UNIQUESS OF SOLUTIONS THEOREM FROM SPDE} to guarantee existence of solutions to \eqref{eq: SINGULAR MCKEAN VLASOV EQUATION IN THE SYSTEMIC RISK MODEL WITH COMMON NOISE WITH SIMPLIFIED COEFFICIENTS IN THE MINIMALITY SECITON}. So, we propose an alternative proof to show existence of solutions. The proof follows in the same faith as Proposition \ref{prop: EXISTENCE OF SOLUTIONS TO THE SINGULAR CASE}. We define the operator $$\Gamma_\varepsilon[\ell] \vcentcolon= \Gamma[(\kappa^\varepsilon \ast \ell)], \quad \textnormal{where} \quad (\kappa^\varepsilon \ast \ell)\vcentcolon= \int_0^\cdot\kappa^\varepsilon(\cdot - s)\ell_s \diff s.$$ Therefore, solutions to \eqref{eq: SMOOTHENED MCKEAN VLASOV EQUATION IN THE SYSTEMIC RISK MODEL WITH COMMON NOISE WITH SIMPLIFIED COEFFICIENTS IN THE MINIMALITY SECITON} are equivalent to finding almost sure fixed points of $\Gamma_\varepsilon$. A simple consequence of Proposition \ref{prop: CONTINUITY OF GAMMA PROPOSITIION}, is that $\Gamma_\varepsilon$ is also continuous.
	
	\begin{corollary}[Continuity of $\Gamma_\varepsilon$]\label{prop: CONTINUITY OF GAMMA VAREPSILON PROPOSITIION}
		Let $\ell^n$ and $\ell$ be a sequence of adapted $W^0$-measurable processes that take values in $M$ such that $\ell^n \xrightarrow[]{} \ell$ almost surely in $M$. Then $\Gamma_\varepsilon[\ell^n] \xrightarrow[]{} \Gamma_\varepsilon[\ell]$ almost surely in $M$.
	\end{corollary}
	
	\begin{proof}
		By Proposition \ref{prop: CONTINUITY OF GAMMA PROPOSITIION}, it is sufficient to show that the map $\Tilde{\ell} \mapsto \kappa^\varepsilon \ast \Tilde{\ell}$ is continuous on $M$. It is clear that if we implicitly define the value of $\kappa^\varepsilon \ast \Tilde{\ell}$ to be $1$ at $\infty$, then it is an element of $M$. Let $\Tilde{\ell}^n$ and $\Tilde{\ell}$ be deterministic functions in $M$ such that $\Tilde{\ell}^n \xrightarrow[]{} \Tilde{\ell}$ in $M$. That is, we have pointwise convergence on the continuity points of $\Tilde{\ell}$. As $\kappa \in \mathcal{W}^{1,1}(\R_+)$, it has a continuous representative. So without loss of generality, we take $\kappa$ to be this representative. Hence $\kappa$ is bounded on compacts, so an easy application of the Dominated Convergence Theorem gives
		\begin{equation*}
			\lim_{n \xrightarrow[]{} \infty} (\kappa^\varepsilon \ast \Tilde{\ell}^n)_t = \lim_{n \xrightarrow[]{} \infty} \int_0^t \kappa^\varepsilon(t-s)\Tilde{\ell}^n_s \diff s = \int_0^t \kappa^\varepsilon(t-s)\Tilde{\ell}_s \diff s = (\kappa^\varepsilon \ast \Tilde{\ell})_t
		\end{equation*}
	\end{proof}
	
	As convolution with non-negative functions preserves monotonicity, we further deduce that $\Gamma_\varepsilon$ is also monotonic by Lemma \ref{lem: MONOTONICITY OF THE OPERATOR GAMMA}.
	
	\begin{corollary}\label{cor: MONOTONICITY OF THE OPERATOR GAMMA VAREPSILON}
		Let $\ell^1$ and $\ell^2$ be $W^0$-measurable processes with paths in $M$ such that $\ell^1 \le \ell^2$ almost surely, then $\Gamma_\varepsilon[\ell^1]\le \Gamma_\varepsilon[\ell^2]$ almost surely.
	\end{corollary}
	
	With monotonicity and continuity of the operator $\Gamma_\varepsilon$ in hand, we have all the necessary results to deduce the existence of solutions to \eqref{eq: SMOOTHENED MCKEAN VLASOV EQUATION IN THE SYSTEMIC RISK MODEL WITH COMMON NOISE WITH SIMPLIFIED COEFFICIENTS IN THE MINIMALITY SECITON}.
	
	\begin{proposition}\label{prop: EXISTENCE OF SOLUTIONS TO THE SMOOTHENED CASE}
		There exists a \cadlag $W^0$-measurable process $\ubar{L}^\varepsilon$ which solves \eqref{eq: SMOOTHENED MCKEAN VLASOV EQUATION IN THE SYSTEMIC RISK MODEL WITH COMMON NOISE WITH SIMPLIFIED COEFFICIENTS IN THE MINIMALITY SECITON} and for any other \cadlag $W^0$-measurable process $L^\varepsilon$ which satisfies \eqref{eq: SMOOTHENED MCKEAN VLASOV EQUATION IN THE SYSTEMIC RISK MODEL WITH COMMON NOISE WITH SIMPLIFIED COEFFICIENTS IN THE MINIMALITY SECITON}, we have $\ubar{L}^\varepsilon \le L^\varepsilon$ almost surely.
	\end{proposition}
	
	\begin{proof}
		By employing Corollary \ref{prop: CONTINUITY OF GAMMA VAREPSILON PROPOSITIION} and Corollary \ref{cor: MONOTONICITY OF THE OPERATOR GAMMA VAREPSILON}, this proof is verbatim to that of Proposition \ref{prop: EXISTENCE OF SOLUTIONS TO THE SINGULAR CASE}.
	\end{proof}
	
	The purpose of $\kappa^\varepsilon$ in \eqref{eq: SMOOTHENED MCKEAN VLASOV EQUATION IN THE SYSTEMIC RISK MODEL WITH COMMON NOISE WITH SIMPLIFIED COEFFICIENTS IN THE MINIMALITY SECITON} is two-fold. Firstly, it smoothens the effect of the feedback component on the system, hence preventing the system from jumping and making it continuous. Secondly, it delays the effect of $L_t^\varepsilon$ of the system. Intuitively, one would expect that the system with instantaneous feedback, i.e. \eqref{eq: SINGULAR MCKEAN VLASOV EQUATION IN THE SYSTEMIC RISK MODEL WITH COMMON NOISE WITH SIMPLIFIED COEFFICIENTS IN THE MINIMALITY SECITON}, will be dominated by that with delayed feedback. Furthermore, intuitively as we decrease $\varepsilon$, then the system with the smaller value of $\varepsilon$ should be dominated by one with a larger value. This is because as $\varepsilon$ decreases, the rate at which the feedback is felt by the system increases.
	
	\begin{lemma}\label{lem: LOSS OF THE SYSTEM WITH INSTANTANEOUS FEEDBACK DOMINATES THAT WITH MOLLIFICATION AND MINIMAL SOLUTION L-EPS IS DECREASING IN EPSILON}
		For any $\varepsilon,\, \Tilde{\varepsilon} > 0$ such that $\Tilde{\varepsilon} < \varepsilon$, it holds that
		\begin{equation*}
			\ubar{L}^\varepsilon \le \ubar{L}. \qquad \textnormal{and} \qquad \ubar{L}^{\Tilde{\varepsilon}} \ge \ubar{L}^\varepsilon
		\end{equation*}
		almost surely.
	\end{lemma}
	
	\begin{proof}
		For any deterministic functions $\Tilde{\ell}^1,\,\Tilde{\ell}^2 \in M$ such that $\Tilde{\ell}^1 \le \Tilde{\ell}^2$, a straightforward computation shows that $(\kappa^\varepsilon \ast \Tilde{\ell}^1) \le \Tilde{\ell}^2$. Furthermore, for any $t > 0$, we have:
		\begin{align*}
			\int_0^t \kappa^{\varepsilon}(s)\tilde{\ell}_{t - s}^1\diff s 
			&= \int_0^t \varepsilon^{-1}\kappa(s \varepsilon^{-1}) \tilde{\ell}_{t - s}^1\diff s\\
			&= \int_0^{\frac{t\tilde{\varepsilon}}{\varepsilon}} \tilde{\varepsilon}\kappa(\tilde{s} \tilde{\varepsilon}^{-1}) \tilde{\ell}_{t - \frac{\tilde{s}\tilde{\varepsilon}}{\varepsilon}}^1\diff \tilde{s} \\
			&\le \int_0^t \kappa^{\tilde{\varepsilon}}(\tilde{s})\tilde{\ell}_{t - \tilde{s}}^1\diff \tilde{s}\\
			&\le \int_0^t \kappa^{\tilde{\varepsilon}}(\tilde{s})\tilde{\ell}_{t - \tilde{s}}^2\diff \tilde{s}.
		\end{align*}
		The second equality follows from employing the substitution $\tilde{s}\tilde{\varepsilon}^{-1} = {s}{\varepsilon}^{-1}$, and the third inequality follows from the fact that $\tilde{\varepsilon}\varepsilon^{-1} <1$, hence $\frac{t\tilde{\varepsilon}}{\varepsilon} < t$ and $\tilde{\ell}_{t - \frac{\tilde{s}\tilde{\varepsilon}}{\varepsilon}}^1 \le \tilde{\ell}_{t - \tilde{s}}^1$. The claim now follows from the monotonicity from Proposition \ref{prop: CONTINUITY OF GAMMA PROPOSITIION} and Lemma \ref{lem: MONOTONICITY OF THE OPERATOR GAMMA}.        
	\end{proof}
	\subsection{Convergence of minimal solutions}
	From now on, we will fix a sequence of positive real numbers $(\varepsilon_n)_{n \ge 1}$ that converge to zero. As we have established that $\ubar{L}^\varepsilon$ is a decreasing process in $\varepsilon$ by Lemma \ref{lem: LOSS OF THE SYSTEM WITH INSTANTANEOUS FEEDBACK DOMINATES THAT WITH MOLLIFICATION AND MINIMAL SOLUTION L-EPS IS DECREASING IN EPSILON}, we shall exploit this structure to construct a solution to \eqref{eq: SINGULAR MCKEAN VLASOV EQUATION IN THE SYSTEMIC RISK MODEL WITH COMMON NOISE WITH SIMPLIFIED COEFFICIENTS IN THE MINIMALITY SECITON}. This will be a $W^0$-measurable solution that will be dominated by every other $W^0$-measurable solution. Therefore, we may conclude that this solution must coincide with $\ubar{L}$ on a set of full measure. 
	
	\begin{theorem}[Almost sure convergence]\label{thm: Propogation of Minimality}
		Let $(\varepsilon_n)_{n \ge 1}$ be a sequence of positive real numbers that converges to zero. Let $(\ubar{X}^\varepsilon,\ubar{L}^\varepsilon)$ denote the $W^0$-measurable solution to \eqref{eq: SMOOTHENED MCKEAN VLASOV EQUATION IN THE SYSTEMIC RISK MODEL WITH COMMON NOISE WITH SIMPLIFIED COEFFICIENTS IN THE MINIMALITY SECITON} constructed in Proposition \ref{prop: EXISTENCE OF SOLUTIONS TO THE SMOOTHENED CASE}, and $(\ubar{X},\ubar{L})$ denote the $W^0$-measurable solution to \eqref{eq: SINGULAR MCKEAN VLASOV EQUATION IN THE SYSTEMIC RISK MODEL WITH COMMON NOISE WITH SIMPLIFIED COEFFICIENTS IN THE MINIMALITY SECITON} constructed in Proposition \ref{prop: EXISTENCE OF SOLUTIONS TO THE SINGULAR CASE}. Then by considering the extended system
		\begin{equation*}
			\Tilde{\ubar{X}}^{\varepsilon_n} \vcentcolon= \begin{cases}
				X_{0-}   \hspace{0.6875cm} t \in [-1,0),\\
				\ubar{X}_t^{\varepsilon_n} \qquad \,\, t \ge 0,
			\end{cases}
			\qquad%
			\Tilde{\ubar{X}} \vcentcolon= \begin{cases}
				X_{0-}   \hspace{0.6875cm} t \in [-1,0),\\
				\ubar{X}_t \qquad \,\, t \ge 0,
			\end{cases}
		\end{equation*}
		we have $\operatorname{Law}(\Tilde{\ubar{X}}^{\varepsilon_n}\mid \mathcal{F}^{W^0}) \xrightarrow[]{}\operatorname{Law}(\Tilde{\ubar{X}}\mid \mathcal{F}^{W^0})$ almost surely in $(\mathcal{P}(D_\R),\mathfrak{T}_{M_1}^{\text{wk}})$. Furthermore, $\ubar{L}^{\varepsilon_n}$ converges to $\ubar{L}$ almost surely in $M$ and $\ubar{L}$ satisfies the physical jump condition.
	\end{theorem}
	
	\begin{proof}
		As $(\varepsilon_n)_{n \ge 1}$ is a bounded sequence of reals converging to zero, we may find a decreasing subsequence $(\varepsilon_{n_{j}})_{j \ge 1}$ which converges to zero. We fix a $\mathbbm{D}\subset \R_+$ that is countable and dense in $\R_+$ and by Lemma \ref{lem: LOSS OF THE SYSTEM WITH INSTANTANEOUS FEEDBACK DOMINATES THAT WITH MOLLIFICATION AND MINIMAL SOLUTION L-EPS IS DECREASING IN EPSILON} we may find a $\Omega_0 \in \mathcal{F}^{W^0}$ such that $L^{\varepsilon_n \vee \varepsilon_m} \le L^{\varepsilon_n \wedge \varepsilon_m}$ for any $n,\,m \in \N$. By the boundness of $L^\varepsilon$ and Lemma \ref{lem: LOSS OF THE SYSTEM WITH INSTANTANEOUS FEEDBACK DOMINATES THAT WITH MOLLIFICATION AND MINIMAL SOLUTION L-EPS IS DECREASING IN EPSILON},
		\begin{equation*}
			\ell_t \vcentcolon= \underset{j \xrightarrow[]{} \infty}{\lim} L_t^{\varepsilon_{n_j}}\ind_{\Omega_0}
		\end{equation*}
		is well defined for any $t \in \mathbbm{D}$. Furthermore by Lemma \ref{lem: LOSS OF THE SYSTEM WITH INSTANTANEOUS FEEDBACK DOMINATES THAT WITH MOLLIFICATION AND MINIMAL SOLUTION L-EPS IS DECREASING IN EPSILON}, we may deduce that 
		\begin{equation*}
			\ell_t \vcentcolon= \underset{n \xrightarrow[]{} \infty}{\lim} L_t^{\varepsilon_{n}}\ind_{\Omega_0}
		\end{equation*}
		for any $t \in \mathbbm{D}$. It is clear by construction that $\ell_t$ is $\mathcal{F}_t^{W^0}$-measurable. Lastly, we define
		\begin{equation*}
			L_t = \underset{s \downarrow t,\, s \in \mathbbm{D}}{\lim} \ell_s.
		\end{equation*}
		It is immediate that $L$ is a \cadlag $W^0$-measurable process. Following the similar procedure as at the end of Proposition \ref{prop: CONTINUITY OF GAMMA PROPOSITIION} with the obvious changes, we obtain that $L^{\varepsilon_n} \xrightarrow[]{} L$ almost surely in $M$. For simplicity, we will denote $X^L$ by simply $X$ and let
		\begin{equation*}
			\Tilde{X} \vcentcolon= \begin{cases}
				X_{0-}   \hspace{0.6875cm} t \in [-1,0),\\
				X_t \qquad \,\, t \ge 0.
			\end{cases}
		\end{equation*}
		Then $\Tilde{X}^{\varepsilon_n} \xrightarrow[]{} \Tilde{X}$ almost surely in $\DR$. As $\Delta \Tilde{X}_t \le 0$ for every $t$ almost surely and 
		\begin{equation*}
			\prob\left[\underset{s \in (\tau_0(\Tilde{X}),\,\tau_0(\Tilde{X}) + h)}{\inf}\left\{\Tilde{X}_s - \Tilde{X}_{\tau_0(\Tilde{X})}\right\} \ge 0\right] = 0
		\end{equation*}
		for any $h \ge 0$, we have that $\tau_0$ is $M_1$-continuous at almost every path of $\Tilde{X}$. Therefore we deduce $(X,\,L)$ is a $W^0$-measurable solution to \eqref{eq: SINGULAR MCKEAN VLASOV EQUATION IN THE SYSTEMIC RISK MODEL WITH COMMON NOISE WITH SIMPLIFIED COEFFICIENTS IN THE MINIMALITY SECITON}. By Lemma \ref{lem: LOSS OF THE SYSTEM WITH INSTANTANEOUS FEEDBACK DOMINATES THAT WITH MOLLIFICATION AND MINIMAL SOLUTION L-EPS IS DECREASING IN EPSILON}, we have that $L \le \ubar{L}$ almost surely. By Proposition \ref{prop: EXISTENCE OF SOLUTIONS TO THE SINGULAR CASE}, we must have $L = \ubar{L}$ almost surely and hence $L^{\varepsilon_n} \xrightarrow[]{} \ubar{L}$ almost surely in $M$. As $\Tilde{\ubar{X}}^{\varepsilon_n}$ converges to $\Tilde{\ubar{X}}$ almost surely in $\DR$, then by the Conditional Dominated Convergence Theorem $\operatorname{Law}(\ubar{\Tilde{X}}^{\varepsilon_n}\mid \mathcal{F}^{W^0}) \xrightarrow[]{}\operatorname{Law}(\ubar{\Tilde{X}}\mid \mathcal{F}^{W^0})$ in $(\mathcal{P}(D_\R),\mathfrak{T}_{M_1}^{\text{wk}})$. By Lemma \ref{app:lem: APPENDIX LEMMA STATING THE UPPER BOUND ON THE JUMPS OF THE LIMITING PROCESS} and \citep[Proposition~3.5]{ledger2021mercy}, we have
		\begin{equation*}
			\Delta \ubar{L}_t = \inf\left\{ x \ge 0\,:\, \bm{\nu}_{t-}[0,\,\alpha x] < x\right\} \qquad \textnormal{a.s.}
		\end{equation*}
		for all $t \ge 0$.
	\end{proof}
	
	\begin{remark}[Propagation of minimality]
		This result is parallel to Theorem 6.6 in \cite{cuchiero2020propagation}, which states that minimal solutions to the finite particle system approximation will converge in probability to the limiting equation provided a unique physical solution exists. The above shows that the $W^0$-measurable minimal solutions to the smoothed system will converge to the $W^0$-measurable minimal solution of the limiting system \emph{without} needing to assume the existence of a unique physical solution. 
	\end{remark}
	
	All of the results in this section only required non-negativity of the initial condition. Moreover, we only established the existence of solutions to \eqref{eq: SINGULAR MCKEAN VLASOV EQUATION IN THE SYSTEMIC RISK MODEL WITH COMMON NOISE WITH SIMPLIFIED COEFFICIENTS IN THE MINIMALITY SECITON} and \eqref{eq: SMOOTHENED MCKEAN VLASOV EQUATION IN THE SYSTEMIC RISK MODEL WITH COMMON NOISE WITH SIMPLIFIED COEFFICIENTS IN THE MINIMALITY SECITON} but made no comments and have no results regarding the number of solutions in such a general setting. However, if we assume that the initial condition satisfies Assumption \ref{ass: MODIFIED AND SIMPLIED FROM SOJMARK SPDE PAPER ASSUMPTIONS II} \eqref{ass: MODIFIED AND SIMPLIED FROM SOJMARK SPDE PAPER ASSUMPTIONS THREE}, then there is a unique solution to \eqref{eq: SMOOTHENED MCKEAN VLASOV EQUATION IN THE SYSTEMIC RISK MODEL WITH COMMON NOISE WITH SIMPLIFIED COEFFICIENTS IN THE MINIMALITY SECITON}. In other words, the $\ubar{L}^\varepsilon$ we constructed is the only solution. Furthermore, if we further assume that the initial condition satisfies
	\begin{equation*}
		\inf \{x > 0\,:\, \bm{\nu}_{t-}[0,\alpha x] < x\} = 0,
	\end{equation*}
	then $0$ is an almost sure continuity point of $\ubar{X}$. Therefore these observations along with Theorem \ref{thm: Propogation of Minimality} allow us to deduce the following result.
	
	\begin{corollary}
		Let $(\varepsilon_n)_{n \ge 1}$ be a sequence of positive real numbers that converges to zero. We suppose that the initial condition, $X_{0-}$, satisfies Assumption \ref{ass: MODIFIED AND SIMPLIED FROM SOJMARK SPDE PAPER ASSUMPTIONS II} \eqref{ass: MODIFIED AND SIMPLIED FROM SOJMARK SPDE PAPER ASSUMPTIONS THREE} and that $\inf \{x > 0\,:\, \bm{\nu}_{t-}[0,\alpha x] < x\} = 0$. Then $\operatorname{Law}({{X}}^{\varepsilon_n}\mid \mathcal{F}^{W^0}) \xrightarrow[]{}\operatorname{Law}({\ubar{X}}\mid \mathcal{F}^{W^0})$ almost surely in $(\mathcal{P}(D_\R),\mathfrak{T}_{M_1}^{\text{wk}})$. Furthermore, ${L}^{\varepsilon_n}$ converges to $\ubar{L}$ almost surely in $M$ and $\ubar{L}$ satisfies the physical jump condition.
	\end{corollary}
	
	
	\section{Rates of convergence}\label{subsec: RATES OF CONVERGENCE OF THE MOLLIFIED PROCESS}
	
	One of the limitations of the previous arguments is that they fail to yield a rate at which the convergence will occur. For simple systems, that is in the case of no drift, no common noise and a volatility parameter set to $1$, 
	we employ a coupling argument to derive the speed of convergence, which depends on the regularity of $L$.
	
	The regularity of the loss process, $L$, has been established in the literature, \cite{delarue2022global, hambly2019mckean}, for a suitable class of initial conditions. 
	In this setting, we not only have almost-sure convergence of the stochastic process along a subsequence, but we will have uniform convergence on any time domain before the time that the regularity of $L$ decays. These results are in some sense parallel to those presented in \cite{particleSystemHittingTimeChristoph}, with the crucial difference that we are looking at the rate of convergence of systems with smoothed loss to the limiting system, as opposed to the convergence of timestepping schemes that approximate the limiting system.
	Note also that \cite{kaushansky2020convergence, cuchiero2024implicit} provide convergence orders for timestepping schemes in the case of H{\"o}lder continuous $L$, but convergence without order in the case jump case.
	
	To be precise, we will be considering the following system of equations
	\begin{equation}\label{eq: THE SYSTEM BEING EXPLORED IN THE RATE OF CONVERGECNE SECTION}
		\begin{cases}
			X_t^\varepsilon = X_{0-} + W_t - \alpha \mathfrak{L}_t^\varepsilon,\\
			\tau^\varepsilon = \inf{\{t \ge 0\,:\, X_t^\varepsilon \le 0 \}},\\
			L_t^\varepsilon = \prob\left(\tau^\varepsilon \le t\right),\\
			\mathfrak{L}_t^\varepsilon = \int_0^t \kappa^\varepsilon(t-s)L_s^\varepsilon \diff s,
		\end{cases}
		\begin{cases}
			X_t = X_{0-} + W_t - \alpha {L}_t,\\
			\tau = \inf{\{t \ge 0\,:\, X_t \le 0 \}},\\
			L_t = \prob\left(\tau \le t\right),
		\end{cases}
	\end{equation}
	where $t \ge 0$, $W$ is a standard Brownian motion, $\kappa$ is a function from $\R$ to $\R$ satisfying Assumption \ref{ass: MODIFIED AND SIMPLIED FROM SOJMARK SPDE PAPER ASSUMPTIONS II} and $\operatorname{supp}(\kappa) \subset [0,\,1]$.
	
	
	\subsection{Theoretical estimates on rates of convergence}
	
	The main result of this section is the following:
	
	\begin{proposition}\label{prop: MAIN RESULT PROPOSITION IN THE GENERAL CASE}
		Let $(X,\,L)_{t \ge 0}$ be a physical solution to \eqref{eq: THE SYSTEM BEING EXPLORED IN THE RATE OF CONVERGECNE SECTION} with initial condition $X_{0-}$. Suppose further that $X_{0-}$ admits a bounded initial density $V_{0-}$ s.t.
		\begin{equation*}
			V_{0-}(x) \le Cx^\beta\mathbbm{1}_{\{x \le x_*\}} + D\mathbbm{1}_{\{x > x_*\}} \quad \forall\, x > 0,
		\end{equation*}
		where $C,\,D,\,x_*,\,\beta > 0$ are constants with $\beta < 1$. Then, 
		for any $t_0 \in (0,\,t_{\mathrm{explode}})$ there exists a constant $\Tilde{K} = \Tilde{K}(t_0)$ s.t.
		\begin{equation*}
			\sup \limits_{s \in [0,\,t_0]} \nnorm{L_s - {L}_s^\varepsilon} \le \Tilde{K}\varepsilon^{\beta/2},
		\end{equation*}
		where
		\begin{equation}\label{eq: DEFINITION OF EXPLOSION TIME IN THE SETTING OF HOLDER REGULARITY NEAR THE BOUNDARY}
			t_{\mathrm{explode}} \vcentcolon= \sup \{t > 0 \; : \; \norm{L}_{H^1(0,\,t)} < +\infty\} \in (0,\, +\infty].
		\end{equation}
	\end{proposition}
	
	\begin{proof}
		By assumption, we are in the setting of \cite[Theorem~1.8]{hambly2019mckean}. Hence, we have a unique solution, $L$, to \eqref{eq: THE SYSTEM BEING EXPLORED IN THE RATE OF CONVERGECNE SECTION} up to the time $t_{\mathrm{explode}}$ defined as in \eqref{eq: DEFINITION OF EXPLOSION TIME IN THE SETTING OF HOLDER REGULARITY NEAR THE BOUNDARY}. Also, for all $t_0 \in (0,\, t_{\mathrm{explode}})$ there exists $K = K(t_0)$ s.t. $L \in \mathcal{S}(\frac{1-\beta}{2},\,K,\,t_0)$ where
		\begin{equation*}
			\mathcal{S}\left(\frac{1-\beta}{2},\,K,\,t_0 \right) \vcentcolon= \{l \in H^1(0,\,t_0) \; : \; {l}_t^{\prime} \le Kt^{-\frac{1-\beta}{2}} \text{ for almost all t } \in [0,\,t_0]\}
		\end{equation*}
		\\[1 ex]
		\noindent \underline{Step 1: Regularity of $L$.}
		Choose $t_0 \in (0,\,t_{\mathrm{explode}})$. As $L \in H^1(0,\,t_0)$, for Lebesgue a.e. $t,\, s \in (0,\,t_0)$ we may write
		\begin{equation*}
			L_t - L_s = \int_s^t {L}_s^\prime \diff{s} \le K(1-\gamma)^{-1}(t^{1 - \gamma} - s^{1-\gamma}),
		\end{equation*}
		where the last inequality is from $L \in \mathcal{S}(\gamma,\,K,\,t_0)$ with $\gamma = (1-\beta)/2$. This implies
		\begin{equation*}
			\frac{\nnorm{L_t - L_s}}{\nnorm{t-s}^{1 - \gamma}} \le \frac{K(t^{1 - \gamma} - s^{1-\gamma})}{(1-\gamma)\nnorm{t-s}^{1 - \gamma}} \le \frac{K}{1- \gamma}.
		\end{equation*}
		The last inequality is due to the subadditivity of concave functions. Therefore, $L_t$ is almost everywhere $\frac{\beta + 1}{2}$- H\"older continuous.
		\\[1 ex]
		\noindent \underline{Step 2: Decomposition of $L$ into an integral form.}
		We may write $L$ as
		\begin{equation*}
			L_t = \int_0^t\kappa^\varepsilon(t-s)L_s\diff{s} + \left[1 - \int_0^t\kappa^\varepsilon(t-s)\diff{s}\right]L_t + \int_0^t\kappa^\varepsilon(t-s)(L_t-L_s)\diff{s}.
		\end{equation*}
		Observe
		\begin{equation*}
			\left[1 - \int_0^t\kappa^\varepsilon(t-s)\diff{s}\right]L_t \le \frac{2K\varepsilon^{1 - \gamma}}{1 - \gamma} \qquad \textnormal{and} \qquad \int_0^t\kappa^\varepsilon(t-s)(L_t-L_s)\diff{s} \le \frac{K\varepsilon^{1-\gamma}}{1- \gamma}.
		\end{equation*}
		Therefore
		\begin{equation}\label{eq: BOUNDS ON THE ERROR TO APPROXIMATE THE LOSS FUNCTION AS AN INTEGRAL}
			L_t = \int_0^t\kappa^\varepsilon(t-s)L_s\diff{s} + \Psi^\varepsilon(t) \qquad \text{ where } \nnorm{\Psi^\varepsilon(t)} \le \frac{3K\varepsilon^{1-\gamma}}{1 - \gamma} \, \forall \,t \in [0,\,t_0].
		\end{equation}
		\\[1 ex]
		\noindent \underline{Step 3: Comparison between the delayed loss and instantaneous loss.}
		By Lemma \ref{lem: LOSS OF THE SYSTEM WITH INSTANTANEOUS FEEDBACK DOMINATES THAT WITH MOLLIFICATION AND MINIMAL SOLUTION L-EPS IS DECREASING IN EPSILON}, we have that $L \ge L^\varepsilon$, in the same spirit as \cite[Proposition~3.1]{hambly2019mckean},
		{\small\begin{align*}
				0 \le  L_t - L_t^\varepsilon \le c_1 \int_0^t \frac{L_u - \mathfrak{L}_u^\varepsilon}{\sqrt{t-u}} L^\prime_u \diff u 
				\le c_1 \int_0^t \int_0^u \kappa^\varepsilon(u - s) \frac{L_s - L_s^\varepsilon}{\sqrt{t-u}} L^\prime_u \diff s \diff u + c_1\int_0^t \frac{\Psi^\varepsilon(s)L^\prime_s}{\sqrt{t - s}} \diff s,
		\end{align*}}
		where $c_1 = \alpha \sqrt{2/\pi}$ and the second inequality follows by \eqref{eq: BOUNDS ON THE ERROR TO APPROXIMATE THE LOSS FUNCTION AS AN INTEGRAL}. As $L \in \mathcal{S}(\gamma,\,K,\,t_0)$, 
		\begin{equation*}
			0 \le \nnnorm{L_t - L_t^\varepsilon} \le K c_1 \int_0^t\int_0^u \frac{\kappa^\varepsilon(u - s)\nnnorm{L_s - L_s^\varepsilon}}{u^\gamma \sqrt{ t- u}} \diff s \diff u + K c_1 \int_0^t \frac{\nnnorm{\Psi^\varepsilon(s)}}{s^\gamma\sqrt{t - s}} \diff s.
		\end{equation*}
		By \eqref{eq: BOUNDS ON THE ERROR TO APPROXIMATE THE LOSS FUNCTION AS AN INTEGRAL}, we may find a constant $C_{K,t_0,\alpha}$ such that the second term above is bounded by $C_{K,t_0,\alpha} \varepsilon^{1 - \gamma}$. Therefore,
		\begin{align}
			0 \le \nnorm{L_t - {L}^\varepsilon_t} \le 
			K c_1 \int_0^t \nnorm{L_s - {L}_s^\varepsilon} \rho^\varepsilon(t,s) \diff s + C_{K,t_0,\alpha} \varepsilon^{1 - \gamma}, \label{eq: UPPERBOUND ON THE DIFFERENCES OF THE LOSSES}
		\end{align}
		where
		\begin{equation*}
			\rho^\varepsilon(t,s) = \int_s^t  \frac{\kappa^{\varepsilon}(u-s)}{u^{\gamma}\sqrt{t-u}} \diff u.
		\end{equation*}
		\\[1 ex]
		\noindent \underline{Step 4: Bounds on $ \rho^\varepsilon(t,s)$}\\[1 ex]
		As $\rho^\varepsilon$ depends on $t$ and $s$, we may not immediately apply Gr\"onwall's lemma or any of its generalisations. To this end, we first construct upper bounds to relax the dependence of $\rho^\varepsilon$ on $t$ and $s$ via the function $\kappa$. 
		In the case when $0 \le t - s \le \varepsilon$
		\begin{align*}
			\rho^\varepsilon(t,s) &= \int_s^t  \frac{\kappa^{\varepsilon}(u-s)}{u^\gamma\sqrt{t-u}} \diff u = \int_0^{\frac{t-s}{\varepsilon}} \frac{\kappa(\Tilde{u})}{(\varepsilon\Tilde{u} + s)^\gamma\sqrt{t-s - \varepsilon\Tilde{u}}} \diff \Tilde{u} \le \frac{\norm{\kappa}_{L^\infty}}{s^{\gamma}\varepsilon^{1/2}} \int_0^{\frac{t-s}{\varepsilon}} \frac{\diff \Tilde{u}}{\sqrt{\frac{t-s}{\varepsilon} - \Tilde{u}}}\\
			&= \frac{2\norm{\kappa}_{L^\infty}(t- s)^{1/2}}{s^\gamma\varepsilon} \le \frac{2\norm{\kappa}_{L^\infty}}{s^\gamma(t- s)^{1/2}},
		\end{align*}
		where we used the substitution $\Tilde{u} = (u-s)\varepsilon^{-1}$. In the case when ${t-s} > \varepsilon$, as the support of $\kappa^\varepsilon$ is in $[0,\,\varepsilon]$
		\begin{align*}
			\rho^\varepsilon(t,s) &= \int_s^t  \frac{\kappa^{\varepsilon}(u-s)}{u^\gamma\sqrt{t-u}} \diff u = \int^{s + \varepsilon}_s  \frac{\kappa^{\varepsilon}(u-s)}{u^\gamma\sqrt{t-u}} \diff u \\
			&\le \frac{\norm{\kappa}_{L^\infty}}{s^\gamma\varepsilon} \int^{s + \varepsilon}_s  \frac{\diff u}{\sqrt{t-u}} = \frac{2\norm{\kappa}_{L^\infty}}{s^\gamma} \left[\frac{(t-s)^{1/2} - (t - s - \varepsilon)^{1/2}}{\varepsilon}\right].
		\end{align*}
		\\[1 ex]
		\noindent \underline{Step 5: Gr\"onwall type argument}\\[1 ex]
		Now that we have sufficiently decoupled $\kappa$ from $\rho^\varepsilon$, we may put \eqref{eq: UPPERBOUND ON THE DIFFERENCES OF THE LOSSES} into a form where we may apply a generalised Gr\"onwall Lemma. By step 4 case 1 and \eqref{eq: UPPERBOUND ON THE DIFFERENCES OF THE LOSSES}, for $t \le \varepsilon$,
		\begin{align*}
			\nnorm{L_t - {L}^\varepsilon_t} \le K c_1 \int_0^t 2\norm{\kappa}_{L^\infty} s^{-\gamma} (t - s)^{-1/2} \nnorm{L_s - {L}_s^\varepsilon} \rho^\varepsilon(t,s) \diff s +  C_{K,t_0,\alpha} \varepsilon^{1 - \gamma}.
		\end{align*}
		By the second case of step 4 and \eqref{eq: UPPERBOUND ON THE DIFFERENCES OF THE LOSSES}, we have for $t > \varepsilon $
		\begin{align*}
			\nnorm{L_t - {L}^\varepsilon_t} &\le K c_1 \int_0^{t - \varepsilon} \nnorm{L_s - {L}_s^\varepsilon} \rho^\varepsilon(t,s) \diff s + K c_1 \int_{t - \varepsilon}^t \nnorm{L_s - {L}_s^\varepsilon} \rho^\varepsilon(t,s) \diff s+C_{K,t_0,\alpha} \varepsilon^{1 - \gamma}\\
			&\le 2 K c_1 \norm{\kappa}_{L^\infty} \int_0^{t - \varepsilon} \left[\frac{(t-s)^{1/2} - (t - s - \varepsilon)^{1/2}}{\varepsilon}\right]s^{-\gamma} \nnorm{L_s - {L}_s^\varepsilon} \diff s \\
			&\phantom{\le}+ 2 K c_1 \norm{\kappa}_{L^\infty} \int_{t - \varepsilon}^t (t -s)^{-1/2}s^{-\gamma}\nnorm{L_s - {L}_s^\varepsilon} \diff s + C_{K,t_0,\alpha} \varepsilon^{1 - \gamma}\\
			&\le 2 K c_1 \norm{\kappa}_{L^\infty} \sum_{j \ge 2 } \Tilde{C}_j \varepsilon^{j-1} \int_0^{t - \varepsilon} (t-s)^{\frac{-2j+1}{2}}s^{-\gamma} \nnorm{L_s - {L}_s^\varepsilon} \diff s \\
			&\phantom{\le}+ 2 K c_1 \norm{\kappa}_{L^\infty} \int_{0}^t (t -s)^{-1/2}s^{-\gamma}\nnorm{L_s - {L}_s^\varepsilon} \diff s+ C_{K,t_0,\alpha} \varepsilon^{1 - \gamma},
		\end{align*}
		where the last line follows from applying Taylor's Theorem and the Monotone Convergence Theorem. We note $\Tilde{C}_j \vcentcolon= (2j-2)!/[j!(j-1)!2^{2j-1}]$ is summable. Now we turn our attention onto the expression in the penultimate line. 
		In the case when $\varepsilon < t \le 2\varepsilon$, \begin{equation*}
			\Tilde{C}_j \varepsilon^{j-1} \int_0^{t - \varepsilon} (t-s)^{\frac{-2j+1}{2}}s^{-\gamma} \diff s \le \Tilde{C}_j \varepsilon^{j-1} \varepsilon^{\frac{-2j+1}{2}} \int_0^{t - \varepsilon} s^{-\gamma}  \diff s 
			\le \Tilde{C}_j \varepsilon^{\frac{-1}{2}} \int_0^{\varepsilon} s^{-\gamma} \diff s 
			\le \frac{\Tilde{C}_j \varepsilon^{\frac{1}{2} - \gamma}}{1 - \gamma}
		\end{equation*}
		where the first inequality follows from the fact that $(-2j + 1)/2 < 0$ as $j \ge 2$ and $t - s \in [\varepsilon,\,t]$ for $s \in [0,\,t-\varepsilon]$. In the case when $ t >  2\varepsilon$, we observe that
		\begin{equation*}
			\Tilde{C}_j \varepsilon^{j-1} \int_0^{\varepsilon} (t-s)^{\frac{-2j+1}{2}}s^{-\gamma} \diff s 
			\le \Tilde{C}_j \varepsilon^{j-1} \varepsilon^{\frac{-2j+1}{2}} \int_0^{\varepsilon} s^{-\gamma}  \diff s 
			\le \frac{\Tilde{C}_j \varepsilon^{\frac{1}{2} - \gamma}}{1 - \gamma},
		\end{equation*}
		and
		\begin{align*}
			\Tilde{C}_j \varepsilon^{j-1} \int_\varepsilon^{t - \varepsilon} (t-s)^{\frac{-2j+1}{2}}s^{-\gamma} \diff s &\le \Tilde{C}_j \varepsilon^{j-1} \varepsilon^{-\gamma} \int_\varepsilon^{t - \varepsilon} (t-s)^{\frac{-2j+1}{2}}  \diff s \\
			&= \Tilde{C}_j \varepsilon^{j-1 - \gamma} \frac{2}{2j-3}\left.(t - s)^{\frac{-2j+3}{2}}\right|_{s = \varepsilon}^{t - \varepsilon}\\ 
			&\le \frac{2\Tilde{C}_j \varepsilon^{j-1 - \gamma}\varepsilon^{\frac{-2j+3}{2}}}{2j - 3}
			= \frac{2\Tilde{C}_j \varepsilon^{1/2 - \gamma}}{2j - 3} 
			\le \frac{2\Tilde{C}_j \varepsilon^{1/2 - \gamma}}{1 - \gamma}. 
		\end{align*}
		Therefore, we have shown that
		\begin{equation*}
			\Tilde{C}_j \varepsilon^{j-1} \int_0^{t - \varepsilon} (t-s)^{\frac{-2j+1}{2}}s^{-\gamma} \diff s 
			\le \frac{3\Tilde{C}_j \varepsilon^{1/2 - \gamma}}{1 - \gamma}
		\end{equation*}
		for all $t > \varepsilon$. As $L$ and ${L}^\varepsilon$ are bounded by $1$, we have independent of $t$ being greater or less than $\varepsilon$,
		{\small\begin{align*}
				\nnorm{L_t - {L}^\varepsilon_t} &\le  2 K c_1 \norm{\kappa}_{L^\infty} \int_{0}^t (t -s)^{-1/2}s^{-\gamma}\nnorm{L_s - {L}_s^\varepsilon} \diff s+ \frac{12 K c_1 \norm{\kappa}_{L^\infty}\varepsilon^{1/2 - \gamma} \sum_{j \ge 2} \Tilde{C}_j}{1 - \gamma} + C_{K,t_0,\alpha} \varepsilon^{1 - \gamma} \\
				&= 2 K c_1 \norm{\kappa}_{L^\infty} \int_{0}^t (t -s)^{-1/2}s^{-\gamma}\nnorm{L_s - {L}_s^\varepsilon} \diff s+ C_{K,t_0,\alpha,\gamma}\varepsilon^{1/2 - \gamma},
		\end{align*}}
		for any $t \in [0,\,t_0]$. Lastly, by Proposition \ref{prop: GENERALISED GRONWALL TYPE INEQUALITY}, using $\Tilde{\beta} = 1/2$ and $\Tilde{\alpha} = 1 - \gamma$, then $\Tilde{\alpha} + \Tilde{\beta}  - 1 > 0$ as $\gamma < 1/2$ and
		\begin{align*}
			\nnorm{L_t - {L}^\varepsilon_t} &\le C_{K,t_0,\alpha,\gamma}\varepsilon^{1/2 - \gamma}\sum_{n \ge 0} (2 K c_1 \norm{\kappa}_{L^\infty})^n C_n t_0^{n(1/2 - \gamma)}\\
			&=  C_{K,t_0,\alpha,\gamma}\varepsilon^{\beta/2}\sum_{n \ge 0} (2 K c_1 t_0^{\beta/2} \norm{\kappa}_{L^\infty})^n C_n,
		\end{align*}
		where in the last equality we used $\gamma = (1 - \beta)/2$. This completes the proof.
	\end{proof}
	
	The works of Fasano et al.\ \cite{fasano1981new, fasano1983critical}, Di Benedetto et al.\ \cite{dibenedetto1984ill}, and Chayes et al.\ \cite{chayes1996hydrodynamic, chayes2008two} extensively investigate the supercooled Stefan problem, focusing on the existence of a unique solution without blow-ups for all time or until the entire liquid freezes. Recently, Delarue et al.\ \citep{delarue2022global} established global uniqueness for the system described in \eqref{eq: THE SYSTEM BEING EXPLORED IN THE RATE OF CONVERGECNE SECTION}, under the condition that the density of the initial condition $X_{0-}$ undergoes only a finite number of changes in monotonicity. 
	Under this assumption, the loss is continuously differentiable on $(0,\,t_{\mathrm{explode}})$ for 
	$t_{\mathrm{explode}}$
	defined in \eqref{eq: DEFINITION OF EXPLOSION TIME IN THE SETTING OF HOLDER REGULARITY NEAR THE BOUNDARY}.
	Moreover, if the initial density has sufficient regularity, the loss will be continuously differentiable down to $t=0$. Motivated by these results, we next investigate the rate of convergence when the loss function is differentiable.
	
	
	\begin{proposition}\label{prop: MAIN RESULT PROPOSITION IN THE CONTINUOUS CASE}
		Suppose we have a unique physical solution $(X, L)$ to \eqref{eq: THE SYSTEM BEING EXPLORED IN THE RATE OF CONVERGECNE SECTION} such that $L \in \mathcal{C}^1([0,\,t_{\mathrm{explode}}))$ for $t_{\mathrm{explode}} \in (0,\,\infty]$. Then for any $t_0 \in (0,\,t_{\mathrm{explode}})$, there exists a constant $\Tilde{K} = \Tilde{K}(t_0)$ such that
		\begin{equation*}
			\sup \limits_{s \in [0,\,t_0]} \nnorm{L_s - {L}_s^\varepsilon} \le \Tilde{K}\varepsilon^{1/2}.
		\end{equation*}
	\end{proposition}

	\begin{proof}
		See appendix.
	\end{proof}
	
	
	\subsection{Numerical simulations}\label{sec: SECTION ON THE NUMERICAL SCHEMES}
	Lastly, we investigate the convergence rate of the smoothed loss function towards the singular loss function through numerical simulations. The estimates from the previous section, for the case without common noise, give upper bounds on the rate at which the smoothed system will approach the singular system, prior to the decline in regularity of the singular loss function. The proofs employed in the analysis utilised relatively crude upper bounds, prompting the question of whether the obtained rates are optimal.
	
	To the best of our knowledge, there is no existing literature on the regularity of the loss process in the presence of common noise. Consequently, the theoretical methods employed earlier may not be applicable in this scenario. Nevertheless, we can still explore the convergence rate numerically in this context as well.  We consider the simplest setting with common noise,
	\begin{equation}
		\begin{cases}
			X_t^\varepsilon = X_{0-} + (1 - \rho^2)^{1/2}W_t + \rho W_t^0 - \alpha \mathfrak{L}_t^\varepsilon,\\
			\tau^\varepsilon = \inf{\{t \ge 0\,:\, X_t^\varepsilon \le 0 \}},\\
			L_t^\varepsilon = \prob\left[\tau^\varepsilon \le t \mid \mathcal{F}_t^{W^0}\right],\\
			\mathfrak{L}_t^\varepsilon = \int_0^t \kappa^\varepsilon(t-s)L_s^\varepsilon \diff s,
		\end{cases}
		\begin{cases}
			X_t = X_{0-} + (1 - \rho^2)^{1/2}W_t + \rho W_t^0 - \alpha {L}_t,\\
			\tau = \inf{\{t \ge 0\,:\, X_t \le 0 \}},\\
			L_t = \prob\left[\tau \le t\mid \mathcal{F}^{W^0}_t\right],
		\end{cases}
	\end{equation}
	where $\rho \in [0,1)$ is a fixed constant. We propose a numerical scheme that employs a particle system approximation to compute both the limiting and smoothed loss functions. Instead of employing numerical integration to compute the mollified loss of $X^\varepsilon$, the system will feel the impulse from a particle hitting the boundary at a random time in the future sampled from a random variable whose probability density function is the mollification kernel. The scheme is given in Algorithm \ref{alg: ALGORITHM USED TO APPROXIMATE THE SMOOTHED LOSS WITH COMMON NOISE}.

	By setting $\rho$ to zero, the algorithm approximates the loss in the setting without common noise. To compute the limiting loss function we set $\varsigma$ to zero. In the case when $\rho = 0$ and $\varsigma = 0$, we recover the numerical scheme proposed in \cite{particleSystemHittingTimeChristoph, kaushansky2020convergence}. In the numerical experiments below, we used $10^{6.5}$ particles and a uniform time mesh of width $\Delta_t \vcentcolon= \min_i\{\varepsilon_i\} / 10$, with $\{\varepsilon_i\}_i$ the set of delay values, so that $t_i = i \times \Delta_t$ in Algorithm \ref{alg: ALGORITHM USED TO APPROXIMATE THE SMOOTHED LOSS WITH COMMON NOISE}.
	
	Overall, given sufficient regularity of the loss function, a rate of convergence close to $1$ is observed. In the other cases studied with H\"older initial data, with the possibility of there being a jump after the test interval, and with common noise, the rate of convergence appears to be between $1/2$ and $1$. See Appendix \ref{app:sec: FURTHER ANALYSIS AND DISCUSSION ON NUMERICS} for further analysis regarding the rate of convergence and how $\Delta_t$ affects the estimated rate. 
	
	\subsection{Initial density vanishing at zero and no discontinuity or common noise}\label{subsec: SUBSECTION ON THE RATE OF CONVERGENCE IN THE CONTINUOUS SETTING}
	
	Two different initial conditions were examined in our experimental analysis, and no discontinuity was observed in either case. In the first simulation, we set $X_{0-}$ to follow a uniform distribution on $[0.25, 0.35]$, with $\alpha =0.5$. In the second scenario, $X_{0-}$ was generated from a gamma distribution with parameters $\left(2.1,\frac{1}{2}\right)$, with $\alpha=1.3$. Interestingly, the data from \cref{fig: FIGURE IN THE CONTINUOUS SETTING} indicate a convergence rate of $1$ in both cases. This exceeds the theoretically guaranteed convergence rate of $1/2$.
	
	\vspace{10pt}
	\begin{algorithm}[H]
		\DontPrintSemicolon
		\SetAlgoLined
		\Require{$N\,-\,$ number of interacting particles}
		\Require{$n\,-\,$ number of time steps: $0 < t_1 < t_2 < \ldots < t_n$}
		\Require{$\varepsilon\,-\,$ the strength of the delay}
		Draw one sample of  $W^0$\;
		Draw $N$ samples of $X_{0-}$,  $W$, and $\varsigma$ (r.v. with distribution $\kappa^\varepsilon(t)\diff t$)\;
		\For{$i = 1:n$\,}{
			$\hat{L}^\varepsilon_{t_i} = \frac{1}{N}\sum_{m=1}^N \ind_{(-\infty,0]}(\min_{t_j < t_i} \{ \hat{X}_{t_j}^{(m)}\})$\;
			\For{$k = 1:N$\,}{
				Update $\hat{X}_{t_i}^{(k)} = X_{0-}^{(k)} + (1 - \rho^2)^{1/2}W_{t_i}^{(k)} + \rho W^0_{t_i} - \frac{\alpha}{N}\sum_{m=1}^N \ind_{(-\infty,0]}(\min_{t_j < t_i - \varsigma^{(m)}} \{ \hat{X}_{t_j}^{(m)}\})$\;
			}
		}
		\caption{Discrete time Monte Carlo scheme for simulation of the smoothed loss process with common noise}
		\label{alg: ALGORITHM USED TO APPROXIMATE THE SMOOTHED LOSS WITH COMMON NOISE}
	\end{algorithm}
	
	\begin{figure}[H]
		\makebox[\linewidth][c]{
			\begin{subfigure}[b]{0.55\columnwidth}
				\centering
				\includegraphics[width=\columnwidth]{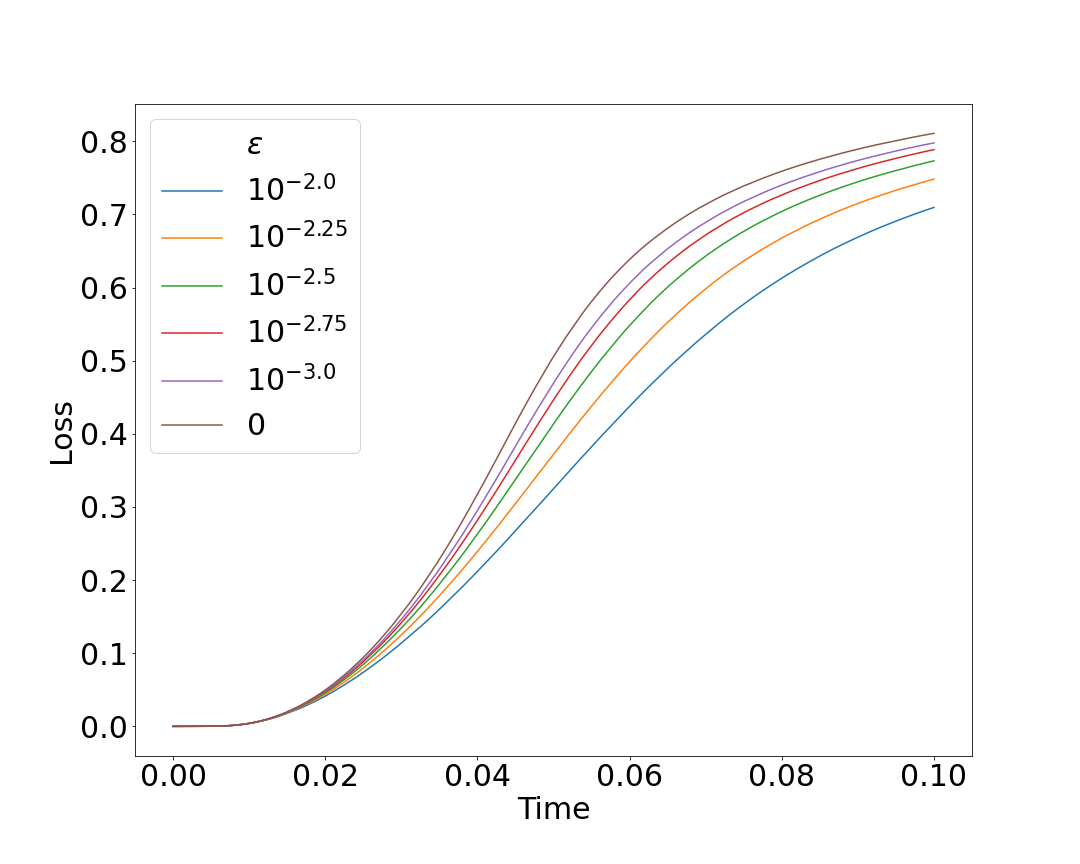}
				\caption{$X_{0-} \sim^d  \operatorname{Uniform}[0.25,\,0.35],\, \alpha = 0.5$}
				
			\end{subfigure}
			\hfill
			\begin{subfigure}[b]{0.55\columnwidth}
				\centering
				\includegraphics[width=\columnwidth]{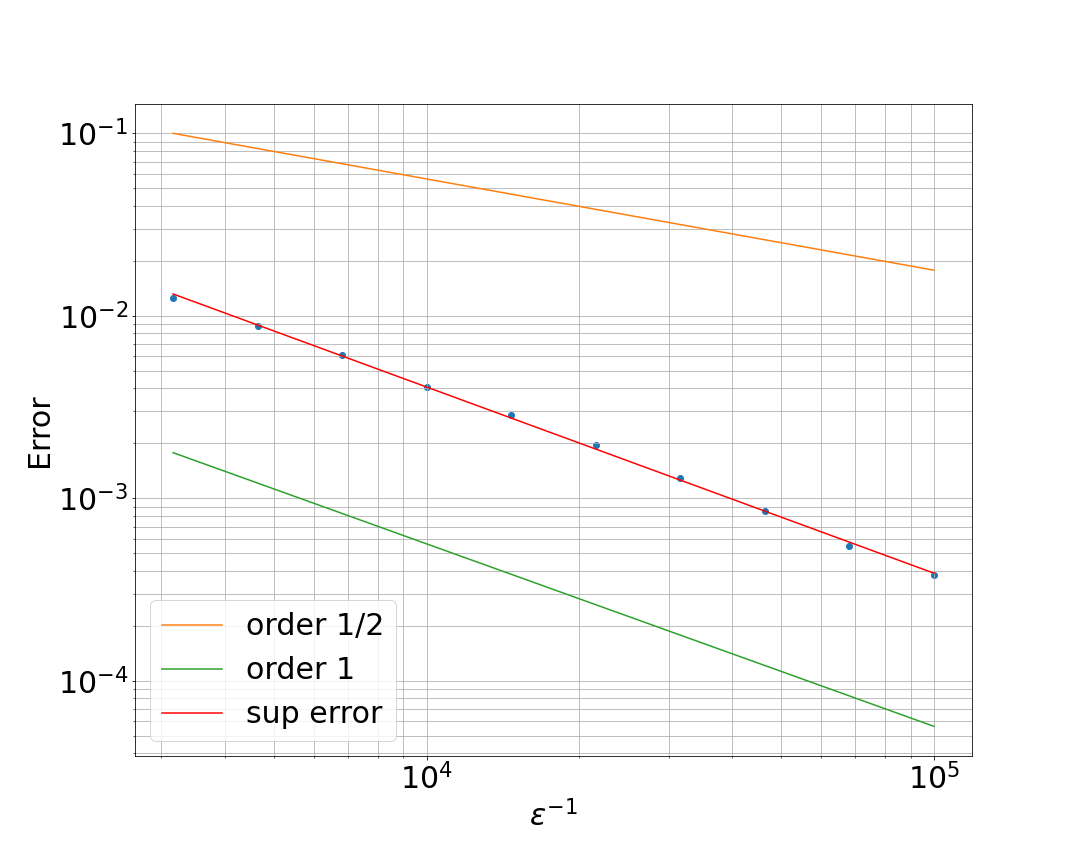}
				\caption{Rate of Convergence}
				
			\end{subfigure}
		}
		
		\makebox[\linewidth][c]{
			\begin{subfigure}[b]{0.55\columnwidth}
				\centering
				\includegraphics[width=\columnwidth]{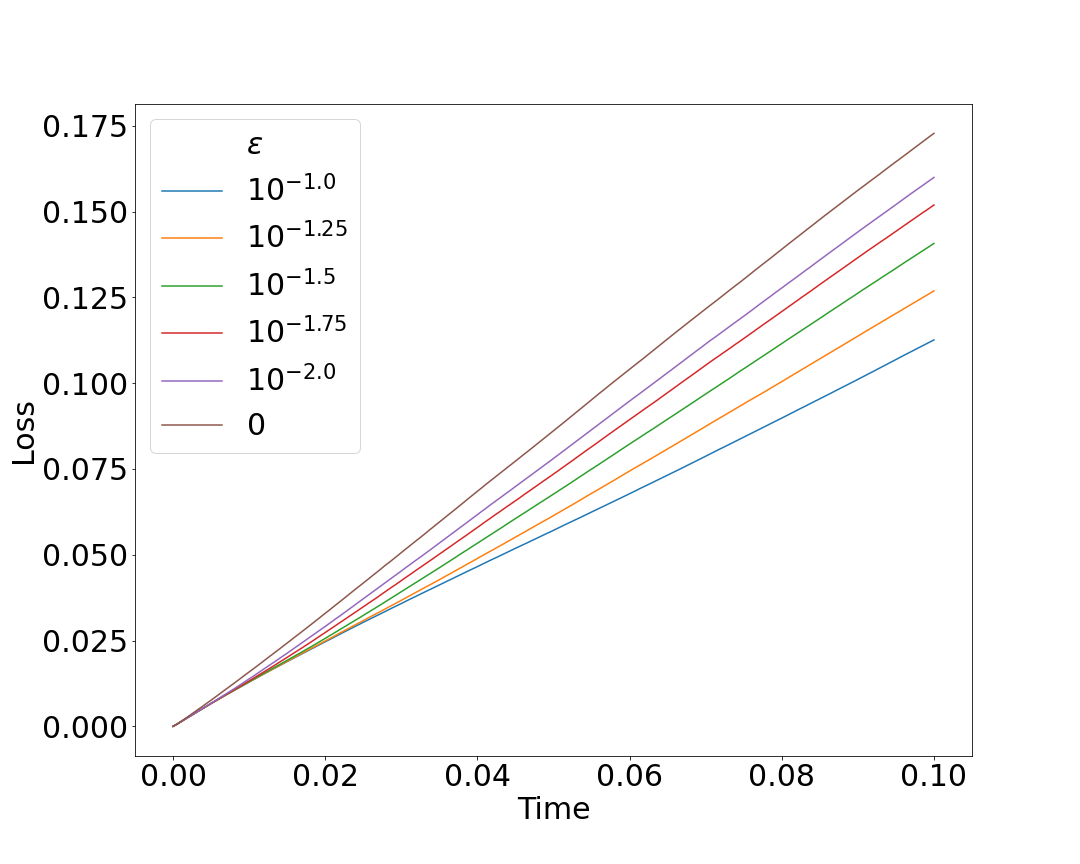}
				\caption{$X_{0-} \sim^d  \Gamma(2.1,\frac{1}{2}),\, \alpha = 1.3$}
				
			\end{subfigure}
			\hfill
			\begin{subfigure}[b]{0.55\columnwidth}
				\centering
				\includegraphics[width=\columnwidth]{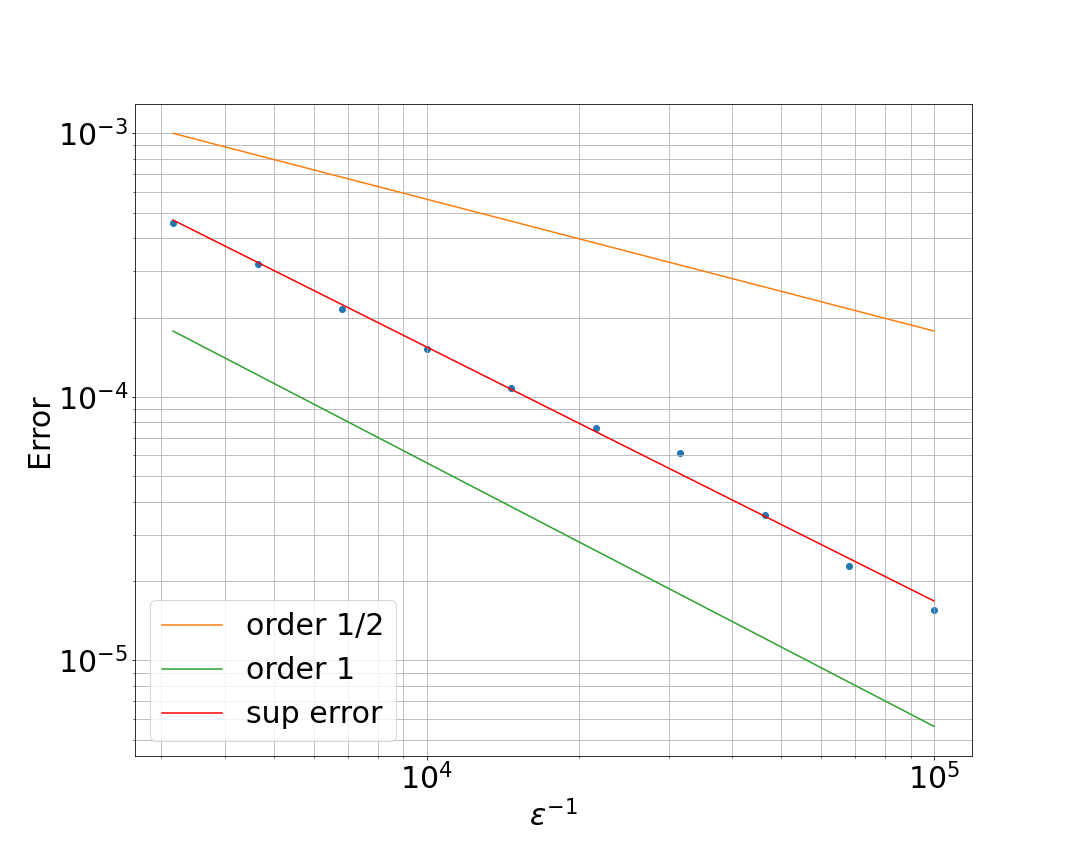}
				\caption{Rate of Convergence}
				
			\end{subfigure}
		}
		\caption{Initial density vanishing at zero with no discontinuity or common noise}
		\label{fig: FIGURE IN THE CONTINUOUS SETTING}
	\end{figure}
	
	\subsection{Setting with discontinuity and without common noise}
	
	To simulate a setting where we would see a systemic event, we changed the parameters of the Gamma distribution such that most of the mass will be near the boundary and made $\alpha$ sufficiently large. In \cref{fig: FIGURE IN THE DISCONTINUOUS SETTING}, we conducted simulations using two different initial conditions. In the first case, we set $X_{0-}$ to follow a Gamma distribution with parameters $(1.2, 0.5)$ and set $\alpha= 0.9$. In the second case, $X_{0-}$ was generated from a Gamma distribution with parameters $(1.4, 0.5)$, and $\alpha= 2$. {Within this particular setup, we observe a convergence rate between $1/2$ and $1$ 
		prior to the first jump, while the theory only gives a rate below $0.5$.} The rate 
	does not display a clear dependence on the characteristics of the density of $X_{0-}$ near the boundary, in contrast to the theoretical estimates; see Table \ref{tab: ALTERNATIVE RATE OF CONVERGENCE ANALYSIS TABLE} for more details on the data. 
	
	\begin{figure}[!t]
		
		\makebox[\linewidth][c]{
			\begin{subfigure}[b]{0.42\columnwidth}
				\centering
				\includegraphics[width=\columnwidth]{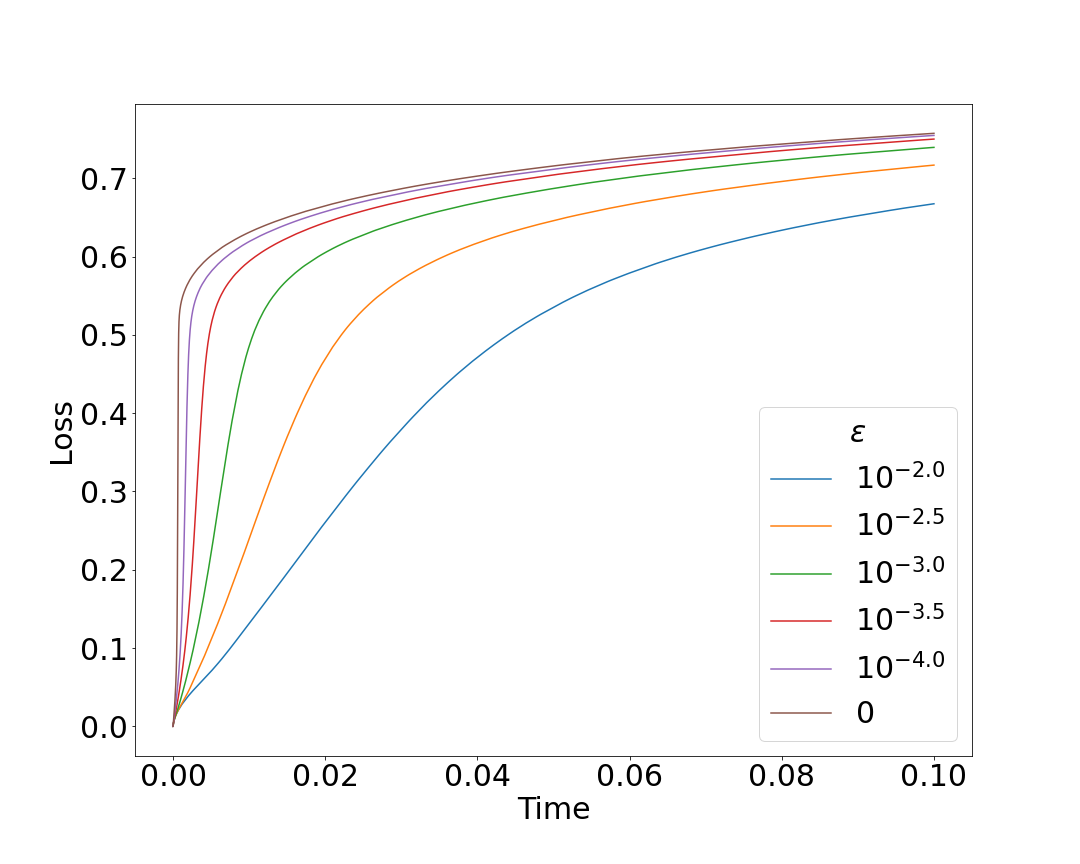}
				\caption{$X_{0-} \sim^d  \Gamma(1.2,0.5),\, \alpha = 0.9$}
				
			\end{subfigure}
			\hfill
			\begin{subfigure}[b]{0.42\columnwidth}
				\centering
				\includegraphics[width=\columnwidth]{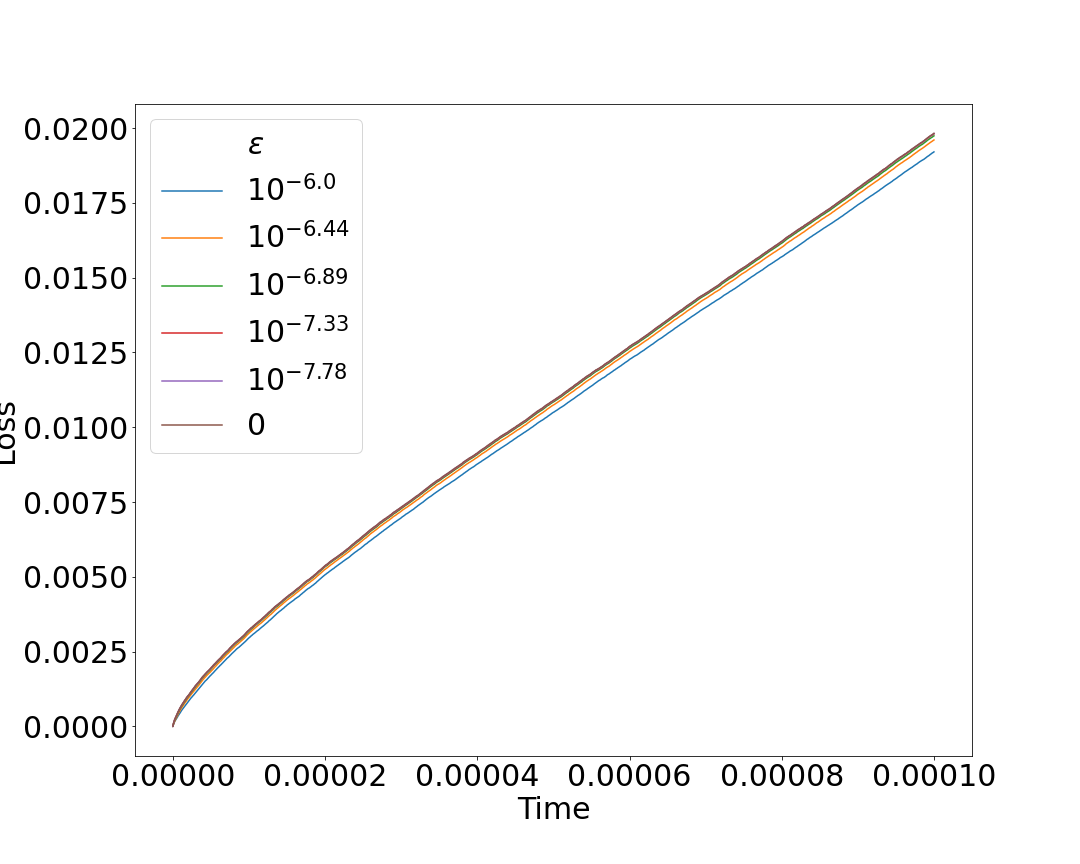}
				\caption{$t_{\max} = 0.0001$}
				
			\end{subfigure}
			\hfill
			\begin{subfigure}[b]{0.42\columnwidth}
				\centering
				\includegraphics[width=\columnwidth]{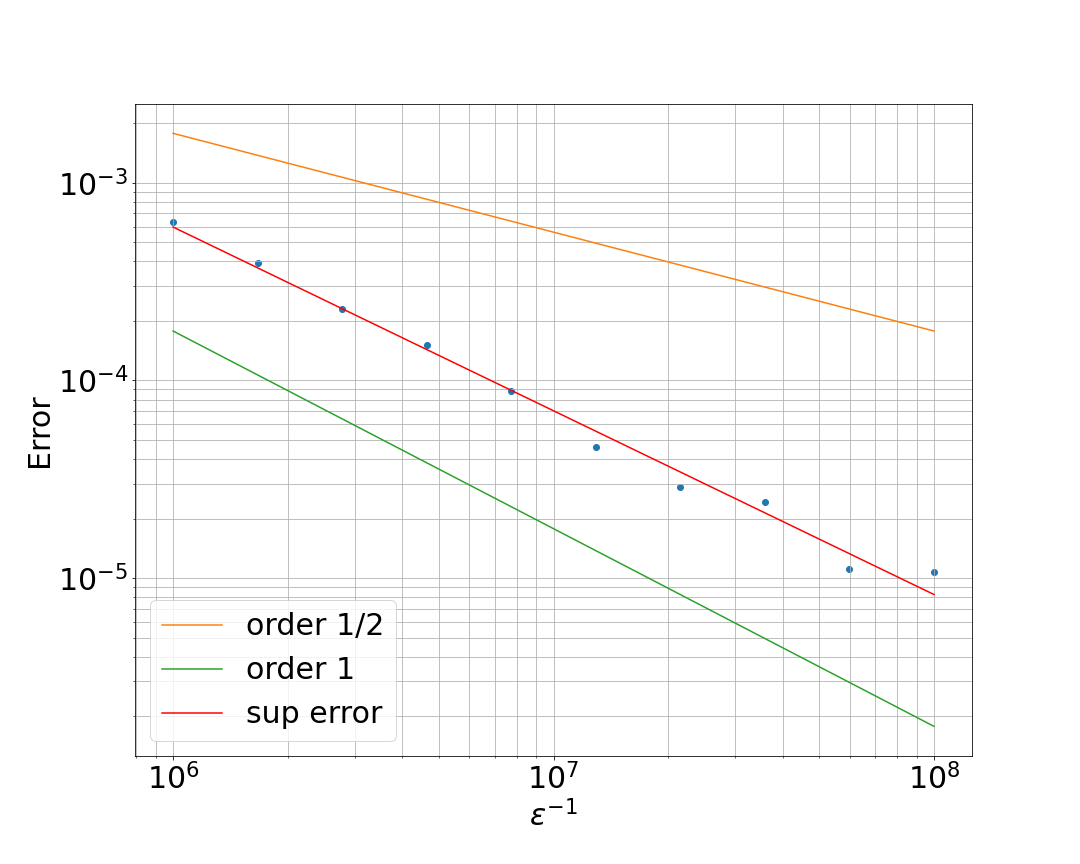}
				\caption{Rate of Convergence}
				
			\end{subfigure}
		}
		
		\makebox[\linewidth][c]{
			\begin{subfigure}[b]{0.42\columnwidth}
				\centering
				\includegraphics[width=\columnwidth]{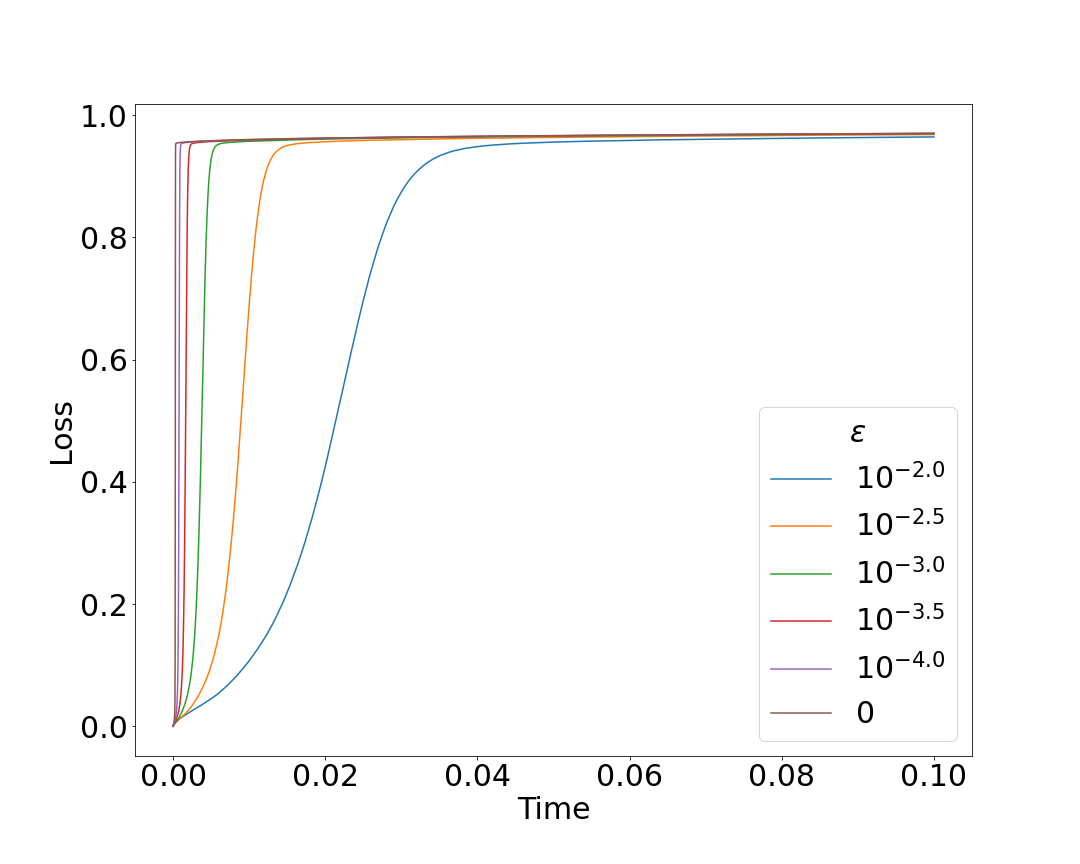}
				\caption{$X_{0-} \sim^d  \Gamma(1.4,0.5),\, \alpha = 2$}
				
			\end{subfigure}
			\hfill
			\begin{subfigure}[b]{0.42\columnwidth}
				\centering
				\includegraphics[width=\columnwidth]{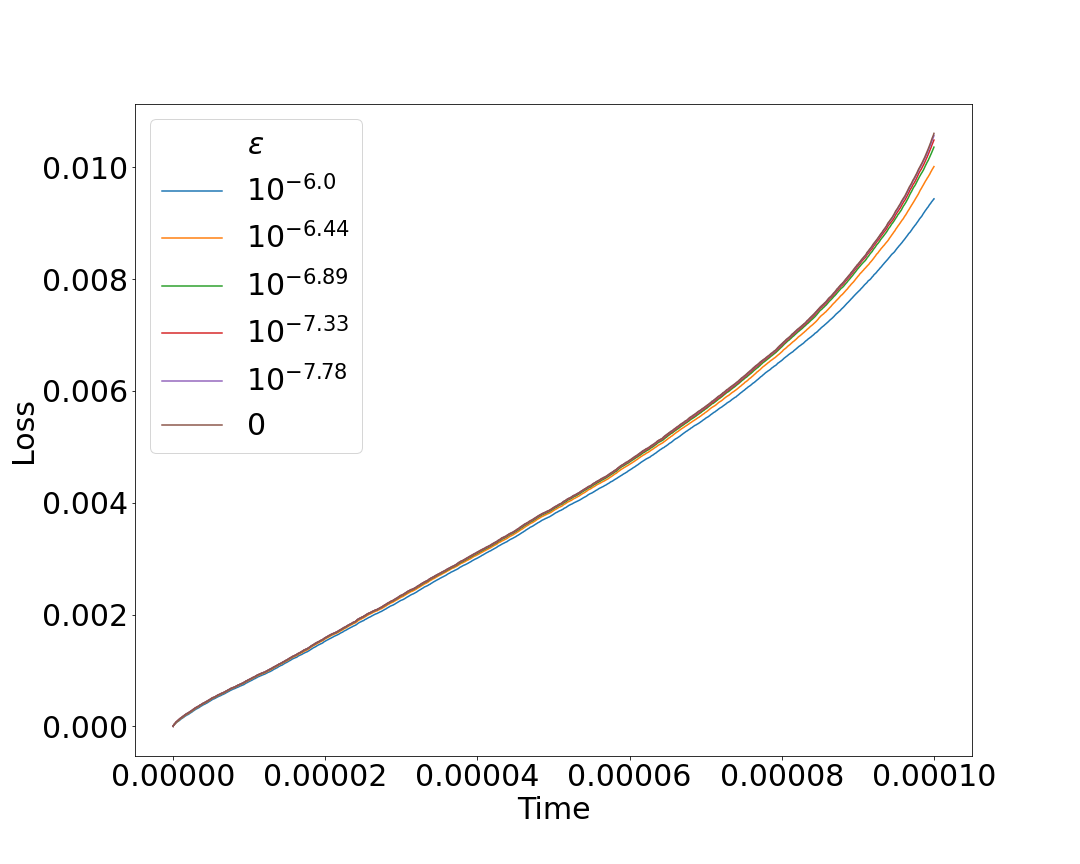}
				\caption{$t_{\max} = 0.0001$}
				
			\end{subfigure}
			\hfill
			\begin{subfigure}[b]{0.42\columnwidth}
				\centering
				\includegraphics[width=\columnwidth]{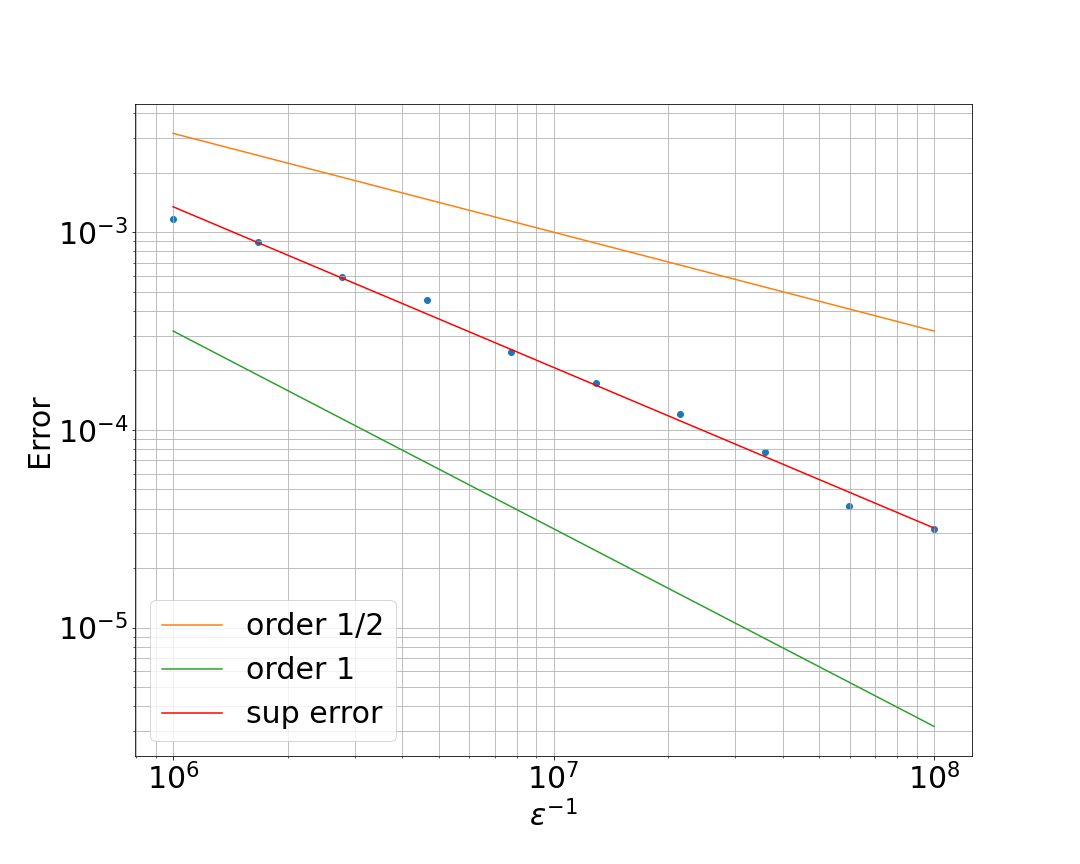}
				\caption{Rate of Convergence}
				
			\end{subfigure}
		}
		\caption{Initial density vanishing at zero with discontinuity and no common noise}
		\label{fig: FIGURE IN THE DISCONTINUOUS SETTING}
	\end{figure}

	\subsection{Simulations with common noise}
	In both experiments in this subsection, we assigned $X_{0-}$ a uniform distribution over the interval $[0.25, 0.35]$, 
	the same as in subsection \ref{subsec: SUBSECTION ON THE RATE OF CONVERGENCE IN THE CONTINUOUS SETTING}, and
	set $\alpha=0.5$ and $\rho=0.5$, with different common noise paths for each experiment. In the first simulation, the common noise path increases to $1$ over the time domain. This led to the loss process becoming rougher than the loss in the previous setting. In the second simulation, the common noise path decreases to $-1$. This induces a systemic event due to the rapid loss of mass. {Despite the differences between the scenarios, we observed a similar rate of convergence between $1/2$ and $1$ as illustrated in \cref{fig: FIGURE IN THE COMMON NOISE SETTING}.}
	

\begin{figure}[!t]
	\makebox[\linewidth][c]{
		\begin{subfigure}[b]{0.42\columnwidth}
			\centering
			\includegraphics[width=\columnwidth]{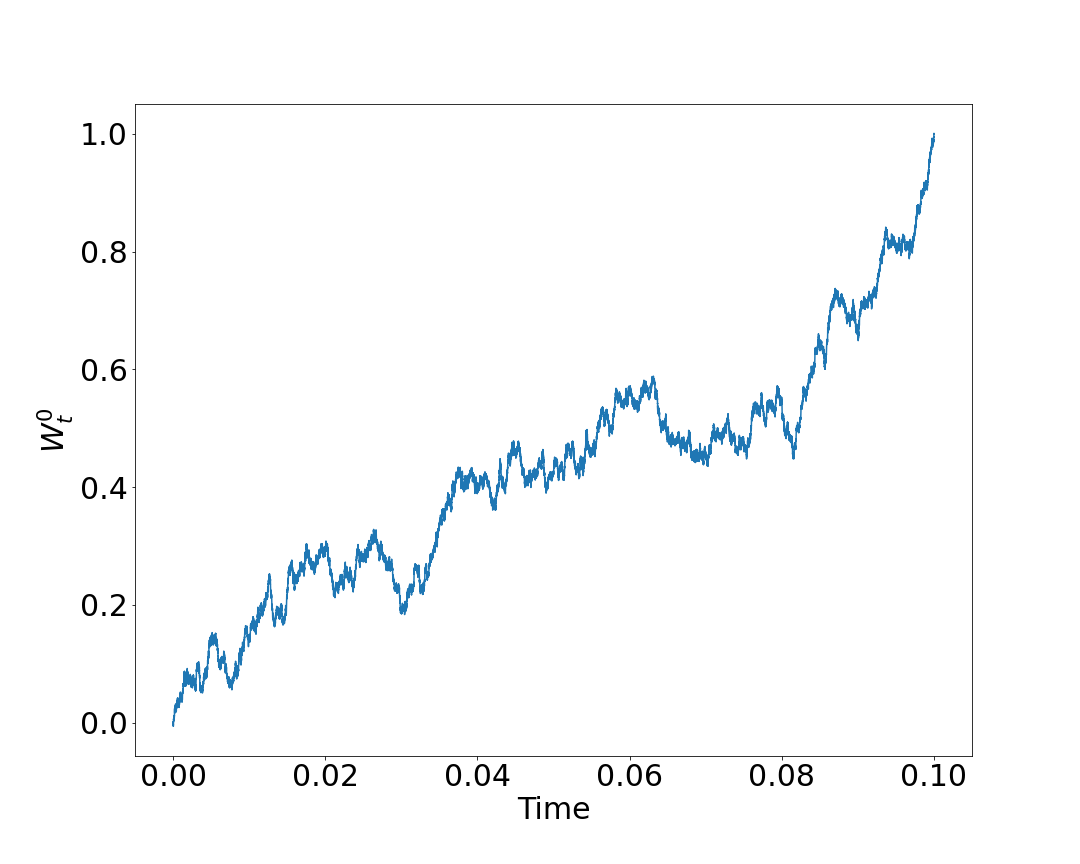}
			\caption{Common Noise Path}
			
		\end{subfigure}
		\hfill
		\begin{subfigure}[b]{0.42\columnwidth}
			\centering
			\includegraphics[width=\columnwidth]{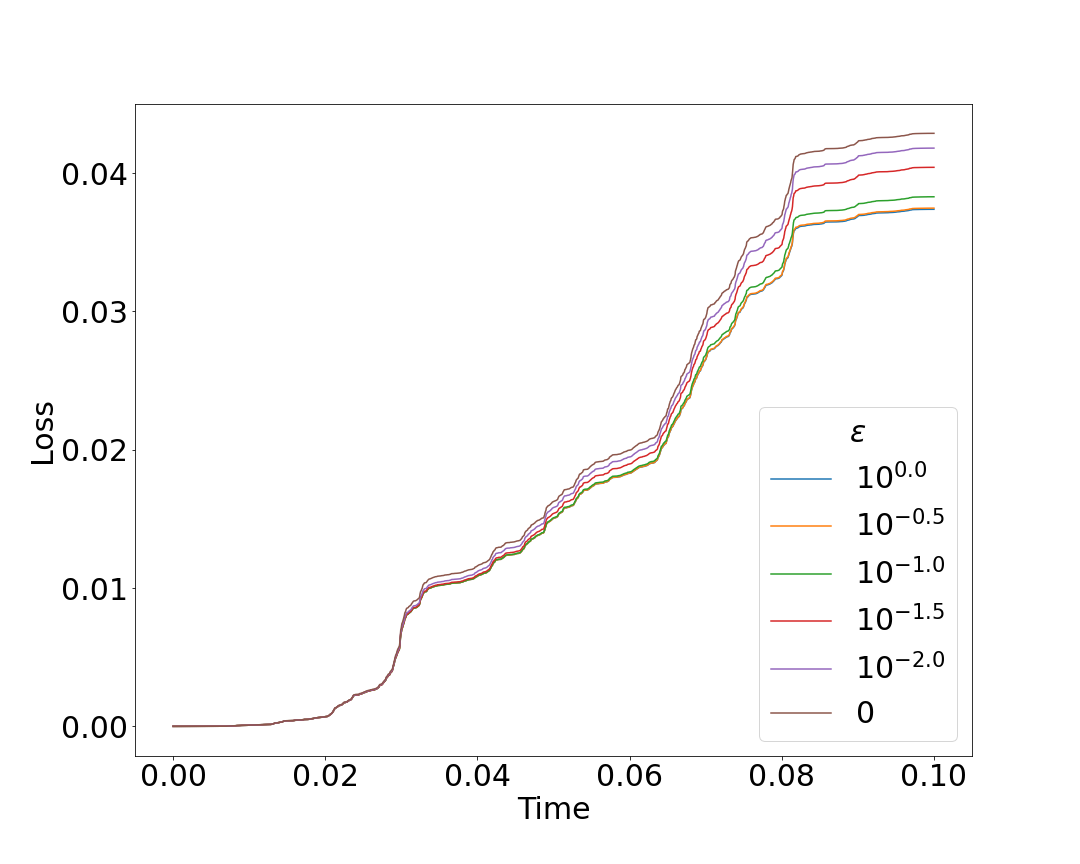}
			\caption{$X_{0-} \sim^d  \operatorname{Uniform}[0.25,\,0.35],\, \alpha = 0.5$}
			
		\end{subfigure}
		\hfill
		\begin{subfigure}[b]{0.42\columnwidth}
			\centering
			\includegraphics[width=\columnwidth]{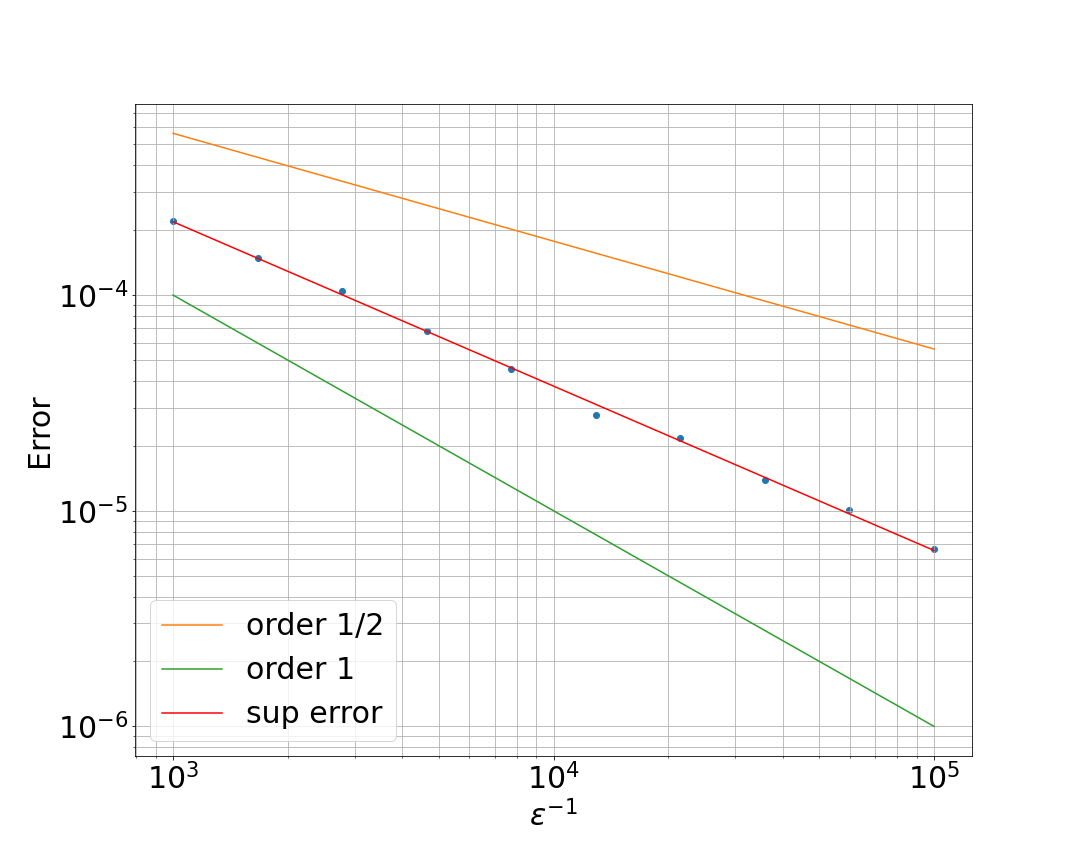}
			\caption{Rate of Convergence}
			
		\end{subfigure}
	}
	
	\makebox[\linewidth][c]{
		\begin{subfigure}[b]{0.42\columnwidth}
			\centering
			\includegraphics[width=\columnwidth]{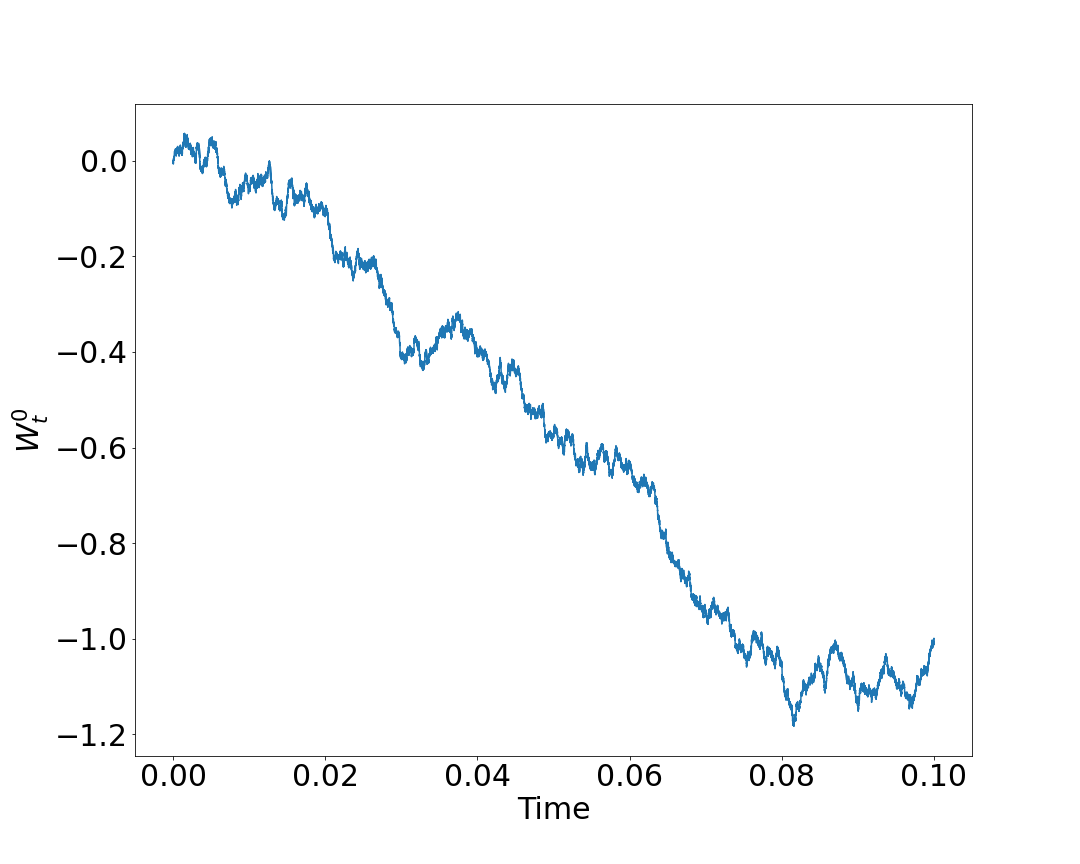}
			\caption{Common Noise Path}
			
		\end{subfigure}
		\hfill
		\begin{subfigure}[b]{0.42\columnwidth}
			\centering
			\includegraphics[width=\columnwidth]{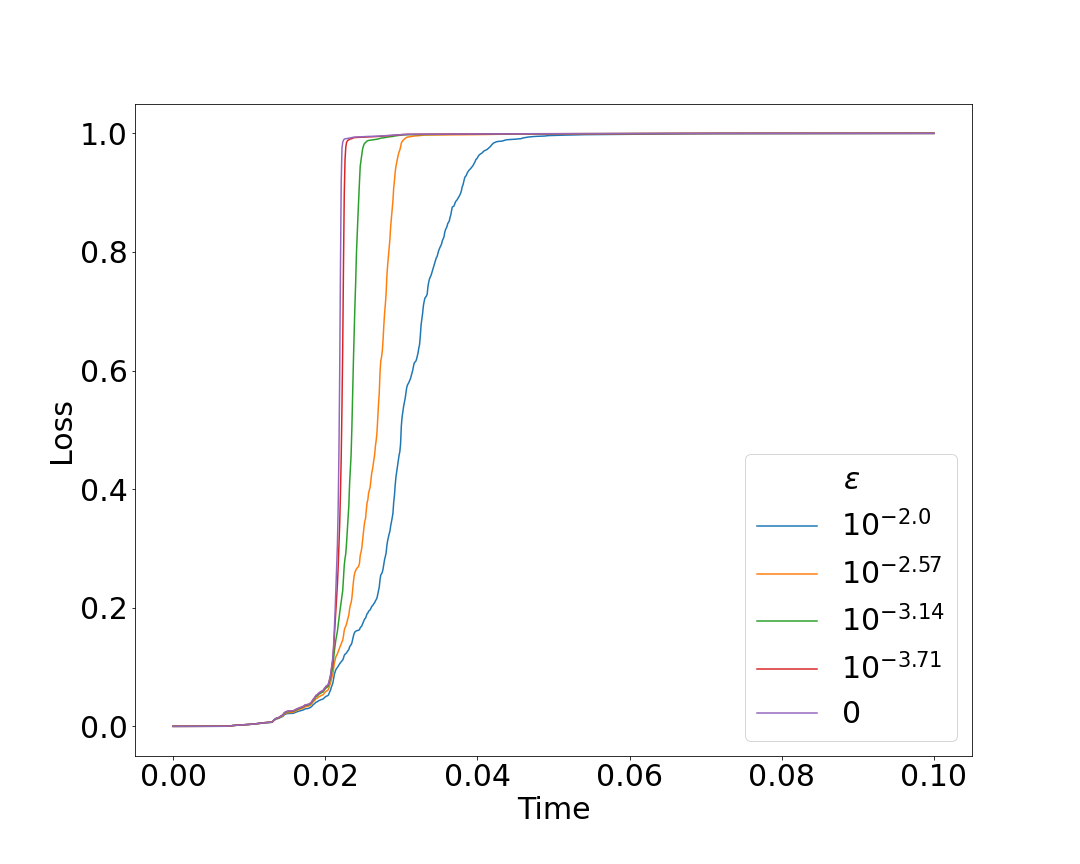}
			\caption{$X_{0-} \sim^d  \operatorname{Uniform}[0.25,\,0.35],\, \alpha = 0.5$}
			
		\end{subfigure}
		\hfill
		\begin{subfigure}[b]{0.42\columnwidth}
			\centering
			\includegraphics[width=\columnwidth]{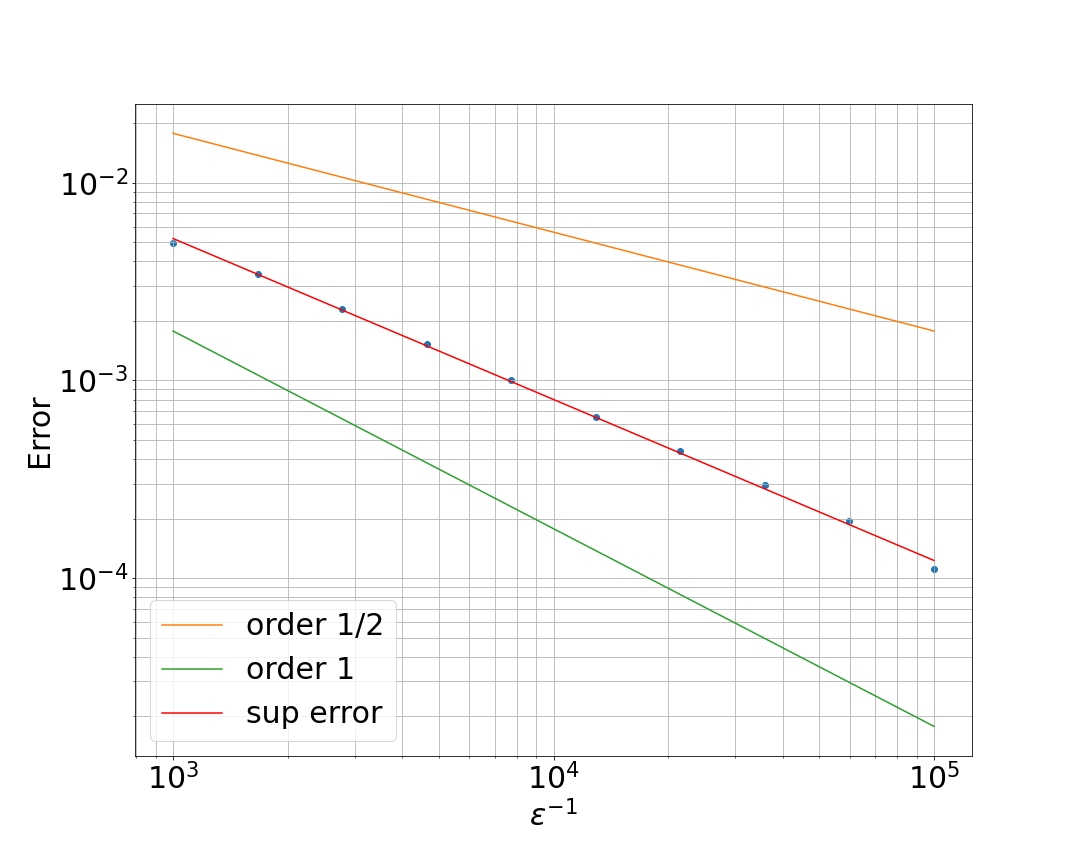}
			\caption{Rate of Convergence}
			
		\end{subfigure}
	}
	\caption{Initial density vanishing at zero with common noise}
	\label{fig: FIGURE IN THE COMMON NOISE SETTING}
\end{figure}

\newpage
\begin{appendices}
	
	\section{Technical lemmas}
	
	\begin{lemma}\label{app:lem: Y* IS  A CONTINUOUS LOCAL MARTINGALE}
		Suppose that $\tilde{P}^{\varepsilon_n} = \operatorname{Law}(\tilde{X}^{\varepsilon_n},\,\tilde{\mathcal{Y}}^{\varepsilon_n})$ converges weakly in $\mathcal{P}(D_\R \times \mathcal{C}_\R)$ to $\tilde{P}^*$, where $\tilde{X}^{\varepsilon_n}$ and $\tilde{\mathcal{Y}}^{\varepsilon_n}$ are the extensions of ${X}^{\varepsilon_n}$ and ${\mathcal{Y}}^{\varepsilon_n}$, respectively. Let $X^*$ and $Y^*$ be the canonical processes on $D_\R \times \mathcal{C}_\R$ such that for $(\eta,\,\omega) \in D_\R \times \mathcal{C}_\R$, $X^*(\eta,\,\omega) = \eta$ and  $Y^*(\eta,\,\omega) = \omega$. Then, under $\tilde{P}^*$, $Y^*$ is a martingale with respect to the filtration generated by $(X^*,\,Y^*)$ with quadratic variation
		\begin{equation*}
			\left \langle Y^* \right \rangle_t =
			\begin{cases}
				\begin{aligned}
					&0 &\qquad &t \in [-1,\,0), \\
					&\int_0^t \sigma(s,\,X^*_s)^2 \diff s &\qquad &t \in [0,\,T], \\
					& \int_0^T \sigma(s,\,X_s^*)^2 \diff s + (t - T) &\qquad &t \in (T,\,\Bar{T}]. 
				\end{aligned}
			\end{cases}
		\end{equation*}
	\end{lemma}
	
	\begin{proof}
		Set $\tilde{P}^*$ to be the limit point of $(\tilde{P}^{\varepsilon_n})_{n \ge 0}$ and 
		\begin{equation*}
			\mathbb{T}^{\tilde{P}^*} \coloneqq \left\{t \in [-1,\,\Bar{T}]\; : \; \tilde{P}^*(\eta_t = \eta_{t-}) = 1 \right\}.
		\end{equation*}
		Now for any $s_0,\,t_0 \in \mathbb{T}^{\tilde{P}^*}$ with $s_0 < t_0$ and $\{s_i\}_{i= 1}^k \subset [-1,\,s_0] \cap  \mathbb{T}^{\tilde{P}^*}$, we define the function
		\begin{equation*}
			F \, : \, D_\R \times \mathcal{C}_\R \to \R,\quad (\eta,\omega) \mapsto (\omega_{t_0} - \omega_{s_0})\prod_{i = 1}^k f_i(\eta_{s_i},\omega_{s_i}),
		\end{equation*}
		for arbitrary $f_i \in \mathcal{C}_b(D_\R\times\mathcal{C}_\R)$. To show that $Y^*$ is a martingale, it is sufficient to show that $\E^{\tilde{P}^*}[F(X^*,\,Y^*)] = 0$.
		
		As $\tilde{P}^{\varepsilon_n} \implies \tilde{P}^*$, then by Skorokhod's Representation Theorem, see \citep[Theorem~7.6]{billingsley2013convergence}, there exist $((x^n,\,y^n))_{n \ge 1}$ and $(x,\,y)$ defined on the same background space such that $(x^n,\,y^n)$ converges to $(x,\,y)$ almost surely in $(D_\R,\, M_1) \times (\mathcal{C}_\R,\,\norm{\cdot}_{\infty})$ with $\operatorname{Law}(x^n,\,y^n) = \tilde{P}^{\varepsilon_n}$ and $\operatorname{Law}(x,y) = \tilde{P}^*$. Now for any $p > 1$,
		\begin{equation*}
			\E\left[\nnnorm{F(x^n,\,y^n)}^p\right] \le C \E \left[\sup_{s \le \Bar{T}}\nnorm{\tilde{\mathcal{Y}}_s^{\varepsilon_n}}^p\right] \le C,
		\end{equation*}
		where $C$ is a constant that changes from line to line and depends only on $p,\,\sigma,\,T$ and the $f_i$'s but is independent of $\varepsilon$. Therefore, $F(x^n,\,y^n)$ is $L^p$-bounded uniformly in $n$. For $t \in \{t_0,\,s_0,\,s_1,\, \ldots,\,s_k\}$, $t$ is an almost sure continuity point of $x$. Therefore, by the properties of $M_1$-convergence, $(x^n_t,\,y^n_t)$ converges to $(x_t,\,y_t)$ almost surely. Hence, we have almost sure convergence of $F(x^n,\,y^n)$ to $F(x,\,y)$. Vitali's Convergence Theorem states that almost sure convergence and uniform integrability imply convergence of means, hence
		\begin{equation*}
			\E^{\tilde{P}^*}[F(X^*,\,Y^*)] = \E[F(x,\,y)] = \lim_{n \to \infty} \E[F(x^n,\,y^n)] = \lim_{n \to \infty} \E[F(\tilde{X}^{\varepsilon_n},\,\tilde{\mathcal{Y}}^{\varepsilon_n})] = 0,
		\end{equation*}
		where the last equality follows from the fact that $\E[F(\tilde{X}^{\varepsilon_n},\,\tilde{\mathcal{Y}}^{\varepsilon_n})] = 0$ for all $n$ as $\tilde{\mathcal{Y}}^{\varepsilon_n}$ is a martingale. Therefore, by a monotone class theorem argument, $Y^*$ is a continuous local martingale.
		
		Recall $x^n \to x$ almost surely in $(D_\R,\,M_1)$, hence we have pointwise convergence at the continuity points of $x$, see \citep[Theorem~12.5.1]{whitt2002stochastic}. As $\sigma$ is in $\mathcal{C}^{1,\,2}$ by Assumption \ref{ass: MODIFIED AND SIMPLIED FROM SOJMARK SPDE PAPER ASSUMPTIONS II}, there exists a set of full probability such that $\sigma(s,x^n_s) \to \sigma(s,\,x_s)$ for a set of $s$'s that have full Lebesgue measure in $[0,\,T]$. Furthermore, as $\sigma$ is bounded, by the Bounded Convergence Theorem
		\begin{equation}\label{app:eqn: FIRST EQUATION IN app:lem: Y* IS  A CONTINUOUS LOCAL MARTINGALE}
			\int_0^t \sigma(s,\,x^n_s)^2 \diff s \to \int_0^t \sigma(s,\,x_s)^2 \diff s 
		\end{equation}
		almost surely for any $t \in [0,\,T]$. Set 
		\begin{equation*}
			\left \langle Y \right \rangle_t =
			\begin{cases}
				\begin{aligned}
					&0 &\qquad &t \in [-1,\,0), \\
					&\int_0^t \sigma(s,\,X_s)^2 \diff s &\qquad &t \in [0,\,T], \\
					& \int_0^T \sigma(s,\,X_s)^2 \diff s + (t - T) &\qquad &t \in (T,\,\Bar{T}],
				\end{aligned}
			\end{cases}
		\end{equation*}
		where $Y$ is any one of $Y^*,\,y^n$ or $\tilde{\mathcal{Y}}^{\varepsilon_n}$ and $X$ is the respective $X^*,\,x^n$ or $\tilde{X}^{\varepsilon_n}$. Employing \eqref{app:eqn: FIRST EQUATION IN app:lem: Y* IS  A CONTINUOUS LOCAL MARTINGALE}, 
		\begin{equation*}
			F(x^n,\, (y^n)^2 - \langle y^n\rangle) \to F(x,\, y^2 - \langle y\rangle) \qquad \text{almost surely.}
		\end{equation*}
		Also, by the above and the boundedness of $\sigma$ by Assumption \ref{ass: MODIFIED AND SIMPLIED FROM SOJMARK SPDE PAPER ASSUMPTIONS II}, $F(x^n,\, (y^n)^2 - \langle y^n\rangle)$ is $L^p$-bounded uniformly in $n$. Hence, by Vitali's Convergence Theorem,
		\begin{align*}
			\E^{\tilde{P}^*}\left[F(X^*,\, (Y^*)^2 - \langle Y^*\rangle)\right] &= \E\left[F(x,\, y^2 - \langle y\rangle)\right] \\
			&= \lim_{n \to \infty}\E\left[F(x^n,\, (y^n)^2 - \langle y^n\rangle)\right] \\
			&= \lim_{n \to \infty}\E\left[F(\tilde{X}^{\varepsilon_n},\, (\tilde{\mathcal{Y}}^{\varepsilon_n})^2 - \langle \tilde{\mathcal{Y}}^{\varepsilon_n}\rangle)\right] = 0,
		\end{align*}
		where the last equality follows from the fact that $(\tilde{\mathcal{Y}}^{\varepsilon_n})^2 - \langle \tilde{\mathcal{Y}}^{\varepsilon_n} \rangle$ is a true martingale due to the boundedness of $\sigma$. This completes the proof.
	\end{proof}
	
	\begin{lemma}\label{app:lem: GENERAL STRONG CROSSSING PROPERTY RESULT}
		Consider the process $Z_t = M_t + tX$ for $t \in [-1,\,\Bar{T}]$, where $M_t$ is a continuous local martingale with $c_M (t - s) \le \langle M \rangle_t - \langle M \rangle_s \le C_M(t-s)$ for any $0 \le s < t$ almost surely, and $X$ is a non-negative random variable such that $\E[X] < \infty$. Then, for any stopping time $\tau$ where $\tau \ge 0$ almost surely, then 
		\begin{equation*}
			\prob\left[\inf_{s \in (\tau,(\tau + h) \wedge \Bar{T})}\{Z_s - Z_\tau\} \ge 0,\, \tau < \Bar{T}\right]= 0,
		\end{equation*}
		for any $h > 0$.
	\end{lemma}
	
	\begin{proof}
		In the case when $M$ is simply a Brownian motion, the result readily follows from the Strong Markov Property and the standard properties of Brownian motion. As $M$ is a continuous local martingale, we may view it as a (random) time-changed Brownian motion. We exploit this fact to show the claim. To begin, fix a $\Delta \in (0,h)$, $\lambda > 0$ and set $\Bar{\tau} \coloneqq \tau \wedge (\Bar{T} - \Delta)$. Then, conditioning on the event $E \coloneqq \{\tau \le \Bar{T} - \Delta\} \cap \{X \le \lambda\}$ and its complement, we have
		\begin{equation}\label{app:eq: FIRST EQUATION IN app:lem: GENERAL STRONG CROSSSING PROPERTY RESULT}
			\begin{split}
				\prob\left[\underset{s \in (\tau,(\tau + h) \wedge \Bar{T})}{\inf}\{Z_s - Z_\tau\} > -\Delta,\, \tau < \Bar{T}\right]  &\le \prob\left[\underset{s \in (\Bar{\tau},\Bar{\tau} + \Delta)}{\inf}\{M_s - M_{\Bar{\tau}} + (s - \Bar{\tau})\lambda\} > -\Delta\right]\\
				&\qquad + \prob\left[E^\complement,\,\tau < \Bar{T}\right].
			\end{split}
		\end{equation}
		Focussing on the first term, we observe
		\begin{equation*}
			\prob\left[\underset{s \in (\Bar{\tau},\Bar{\tau} + \delta)}{\inf}\{M_s - M_{\Bar{\tau}} + (s - \Bar{\tau})\lambda\} > -\Delta\right] \le \prob\left[\underset{s \in (\Bar{\tau},\Bar{\tau} + \delta)}{\inf}\{M_s - M_{\Bar{\tau}}\} > -\Delta(1 + \lambda)\right]
		\end{equation*}
		By the Dubins-Schwarz Theorem, $M$ is a time-changed Brownian motion. Therefore there exists a Brownian Motion $B$ such that
		\begin{equation*}
			\prob\left[\underset{s \in (\Bar{\tau},\Bar{\tau} + \delta)}{\inf}\{M_s - M_{\Bar{\tau}}\} > -\Delta(1 + \lambda)\right] = \prob\left[\underset{s \in (\Bar{\tau},\Bar{\tau} + \delta)}{\inf}\{B_{\langle M \rangle_s - \langle M \rangle_{\Bar{\tau}}}\} > - \Delta(1 + \lambda) \right]
		\end{equation*}
		Now as $\Bar{\tau} = \tau \wedge (\Bar{T} - \Delta) > 0$ almost surely, $\langle M \rangle_s - \langle M \rangle_{\Bar{\tau}} \ge c_M(s - \Bar{\tau})$ for any $s > \Bar{\tau}$ almost surely. So
		\begin{equation*}
			\prob\left[\underset{s \in (\Bar{\tau},\Bar{\tau} + \delta)}{\inf}\{B_{\langle M \rangle_s - \langle M \rangle_{\Bar{\tau}}}\} > - \Delta(1 + \lambda) \right] \le \prob\left[\underset{s \in (0,\Delta)}{\inf}\{B_{c_Ms}\} > - \Delta(1 + \lambda) \right]
		\end{equation*}
		By the reflection principle of Brownian motion, we have
		{\small\begin{align*}
				\prob\left[\inf_{s \le \Delta} B_{c_M s} \le - \Delta(1 + \lambda) \right] 
				&= 2\prob\left[B_{\gamma c_M} \le  -  \Delta(1 + \lambda)\right] \\
				&= 2 (2\pi)^{-1/2}\int_{-\infty}^{-\Delta^{1/2}c_M^{-1/2}(1 + \lambda)} e^{\frac{-y^2}{2}} \diff y 
				\ge 1 - 2\Delta^{1/2}(1 + \lambda)(2\pi c_M)^{-1/2}.
		\end{align*}}
		In conclusion, we have shown
		{\begin{align*}
				\prob\left[\underset{s \in (\Bar{\tau},\Bar{\tau} + \Delta)}{\inf}\{M_s - M_{\Bar{\tau}} + (s - \Bar{\tau})\lambda\} > -\Delta\right] &\le \prob\left[\underset{s \in (0,\Delta)}{\inf}\{B_{c_Ms}\} > - \Delta(1 + \lambda) \right]\\ &\le 2\Delta^{1/2}(1 + \lambda)(2\pi c_M)^{-1/2}
		\end{align*}}
		Setting $\lambda = \Delta^{-1/4}$, then by continuity of measure and the above, the expression in \eqref{app:eq: FIRST EQUATION IN app:lem: GENERAL STRONG CROSSSING PROPERTY RESULT} converges to 0 as we send $\Delta$ to zero. This completes the proof.
	\end{proof}
	
	\begin{lemma}[Convergence of Stopping Times]\label{app:lem: CONVERGENCE OF THE STOPPING TIME FOR PARTICLES THAT EXHIBIT THE CROSSING PROPERTY}
		Consider a sequence of functions $(z^n)_{n \ge 1}$ in $\DR$ converging towards some $z \in \DR$ with respect to the $M_1$-topology. We assume that $z$ has the following crossing property:
		\begin{equation}\label{app:eq: STRONG CROSSING PROPERTY}
			\forall \, h > 0\;\tau_0(z)  < \Bar{T} \qquad \implies \qquad \inf_{s \in (\tau_0(z),(\tau_0(z) + h)\wedge \Bar{T})}\left\{z_s - z_{\tau_0(z)}\right\} < 0,
		\end{equation}
		where $\tau_0$ is defined as in \eqref{eq: FIRST HITTING TIME OF ZERO OF THE CANONICAL PROCESS ON THE SPACE OF CADLAG FUNCTIONS} and $\Delta z_t \le 0$ for all $t \in [-1,\Bar{T}]$. Then we have
		\begin{equation*}
			\lim_{n \to \infty} \tau_0(z^n) = \tau_0(z).
		\end{equation*}
	\end{lemma}
	
	\begin{proof}
		The proof is composed of two steps. We shall show that $\limsup_{n \to \infty} \tau_0(z^n) \le \tau_0(z) \le \liminf_{n \to \infty} \tau_0(z^n)$. Hence, we will have equality and the claim follows.
		
		\noindent \underline{Step 1: $\limsup_{n \to \infty} \tau_0(z^n) \le \tau_0(z)$}\\[1 ex]
		We define the set of continuity points of $z$ to be $\mathbb{T}^z \coloneqq \{t \in [-1,\Bar{T}] \, : \, z_t = z_{t-}\}$. We remark that $\mathbb{T}^z$ is co-countable by \citep[Section~13]{billingsley2013convergence}. As $\tau_0(z) < \bar{T}$, by \eqref{app:eq: STRONG CROSSING PROPERTY} for any fixed $m \in \N$ there exists a $t \in (\tau_0(z),(\tau_0(z) + 1/m) \wedge \Bar{T}) \cap \mathbb{T}^z$ such that $z_t < 0$. Now, as $t$ is a continuity point of $z$, by \citep[Theorem~12.5.1]{whitt2002stochastic}, we have that $z^n_t \to z_t$ in $\R$ as $n \to \infty$. Therefore, for large $n$, $z_t^n < 0$ hence
		\begin{equation*}
			\limsup_{n \to \infty} \tau_0(z^n)  \le t \le \tau_0(z) + \frac{1}{m}.
		\end{equation*}
		As $m \in \N$ was arbitrary, the claim follows.
		
		\noindent \underline{Step 2: $\liminf_{n \to \infty} \tau_0(z^n) \ge \tau_0(z)$}\\[1ex]
		As $z^n \to z$ in the $M_1$-topology, we may find a sequence of parametric representations $((u^n,r^n))_{n \ge 1}$ of $(z^n)_{n \ge 1}$ which converges uniformly to a parametric representation $(u,r)$ of $z$, see \citep[Theorem~12.5.1]{whitt2002stochastic}. Therefore, we may find an $s^n \in [0,1]$ such that $(u^n_{s^n},r^n_{s^n}) = (z^n_{\tau_0(z^n)},\tau_0(z^n))$. By Step 1, since $\tau_0(z) < \Bar{T}$, we have 
		\begin{equation*}
			\liminf_{n \to \infty} \tau_0(z^n) \le \limsup_{n \to \infty} \tau_0(z^n) \le \tau_0(z) < \bar{T}.
		\end{equation*}
		Therefore, by the finiteness of $\liminf_{n \to \infty} \tau_0(z^n)$ and compactness of $[0,1]$, we may find a subsequence $n_k$ such that $\tau_0(z^{n_k}) \to \liminf_{n \to \infty} \tau_0(z^n)$ and $s_{n_k} \to s$ for some $s \in [0,1]$. By the uniform convergence of the parametric representations 
		\begin{align*}
			\liminf_{k \to \infty} z^{n_k}_{\tau_0(z^{n_k})} &= \liminf_{k \to \infty} u^{n_k}_{s^{n_k}} = u_s,\\
			\liminf_{k \to \infty} \tau_0(z^{n_k}) &= \liminf_{k \to \infty} r^{n_k}_{s^{n_k}} = r_s.
		\end{align*}
		As $r_s = \liminf_{n \to \infty} \tau_0(z^n)$, we may find $\gamma \in [0,1]$ such that $u_s = \gamma z_{(\liminf_{n \to \infty} \tau_0(z^n))-} + (1 - \gamma)z_{\liminf_{n \to \infty} \tau_0(z^n)}$. We also note $u_s \le 0$ as $\liminf_{k \to \infty} z^{n_k}_{\tau_0(z^{n_k})} \le 0$. Lastly, as $\Delta z_t \le 0$ for all $t$, we have $z_{\liminf_{n \to \infty} \tau_0(z^n)} \le 0$. Therefore, $\tau_0(z) \le {\liminf_{n \to \infty} \tau_0(z^n)}$. This completes the proof. 
	\end{proof}
	
	\begin{lemma}[Functional Continuity II]\label{app: lem: FUNCTIONAL CONTINUITY II}
		Let $\mu \in \mathcal{P}(\DR)$ be any measure such that
		\begin{equation}\label{app:eq: CROSSING PROPERTY EQUATION IN app: lem: FUNCTIONAL CONTINUITY II}
			\mu\left(\underset{s \in (\tau_{0}(\eta),\, (\tau_{0}(\eta) + h)\wedge \Bar{T})}{\inf}\{\eta_s - \eta_{\tau_0(\eta)}\} \ge 0,\,\tau_0(\eta) < \bar{T}\right) = 0,
		\end{equation}
		for any $h > 0$. Then, for any sequence of measures $(\mu^n)_{n \ge 1}$ such that $\mu^n \implies \mu$ in $(\mathcal{P}(\DR), \tmwk)$, we have
		\begin{equation*}
			\nu^{\mu^n}_t \coloneqq \mu^n(\eta_t \in \cdot,\, \tau_0(\eta) > t) \implies \nu^{\mu}_t \coloneqq \mu(\eta_t \in \cdot,\, \tau_0(\eta) > t), 
		\end{equation*}
		in $\MRR$, the space of sub-probability measures on $\R$ endowed with the topology of weak convergence, for and $t \ge 0$ such that $\mu(\eta_t = \eta_{t-}) = 1$ and $\mu(\tnot = t) = 0$.
	\end{lemma}
	
	\begin{proof}
		The proof is an application of the Continuous Mapping Theorem, \citep[Theorem~2.7]{billingsley2013convergence}. We only need to construct $\mu$-almost sure continuous maps.
		
		\noindent \underline{Step 1: Projection of measures from $\DR$ to $\R\times\{0,\,1\}$}\\[1 ex]
		Consider the map
		\begin{equation}\label{app:eq: FIRST CONTINUOUS MAP IN app: lem: FUNCTIONAL CONTINUITY II PROOF}
			(X_t,\,\ind_{\{\tau_{0}(\cdot)\}})\,:\, \DR \to \R\times\{0,1\},\qquad \eta \mapsto (\eta_t,\,\ind_{\{\tnot > t\}})
		\end{equation}
		\eqref{app:eq: FIRST CONTINUOUS MAP IN app: lem: FUNCTIONAL CONTINUITY II PROOF} is a $\mu$-almost sure continuous map. Choose a $\eta \in \DR$ such that $\eta_t = \eta_{t-},\,\tnot \neq t$ and is in the complement of the event in \eqref{app:eq: CROSSING PROPERTY EQUATION IN app: lem: FUNCTIONAL CONTINUITY II}. Such $\eta$'shave full measure under $\mu$. $M_1$-convergence implies pointwise convergence at continuity points, \citep[Theorem~12.5.1]{whitt2002stochastic}, therefore $X_t$ is $M_1$-continuous for every such $\eta$. Also by Lemma \ref{lem:sec2: LEMMA STATING THAT THE CONDITION FEEDBACK IS A CONTINUOUS MAP HENCE WE HAVE WEAK CONVERGENCE OF FEEDBACK ESSENTIALLY}, since \eqref{app:eq: CROSSING PROPERTY EQUATION IN app: lem: FUNCTIONAL CONTINUITY II} holds, $\tau_0$ is an $M_1$-continuous map at $\eta$. As $\tnot \neq t$, $\ind_{\{\tau_{0}(\cdot)\}}$ is $M_1$-continuous at $\eta$. Hence, $(X_t,\,\ind_{\{\tau_{0}(\cdot)\}})$ is a $\mu$-almost sure continuous map. By the Continuous Mapping Theorem,
		\begin{equation*}
			(X_t,\,\ind_{\{\tau_{0}(\cdot)\}})^\#\mu^n \implies (X_t,\,\ind_{\{\tau_{0}(\cdot)\}})^\#\mu \quad \text{in} \quad \mathcal{P}(\DR \times \{0,\,1\}).
		\end{equation*}
		
		\noindent \underline{Step 2: Weak convergence of sub-probability measures}\\[1 ex]
		For any $f \in \mathcal{C}_b(\R)$, define the map
		\begin{equation*}
			\hat{f}\,:\, \R \times \{0,\,1\} \to \R,\qquad (x,y) \mapsto f(x)\ind_{\{1\}}(y). 
		\end{equation*}
		It is clear $\hat{f} \in \mathcal{C}_b(\R\times\{0,\,1\})$. By step 1,
		\begin{equation*}
			\langle(X_t,\,\ind_{\{\tau_{0}(\cdot)\}})^\#\mu^n,\hat{f}\rangle \xrightarrow[]{} \langle(X_t,\,\ind_{\{\tau_{0}(\cdot)\}})^\#\mu,\hat{f}\rangle.
		\end{equation*}
		But by definition, $\langle(X_t,\,\ind_{\{\tau_{0}(\cdot)\}})^\#\mu^n,\hat{f}\rangle = \nu^{\mu^n}_t(f)$ and $\langle(X_t,\,\ind_{\{\tau_{0}(\cdot)\}})^\#\mu,\hat{f}\rangle = \nu^\mu_t(f)$. So $\nu^{\mu^n}_t(f) \xrightarrow[]{} \nu^{\mu}_t(f)$. The conclusion now follows by Portmanteau's Theorem.
	\end{proof}
	
	
	\begin{lemma}[Weak convergence of sub-probability measures]\label{app: lem: CONVERGENCE OF SUB PROBABILITY MEASURES LEMMA}
		Suppose that $\tilde{\mathbf{P}}^{\varepsilon_n} \implies \tilde{\mathbf{P}}^*$ on $(\mathcal{P}(\DR),\tmwk)$ for a positive sequence $(\varepsilon_n)_{n \ge 1}$ which converges to zero. Set $$\mathbb{T} \coloneqq \left\{t \in [-1,\Bar{T}]\;:\; \E\left[\tilde{\mathbf{P}}^*(\eta_t = \eta_{t-})\right] = 1,\,\E\left[\tilde{\mathbf{P}}^*(\tnot = t)\right] = 0\right\}.$$ Then for any $t \in \mathbb{T}$,
		\begin{equation*}
			\bm{\nu}_t^{\varepsilon_n} \implies \bm{\nu}_t^* \coloneqq \tilde{\mathbf{P}}^*(\eta_t \in \cdot,\, \tnot > t) \quad \text{in} \quad \MRR.
		\end{equation*}
	\end{lemma}
	
	
	\begin{proof}
		By definition of $\mathbb{T}$ and Lemma \ref{lem:sec2: LEMMA STATING THE STRONG CROSSING PROPERTY OF THE LIMITING RANDOM MEASURE P*}. for any $t \in \mathbb{T}$ there is a set of $\mu$'s of full $\operatorname{Law}(\mathbf{P}^*)$-measure such that
		\begin{equation}\label{app:eq: FIRST EQUATION STATING THE PROPERTIES OF MU IN app: lem: CONVERGENCE OF SUB PROBABILITY MEASURES LEMMA}
			\mu(\eta_t \neq \eta_{t-}) = \mu(\tnot = t) =0,\, \mu\left(\underset{s \in (\tau_{0}(\eta),\, (\tau_{0}(\eta) + h)\wedge \bar{T})}{\inf}\{\eta_s - \eta_{\tau_0(\eta)}\} \ge 0,\,\tau_0(\eta) < \bar{T}\right) = 0.
		\end{equation}
		
		As $\tilde{\mathbf{P}}^{\varepsilon_n} \implies \tilde{\mathbf{P}}^*$, by Skorokhod's Representation Theorem there exists a $(\mathbf{Q}^n)_{n \ge 1},\,\mathbf{Q}^*$ such that $\mathbf{Q}^n \xrightarrow[]{} \mathbf{Q}^*$ almost surely, $\operatorname{Law}(\mathbf{Q}^n) = \operatorname{Law}(\tilde{\mathbf{P}}^{\varepsilon_n})$, $\operatorname{Law}(\mathbf{Q}^*) = \operatorname{Law}(\tilde{\mathbf{P}}^{*})$ and $\mathbf{Q}^*$ satisfies \eqref{app:eq: FIRST EQUATION STATING THE PROPERTIES OF MU IN app: lem: CONVERGENCE OF SUB PROBABILITY MEASURES LEMMA} almost surely. Set
		\begin{equation*}
			\nu_t^{\mathbf{Q}^n} = \mathbf{Q}^n(\eta_t \in \cdot,\, \tnot > t),\quad \text{and}\quad \nu_t^{\mathbf{Q}^*} = \mathbf{Q}^*(\eta_t \in \cdot,\, \tnot > t)
		\end{equation*}
		By Lemma \ref{app: lem: FUNCTIONAL CONTINUITY II}, $\nu_t^{\mathbf{Q}^n} \xrightarrow[]{} \nu_t^{\mathbf{Q}^*}$ almost surely in $\MRR$. Now, for any $F \in \mathcal{C}_b(\MRR)$, by the Dominated Convergence Theorem
		\begin{equation*}
			\underset{n \xrightarrow[]{} \infty}{\lim} \E \left[F(\bm{\nu}_t^{\varepsilon_n})\right] =  \underset{n \xrightarrow[]{} \infty}{\lim} \E \left[F(\nu_t^{\mathbf{Q}^n} )\right] = \E \left[F(\nu_t^{\mathbf{Q}^*} )\right] = \E \left[F(\bm{\nu}_t^{*})\right].
		\end{equation*}
		The result now follows by Portmanteau's Theorem.
	\end{proof}
	
	
	\begin{lemma}[Functional Continuity III]\label{app: lem: FUNCTIONAL CONTINUITY III}
		Let $\mu \in \mathcal{P}(\DR)$ be any measure such that
		\begin{equation}\label{app:eq: CROSSING PROPERTY EQUATION IN app: lem: FUNCTIONAL CONTINUITY III}
			\mu\left(\underset{s \in (\tau_{0}(\eta),\, (\tau_{0}(\eta) + h)\wedge \bar{T})}{\inf}\{\eta_s - \eta_{\tau_0(\eta)}\} \ge 0,\,\tau_0(\eta) < \Bar{T}\right) = 0,
		\end{equation}
		for any $h > 0$ and let $g(t,x,\nu)$ be any function satisfying Assumption \ref{ass: MODIFIED AND SIMPLIED FROM SOJMARK SPDE PAPER ASSUMPTIONS II} \ref{ass: MODIFIED AND SIMPLIED FROM SOJMARK SPDE PAPER ASSUMPTIONS II ONE}. Then $\int_0^tg(s,\eta_s^n,\nu^{\mu^n}_s) \diff s$ converges to $\int_0^t g(s,\eta_s,\nu^{\mu}_s) \diff s$ for any $t \ge 0$ whenever $(\eta^n,\mu^n) \xrightarrow[]{} (\eta,\mu)$ in $(\mathcal{P}(\DR),\tmwk)\times (\DR,M_1)$ along a sequence for which $\sup_{n \ge 1} \langle\mu^n,\,\sup_{s \le \Bar{T}} \nnnorm{\tilde{\eta}_s}^p \rangle < \infty$ for some $p > 1$ and any $t \ge 0$. For any measure $m \in \mathcal{P}(\DR),\, \nu_s^m \coloneqq m(\tilde{\eta}_s\in \cdot,\, \tau_0(\tilde{\eta}) > s)$.
	\end{lemma}
	

	\begin{proof}
		By Assumption \ref{ass: MODIFIED AND SIMPLIED FROM SOJMARK SPDE PAPER ASSUMPTIONS II},
		\begin{equation}\label{app:eq: UPPER BOUND ON DRIFT IN  app: lem: FUNCTIONAL CONTINUITY III}
			\nnnorm{g(s,\eta_s^n,\nu_s^{\mu^n})} \le C (1 + \sup_{m \ge 1}\, \sup_{u \le \bar{T}} \nnnorm{\eta_u^m} + \sup_{m \ge 1} \langle\mu^m,\,\sup_{u \le \bar{T}} \nnnorm{\tilde{\eta}_u} \rangle)
		\end{equation}
		The right-hand side of \eqref{app:eq: UPPER BOUND ON DRIFT IN  app: lem: FUNCTIONAL CONTINUITY III} is finite because $\eta^n \to \eta$ in $(\DR,M_1)$ and by our assumption. Thus, it is sufficient to show that $g(s,\eta_s^n,\nu_s^{\mu^n})$ converges to $g(s,\eta_s,\nu_s^{\mu})$ on a set of full Lebesgue measure. The conclusion then follows from the Dominated Convergence Theorem.
		
		By Assumption \ref{ass: MODIFIED AND SIMPLIED FROM SOJMARK SPDE PAPER ASSUMPTIONS II},
		\begin{equation}\label{eq: DIFFERENCE BETWEEN ANY FUNCTION G SATISFYING THE LOCALLY D0 Lipschitzness ASSUMPTION}
			\nnnorm{g(s,\eta_s^n,\nu_s^{\mu^n}) - g(s,\eta_s,\nu_s^{\mu})} \le C\nnnorm{\eta_s^n - \eta_s} + C(1 + \nnnorm{\eta_s} + \langle \nu_s^{\mu^n},\,\nnnorm{\cdot}\rangle)d_0(\nu_s^{\mu^n},\,\nu_s^{\mu}).
		\end{equation}
		For any $s \in \mathbb{T}^\mu \coloneqq \left\{t \in [-1,\Bar{T}]\;:\; \mu(\eta_t = \eta_{t-})= 1,\,\mu(\tau_0 = t)= 0\right\}$, the first term converges to zero as $\eta^n \to \eta$ in $(\DR,M_1)$. For the second term in \eqref{eq: DIFFERENCE BETWEEN ANY FUNCTION G SATISFYING THE LOCALLY D0 Lipschitzness ASSUMPTION}, it suffices to show $d_0(\nu_s^{\mu^n},\,\nu_s^{\mu}) \to 0$, since $(1 + \nnnorm{\eta_s} + \langle \nu_s^{\mu^n},\,\nnnorm{\cdot}\rangle)$ is bounded uniformly in $n$ by our assumption.
		
		Recall that 
		\begin{equation}\label{eq: DEFINITION OF THE D0 METRIC IN THE PROOF OF THE FUNCTIONAL CONTINUITY III LEMMA IN THE FIRST PAPER}
			d_0(\nu_s^{\mu^n},\,\nu_s^{\mu}) = \sup\left\{\nnnorm{\langle\nu_s^{\mu^n} - \nu_s^{\mu},\, \psi\rangle}\,:\, \psi \in \mathcal{C}_{d_0}\right\},
		\end{equation}
		where $\mathcal{C}_{d_0} \coloneqq \{\psi \in \mathcal{C}(\R)\,:\, \norm{\psi}_{\operatorname{Lip}} \le 1,\, \nnnorm{\psi(0)} \le 1\}$. Fix a $\delta > 0$, then for any $\lambda > 1$, by the Arzerl\`a-Ascoli Theorem, there exists a finite family of functions $\psi_1,\ldots,\psi_m \in \mathcal{C}_{d_0}$ supported on $[-\lambda - 1, \lambda + 1]$, for $m = m(\lambda) \in \N$ such that for any $\psi \in \mathcal{C}_{d_0}$,
		\begin{equation}\label{eq: THE SUP NORM DISTANCE BETWEEN OUR FINITE APPROXIMATING SET IN THE FIRST PAPER PROOF}
			\sup_{x \in [-\lambda,\lambda]} \nnnorm{\psi(x) - \psi_i(x)} < \delta/2
		\end{equation}
		for some $i \in \{1,\ldots,m\}$. Fixing any $\psi \in \mathcal{C}_{d_0}$ and choosing a $\psi_i$ such that \eqref{eq: THE SUP NORM DISTANCE BETWEEN OUR FINITE APPROXIMATING SET IN THE FIRST PAPER PROOF} holds, we have
		\begin{align*}
			\nnnorm{\langle\nu_s^{\mu^n} - \nu_s^{\mu},\psi\rangle} 
			\le& \nnnorm{\int_0^{\lambda} (\psi - \psi_i) \diff(\nu_s^{\mu^n} - \nu_s^{\mu})} 
			+\nnnorm{\int_{\lambda}^{\infty} (\psi - \psi_i) \diff(\nu_s^{\mu^n} - \nu_s^{\mu})} \\
			&+ \nnnorm{\langle\nu_s^{\mu^n} - \nu_s^{\mu},\psi_i\rangle}.
		\end{align*}
		By our choice of $\psi_i$, the first term is bounded by $\delta$ uniformly in $n$. By the linear growth condition for functions in $\mathcal{C}_{d_0}$,
		\begin{equation*}
			\nnnorm{\int^{\infty}_{\lambda} (\psi - \psi_i) \diff(\nu_s^{\mu^n} - \nu_s^{\mu})} \le C \langle\nu_s^{\mu^n} + \nu_s^{\mu},\nnnorm{\cdot}\ind_{[\lambda,\infty)}\rangle
		\end{equation*}
		for a constant $C$ independent of $n$. Consequently, by the definition of $\nu_s^{\mu^n}$,
		\begin{align*}
			\langle\nu_s^{\mu^n},\nnnorm{\cdot}\ind_{[\lambda,\infty)}\rangle \le
			\int_{\DR} \nnnorm{\tilde{\eta}_s}\ind_{\{\tilde{\eta}_s > \lambda\}}\diff\mu^n(\tilde{\eta}) 
			&\le \langle\mu^n,\,\sup_{u \le \bar{T}} \nnnorm{\tilde{\eta}_u}^p \rangle^{\frac{1}{p}} \mu^n\bigg(\sup_{u \le \bar{T}} \nnnorm{\tilde{\eta}_u} > \lambda\bigg)^{\frac{p-1}{p}}\\
			&\le \frac{1}{\lambda^{p - 1}} \langle\mu^n,\,\sup_{u \le \bar{T}} \nnnorm{\tilde{\eta}_u}^p \rangle.
		\end{align*}
		The second inequality follows from H\"older's inequality, while the last inequality follows from Markov's inequality. We may also employ the same argument to upper bound the $\langle\nu_s^{\mu},\nnnorm{\cdot}\ind_{[\lambda,\infty)}\rangle$ term in the above. Hence, by these upper bounds and as $\mu^n \implies \mu$, we deduce that
		\begin{equation*}
			\nnnorm{\int^{\infty}_{\lambda} (\psi - \psi_i) \diff(\nu_s^{\mu^n} - \nu_s^{\mu})} \le
			\frac{C}{\lambda^{p - 1}} \sup_{m \ge 1} \langle\mu^m,\,\sup_{u \le \bar{T}} \nnnorm{\tilde{\eta}_u}^p \rangle,
		\end{equation*}
		where $C$ is independent of $n$. Returning to \eqref{eq: DEFINITION OF THE D0 METRIC IN THE PROOF OF THE FUNCTIONAL CONTINUITY III LEMMA IN THE FIRST PAPER}, we have shown that
		\begin{equation*}
			d_0(\nu_s^{\mu^n},\,\nu_s^{\mu}) \le \delta + \frac{C}{\lambda^{p - 1}} \sup_{m \ge 1} \langle\mu^m,\,\sup_{u \le \bar{T}} \nnnorm{\tilde{\eta}_u}^p \rangle + \sum_{i = 1}^m\nnorm{\langle\nu_s^{\mu^n} - \nu_s^{\mu},\psi_i\rangle} . 
		\end{equation*}
		The middle term will vanish uniformly in $n$ as $\lambda \to \infty$. Thus, we fix a $\lambda$ sufficiently large such that the middle term is bounded by $\delta$ uniformly in $n$. As $\lambda$ is now fixed, so is $m$. By Lemma \ref{app: lem: FUNCTIONAL CONTINUITY II}, $\nu_s^{\mu^n} \implies \nu_s^{\mu}$ in $\MRR$; hence, as the $\psi_i$'s have compact support, they are continuous and bounded. Therefore, the final term will be smaller than $\delta$ for all $n$ sufficiently large. Consequently, for all $n$ sufficiently large, $d_0(\nu_s^{\mu^n},\,\nu_s^{\mu}) \le C\delta$ for some constant $C$ independent of $n$. Hence, $d_0(\nu_s^{\mu^n},\,\nu_s^{\mu}) \to 0$ as $n \to \infty$. This completes the proof.
	\end{proof}
	\color{black}
	
	
	\begin{lemma}\label{app:lem: WEAK STOCHASTIC UPPER BOUND ON THE MEASURE OF THE DELAYED SYSTEM}
		Fix any $t < T$. There is a constant $C >0$ independent of $\varepsilon$ and $t$ such that for any $\gamma < 1 \wedge (T - t)$ we have
		{\small\begin{equation*}
				\prob\left[ \bm{\nu}_t^\varepsilon[0, \alpha_t z + C\gamma^{1/3} + \alpha_t(L_t^\varepsilon - \mathfrak{L}_t^\varepsilon) + (\alpha_{t + \gamma} - \alpha_t)]\ge z, \; \forall \; z \le L_{t + \gamma}^\varepsilon - L_t^\varepsilon - C\gamma^{1/3}\right]  \ge 1 - C\gamma^{1/3}
		\end{equation*}}
	\end{lemma}
	
	
	\begin{proof}
		To begin, fix a $\gamma > 0$ such that $\gamma <  1 \wedge (T - t)$ and fix a $z \in \R$. Then we denote by $E_1^z$ the event
		{\footnotesize
			\begin{equation*}
				\Big\{
				X_t^\varepsilon - \gamma C(1 + \underset{u \le t + \gamma}{\sup}\nnnorm{X_u^\varepsilon} + \E[\underset{u \le t + \gamma}{\sup}\left.\nnnorm{X_u^\varepsilon}\right|W^0]) - \underset{u \le \gamma}{\sup}\nnnorm{\mathcal{Y}^{\varepsilon}_{t + u} - \mathcal{Y}^{\varepsilon}_{t}}- \alpha_t z - \alpha_t(L_t^\varepsilon - \mathfrak{L}_t^\varepsilon) - (\alpha_{t + \gamma} - \alpha_t)\le 0,\, \tau^\varepsilon > t
				\Big\}
		\end{equation*}}
		where $C$ is the constant from the linear growth condition on $b$. Now fix $x \le L_{t + \gamma}^\varepsilon - L_t^\varepsilon$. By the continuity of the loss process, \citep[Theorem~2.4]{hambly2019spde}, there exists a $s \le \gamma$ such that $ x = L_{t + s}^\varepsilon - L_t^\varepsilon$. Employing the integration by parts formula, we observe for any $u \in [t, t+ s]$
		\begin{align*}
			\int_t^u \alpha_v \diff \mathfrak{L}_v^\varepsilon &= \alpha_u \mathfrak{L}_u^\varepsilon - \alpha_t \mathfrak{L}_t^\varepsilon - \int_t^u \alpha^\prime(v) \mathfrak{L}_v^\varepsilon \diff v \le \alpha_u \mathfrak{L}_u^\varepsilon - \alpha_t \mathfrak{L}_t^\varepsilon\\
			&\le \alpha_{t + s} {L}_{t + s}^\varepsilon - \alpha_t \mathfrak{L}_t^\varepsilon \le \alpha_{t} {L}_{t + s}^\varepsilon + (\alpha_{t + s} - \alpha_t) - \alpha_t \mathfrak{L}_t^\varepsilon, 
		\end{align*}
		where to establish upper bounds we use the fact that $\alpha$ is non-negative and non-decreasing, and $\mathfrak{L}_v^\varepsilon \le L_v^\varepsilon \le 1$ for any $v \ge 0$. Therefore for any $u \in [t, t+ s]$
		\begin{align*}
			X_u^\varepsilon &=
			X_t^\varepsilon + (X_u^\varepsilon - X_t^\varepsilon) =
			X_t^\varepsilon + \int_t^u b(v,X_v^\varepsilon,\bm{\nu}_v^\varepsilon) \diff v + \mathcal{Y}_u - \mathcal{Y}_t - \int_t^u \alpha_v \diff \mathfrak{L}_v^\varepsilon \\
			&\ge X_t^\varepsilon - \gamma C(1 + \underset{u \le t + \gamma}{\sup}\nnnorm{X_u^\varepsilon} + \E[\underset{u \le t + \gamma}{\sup}\left.\nnnorm{X_u^\varepsilon}\right|W^0]) - \underset{u \le \gamma}{\sup}\nnnorm{\mathcal{Y}^{\varepsilon}_{t + u} - \mathcal{Y}^{\varepsilon}_{t}}\\
			& \qquad - \{\alpha_{t} {L}_{t + s}^\varepsilon + (\alpha_{t + s} - \alpha_t) - \alpha_t \mathfrak{L}_t^\varepsilon \pm \alpha_t L_t^\varepsilon\}\\
			&\ge X_t^\varepsilon - \gamma C(1 + \underset{u \le t + \gamma}{\sup}\nnnorm{X_u^\varepsilon} + \E[\underset{u \le t + \gamma}{\sup}\left.\nnnorm{X_u^\varepsilon}\right|W^0]) - \underset{u \le \gamma}{\sup}\nnnorm{\mathcal{Y}^{\varepsilon}_{t + u} - \mathcal{Y}^{\varepsilon}_{t}}\\
			& \qquad - \alpha_{t}x - (\alpha_{t + s} - \alpha_t) - \alpha_t(L_t^\varepsilon -  \mathfrak{L}_t^\varepsilon)
		\end{align*}
		Therefore as $L^\varepsilon$ is $W^0$ measurable and conditioning on $W^0$ fixes $L^\varepsilon$, we have
		\begin{equation*}
			\prob\left[\left. E_1^x\right|W^0\right] \ge \prob\left[\left.\underset{t \le u \le t + s}{\inf}X_u^\varepsilon \le 0, \tau^\varepsilon > t\right|W^0\right] = L_{t + s}^\varepsilon - L_t^\varepsilon = x. 
		\end{equation*}
		Now, for any fixed $z \le L_{t + \gamma}^\varepsilon - L_t^\varepsilon - 2 \gamma^{1/3}$, set $z_0 = z + 2 \gamma^{1/3}$. We define the event
		\begin{equation*}
			E_2 \coloneqq \Big\{ \gamma C(1 + \underset{u \le t + \gamma}{\sup}\nnnorm{X_u^\varepsilon} + \E[\underset{u \le t + \gamma}{\sup}\left.\nnnorm{X_u^\varepsilon}\right|W^0]) + \underset{u \le \gamma}{\sup}\nnnorm{\mathcal{Y}^{\varepsilon}_{t + u} - \mathcal{Y}^{\varepsilon}_{t}} \ge \gamma^{1/3} \Big\}
		\end{equation*}
		Then on the event $E_1^{z_0} \cap E_2^\complement$
		{\begin{align*}
				X_t^\varepsilon - \alpha_t z_0 &\le 
				\gamma C(1 + \underset{u \le t + \gamma}{\sup}\nnnorm{X_u^\varepsilon} + \E[\underset{u \le t + \gamma}{\sup}\left.\nnnorm{X_u^\varepsilon}\right|W^0]) + \underset{u \le \gamma}{\sup}\nnnorm{\mathcal{Y}^{\varepsilon}_{t + u} - \mathcal{Y}^{\varepsilon}_{t}}\\
				&\qquad+ \alpha_t(L_t^\varepsilon - \mathfrak{L}_t^\varepsilon) + (\alpha_{t + \gamma} - \alpha_t) \\
				&\le \gamma^{1/3} + \alpha_t(L_t^\varepsilon - \mathfrak{L}_t^\varepsilon) + (\alpha_{t + \gamma} - \alpha_t).
		\end{align*}}
		Therefore on the same event
		\begin{equation*}
			X_t^\varepsilon - \alpha_t z =  X_t^\varepsilon - \alpha_t z_0 + 2 \alpha_t \gamma^{1/3} \le (1 + 2 \alpha_t)\gamma^{1/3} + \alpha_t(L_t^\varepsilon - \mathfrak{L}_t^\varepsilon) + (\alpha_{t + \gamma} - \alpha_t)
		\end{equation*}
		Consequently, we deduce
		\begin{align*}
			\bm{\nu}_t^\varepsilon[0, \alpha_t z + (1 + 2 \alpha_t)\gamma^{1/3} + \alpha_t(L_t^\varepsilon - \mathfrak{L}_t^\varepsilon) + (\alpha_{t + \gamma} - \alpha_t)] 
			&\ge \prob\left[\left.E_1^{z_0} \cap E_2^\complement\right|W^0\right]\\
			&\ge \prob\left[\left.E_1^{z_0}\right|W^0\right] - \prob\left[\left. E_2\right|W^0\right]\\
			&\ge z_0- \prob\left[\left.E_2\right|W^0\right].
		\end{align*}
		Therefore if we have control over the mass $\prob\left[\left.E_2\right|W^0\right]$, we may estimate the mass with respect to $\bm{\nu}_t$ that is near the boundary. Therefore, defining the event $E_3 \coloneqq \{\prob\left[\left.E_2\right|W^0\right] \le \gamma^{1/3}\}$ we deduce on $E_3$
		\begin{equation*}
			\bm{\nu}_t^\varepsilon[0, \alpha_t z + (1 + 2 \alpha_t)\gamma^{1/3} + \alpha_t(L_t^\varepsilon - \mathfrak{L}_t^\varepsilon) + (\alpha_{t + \gamma} - \alpha_t)] \ge z_0 - \gamma^{1/3} \ge z.
		\end{equation*}
		The last inequality follows from the fact that $z_0 = z + 2 \gamma^{1/3}$. Now we only need to find a $C$ independent of $\varepsilon,\, \gamma$ and $t$ such that $\prob[E_3^\complement] \le C \gamma^{1/3}$. By application of Markov's inequality twice
		\begin{align*}
			\prob\left[E_3^\complement\right]
			&\le \gamma^{-1/3} \prob\left[E_2\right]\\
			&\le \gamma^{-1/3} \prob \left[\gamma C(1 + \underset{u \le t + \gamma}{\sup}\nnnorm{X_u^\varepsilon} + \E[\underset{u \le t + \gamma}{\sup}\left.\nnnorm{X_u^\varepsilon}\right|W^0]) \ge 2^{-1}\gamma^{1/3}/2\right]\\
			&\qquad+ \gamma^{-1/3 }\prob \left[\underset{u \le \gamma}{\sup}\nnnorm{\mathcal{Y}^{\varepsilon}_{t + u} - \mathcal{Y}^{\varepsilon}_{t}} \ge 2^{-1}\gamma^{1/3}/2\right]\\
			&\le 2C\gamma^{1/3} \E \left[1 + \underset{u \le t + \gamma}{\sup}\nnnorm{X_u^\varepsilon} + \E[\underset{u \le t + \gamma}{\sup}\left.\nnnorm{X_u^\varepsilon}\right|W^0]\right]
			+ 2^8\gamma^{-3 }\E \left[\underset{u \le \gamma}{\sup}\nnnorm{\mathcal{Y}^{\varepsilon}_{t + u} - \mathcal{Y}^{\varepsilon}_{t}}^8\right]\\
			&\le c_1 (\gamma^{1/3} + \gamma) \le 2 c_1 \gamma^{1/3},
		\end{align*}
		where $c_1$ depends on the constant from Proposition \ref{prop: GRONWALL TYPE UPPERBOUND ON THE SUP OF THE MCKEAN VLASOV EQUATION WITH CONVOLUTION IN THE BANKING MODEL SIMPLIFIED AND INDEXED GENERALISED}, the constant from Burkholder-Davis-Gundy to bound the second term and the uniform bounds on $\sigma$, but is notably independent of $\varepsilon$. Therefore, setting $C = \max\{1 + 2 \alpha(T),c_1\}$ completes the proof.
	\end{proof}
	
	
	\begin{lemma}\label{app:lem: WEAK STOCHASTIC UPPER BOUND ON THE MEASURE OF THE SINGULAR SYSTEM AT CONTINUITY POINTS}
		Suppose that $\tilde{\mathbf{P}}^{\varepsilon_n} \implies \tilde{\mathbf{P}}^*$ on $(\mathcal{P}(\DR),\tmwk)$ for a positive sequence $(\varepsilon_n)_{n \ge 1}$ which converges to zero. Set $$\mathbb{T} \coloneqq \left\{t \in [-1,\Bar{T}]\;:\; \E\left[\tilde{\mathbf{P}}^*(\eta_t = \eta_{t-})\right] = 1,\,\E\left[\tilde{\mathbf{P}}^*(\tnot = t)\right] = 0\right\}.$$ Then for any $t \in \mathbb{T}\cap [0,T)$ and $\gamma > 0$ such that $t + \gamma \in \mathbb{T} \cap [0,T)$ we have
		\begin{equation*}
			\prob\left[ \bm{\nu}_t[0, \alpha(t)z + C\gamma^{1/3} + \alpha(t + \gamma) - \alpha(t)]\ge z, \; \forall \; z \le L_{t + \gamma} - L_t - C\gamma^{1/3}\right]  \ge 1 - C\gamma^{1/3}
		\end{equation*}
	\end{lemma}
	
	
	\begin{proof}
		As $\tilde{\mathbf{P}}^{\varepsilon_n} \implies \tilde{\mathbf{P}}^*$, by employing Skorokhod's Representation Theorem, there exists a $(\bm{\mu}^n)_{n \ge 1}$ and $\bm{\mu}$ such that $\operatorname{Law}(\bm{\mu}^n) = \operatorname{Law}(\tilde{\mathbf{P}}^{\varepsilon_n})$, $\operatorname{Law}(\bm{\mu}) = \operatorname{Law}(\tilde{\mathbf{P}}^{*})$ and $\bm{\mu}^n \xrightarrow[]{} \bm{\mu}$ almost surely in  $(\mathcal{P}(D_\R),\mathfrak{T}_{M_1}^{\text{wk}})$. As $\operatorname{Law}(\mathbf{P}^*)$-almost every measure $\mu$ satisfy \eqref{app:eq: CROSSING PROPERTY EQUATION IN app: lem: FUNCTIONAL CONTINUITY II} by Lemma \ref{lem:sec2: LEMMA STATING THE STRONG CROSSING PROPERTY OF THE LIMITING RANDOM MEASURE P*}, then by Lemma \ref{app: lem: FUNCTIONAL CONTINUITY II} $\nu_t^{\bm{\mu}^n} \xrightarrow[]{} \nu_t^{\bm{\mu}}$ almost surely for any $t \in \mathbb{T}$. Furthermore for any $t \in \mathbb{T}\cap [0,T)$, by Lemma \ref{lem:sec2: LEMMA STATING THAT THE CONDITION FEEDBACK IS A CONTINUOUS MAP HENCE WE HAVE WEAK CONVERGENCE OF FEEDBACK ESSENTIALLY} and Corollary \ref{cor:sec2: COROLLARY SHOWING THAT WE HAVE CONVERGENCE OF THE DELAYED / SMOOTHENED LOSS TO THE LIMITING LOSS}, we have
		{\small\begin{equation*}
				\bm{\mu}^n(\tnot \le t) \xrightarrow[]{} \bm{\mu}(\tnot \le t) \quad \textnormal{and} \quad \int_0^t \kappa^{\varepsilon_n}(t-s)\bm{\mu}^n(\tnot \le s) \diff s  \xrightarrow[]{} \int_0^t \kappa^{\varepsilon_n}(t-s)\bm{\mu}(\tnot \le s) \diff s
		\end{equation*}}
		almost surely. Therefore for simplicity and notational convenience for the remainder of this proof, we may suppose
		\begin{align*}
			&\bm{\nu}_t^{\varepsilon_n} \xrightarrow[]{} \bm{\nu}_t^* &\quad &\textnormal{a.s. in } \MRR,\\ 
			&L_t^{\varepsilon_n} \xrightarrow[]{} L_t^* &\quad &\textnormal{a.s. in } \R,\\ 
			&\mathfrak{L}_t^{\varepsilon_n} \xrightarrow[]{} {L}_t^* &\quad &\textnormal{a.s. in } \R.
		\end{align*}
		Recall by Lemma \ref{app:lem: WEAK STOCHASTIC UPPER BOUND ON THE MEASURE OF THE DELAYED SYSTEM},
		{\small\begin{equation*}
				\prob\left[ \bm{\nu}_t^\varepsilon[0, \alpha_t z + C\gamma^{1/3} + \alpha_t(L_t^\varepsilon - \mathfrak{L}_t^\varepsilon) + (\alpha_{t + \gamma} - \alpha_t)]\ge z, \; \forall \; z \le L_{t + \gamma}^\varepsilon - L_t^\varepsilon - C\gamma^{1/3}\right]  \ge 1 - C\gamma^{1/3},
		\end{equation*}}
		for any $\varepsilon > 0$. It is well known that the Levy-Prokhorov metric, $d_L$, metricizes weak convergence, \citep[Theorem~1.11]{prokhorov1956convergence}. Fixing $\delta_1,\,\delta_2,\,\delta_3,\,\delta_4 > 0$, we define the event
		\begin{equation*}
			A_n \coloneqq \left\{\nnnorm{L_t^{\varepsilon_n} - L_t^*} < \delta_1,\nnnorm{L_{t+\gamma}^{\varepsilon_n} - L_{t+ \gamma}^*} < \delta_2, \nnnorm{\mathfrak{L}_t^{\varepsilon_n} - L_t^{\varepsilon_n}} < \delta_3, d_L(\bm{\nu}_t^{\varepsilon_n},\bm{\nu}_t^*) < \delta_4\right\}
		\end{equation*}
		Therefore
		{\footnotesize
			\begin{align*}
				&1 - C\gamma^{1/3} 
				\le\prob\left[ \bm{\nu}_t^{\varepsilon_n}[0, \alpha_t z + C\gamma^{1/3} + \alpha_t(L_t^{\varepsilon_n} - \mathfrak{L}_t^{\varepsilon_n}) + (\alpha_{t + \gamma} - \alpha_t)]\ge z, \; \forall \; z \le L_{t + \gamma}^{\varepsilon_n} - L_t^{\varepsilon_n} - C\gamma^{1/3}\right]\\
				&\qquad \qquad \qquad+ \prob[A_n^\complement]\\
				&\quad\le\prob\left[ \delta_4 + \bm{\nu}_t^*(-\delta_4, \alpha_t z + C\gamma^{1/3} + \alpha_t \gamma_3 + (\alpha_{t + \gamma} - \alpha_t) + \delta_4)\ge z, \; \forall \; z \le L_{t + \gamma}^* - L_t^* - C\gamma^{1/3}- \delta_1 - \delta_2\right]\\
				&\quad+ \prob[A_n^\complement]
		\end{align*}}
		Sending $\varepsilon_n \xrightarrow[]{} 0$, then $\prob[A_n^\complement] \xrightarrow[]{} 0$ by the Dominated Convergence Theorem as we have almost sure convergence. Lasts by sending $\delta_1 \xrightarrow[]{} 0 ,\, \delta_2 \xrightarrow[]{} 0 ,\,\delta_3 \xrightarrow[]{} 0 ,\,\delta_4 \xrightarrow[]{} 0$ one at a time and in order, then by employing continuity of measure we may conclude
		\begin{equation*}
			\prob\left[ \bm{\nu}_t^*[0,\alpha_t z + C\gamma^{1/3} + (\alpha_{t + \gamma} - \alpha_t)]\ge z, \; \forall \; z \le L_{t + \gamma}^* - L_t^* - C\gamma^{1/3}\right] \ge 1 - C\gamma^{1/3}
		\end{equation*}
	\end{proof}
	
	
	\begin{lemma}\label{app:lem: APPENDIX LEMMA STATING THE UPPER BOUND ON THE JUMPS OF THE LIMITING PROCESS}
		Suppose that $\tilde{\mathbf{P}}^{\varepsilon_n} \implies \tilde{\mathbf{P}}^*$ on $(\mathcal{P}(\DR),\tmwk)$ for a positive sequence $(\varepsilon_n)_{n \ge 1}$ which converges to zero. Then we have
		\begin{equation}\label{app:eq: EQUATION SHOWS THE UPPER BOUND ON THE LIMITING LOSS IN app:lem: APPENDIX LEMMA STATING THE UPPER BOUND ON THE JUMPS OF THE LIMITING PROCESS}
			L_t^* \le \inf \left\{x \ge 0\,;\, \bm{\nu}_{t-}^*([0,\,\alpha(t)x]) < x \right\}
		\end{equation}
		almost surely for any $t \in [0,\,T)$.
	\end{lemma}
	
	
	\begin{proof}
		It is clear that Lemma \ref{app:lem: APPENDIX LEMMA STATING THE UPPER BOUND ON THE JUMPS OF THE LIMITING PROCESS} holds for any $t \in \mathbb{T}$. Hence we must only show the upper bound for $t \not \in \mathbb{T}$. We first consider the case when $t \in (0,T) \cap \mathbb{T}^\complement$. The case when $t = 0$ will be treated separately. Now as $\mathbb{T}$ is dense in $[0,T]$, we may find a $(t_n)_{n \ge 1},\, (t_n + \gamma_n)_{n \ge 1} \subset \mathbb{T}$ such that $t_n \uparrow t$, $t_n + \gamma_n \downarrow t$ and $\gamma_n < 2^{-3n}$. Now by the Borel-Cantelli Lemma, we have a set of full measure such that
		\begin{equation}\label{app:eq: FIRST EQUATION IN app:lem: APPENDIX LEMMA STATING THE UPPER BOUND ON THE JUMPS OF THE LIMITING PROCESS}
			\bm{\nu}_{t_n}^*[0,\alpha_{t_n} z + C\gamma_n^{1/3} + (\alpha_{{t_n} + \gamma_n} - \alpha_{t_n})]\ge z, \; \forall \; z \le L_{{t_n} + \gamma_n} - L_{t_n} - C\gamma_n^{1/3},
		\end{equation}
		for all (possibly stochastic) $n$ large. Furthermore, by the dominated convergence theorem, we have
		\begin{equation}\label{app:eq: SECOND EQUATION IN app:lem: APPENDIX LEMMA STATING THE UPPER BOUND ON THE JUMPS OF THE LIMITING PROCESS}
			\bm{\nu}_{t_n}^* \xrightarrow[]{} \bm{\nu}_{t-}^*,
		\end{equation}
		$\operatorname{Law}(\tilde{\mathbf{P}}^*)$-almost surely in $\MRR$ as for any $\phi \in \mathcal{C}_b(\R)$
		\begin{equation*}
			\lim_{n \xrightarrow[]{} \infty} \bm{\nu}_{t_n}^*(\phi) = \lim_{n \xrightarrow[]{} \infty} \int_{\DR} \phi(\eta_{t_n})\ind_{\{\tnot > t_n\}} \diff \tilde{\mathbf{P}}^*(\eta) =  \int_{\DR} \phi(\eta_{t-})\ind_{\{\tnot \ge t\}} \diff \tilde{\mathbf{P}}^*(\eta).
		\end{equation*}
		So on an event of full $\operatorname{Law}(\mathbf{P}^*)$ measure where \eqref{app:eq: FIRST EQUATION IN app:lem: APPENDIX LEMMA STATING THE UPPER BOUND ON THE JUMPS OF THE LIMITING PROCESS} and \eqref{app:eq: SECOND EQUATION IN app:lem: APPENDIX LEMMA STATING THE UPPER BOUND ON THE JUMPS OF THE LIMITING PROCESS} holds, by Portmanteau Theorem for any $\gamma > 0$, $z < \Delta L_t$
		\begin{equation*}
			\bm{\nu}_{t-}^*[0,\,\alpha(t)z + \gamma] \ge \limsup_{n \xrightarrow[]{} \infty} \bm{\nu}_{t_n}^*[0,\,\alpha(t)z + \gamma] \ge \limsup_{n \xrightarrow[]{} \infty} \bm{\nu}_{t_n}^*[0,\alpha_{t_n} z + C\gamma_n^{1/3} + (\alpha_{{t_n} + \gamma_n} - \alpha_{t_n})]\ge z.
		\end{equation*}
		This holds as $z < \Delta L_t^*$ and $ L_{{t_n} + \gamma_n}^* - L_{t_n}^* - C\gamma_n^{1/3} \xrightarrow[]{} \Delta L_t$. Sending $\gamma$ to zero shows the claim for every $t > 0$. 
		
		In the case when $t = 0$, we have by Lemma \ref{app:lem: WEAK STOCHASTIC UPPER BOUND ON THE MEASURE OF THE DELAYED SYSTEM}
		\begin{equation*}
			\prob\left[ \bm{\nu}_{0-}^\varepsilon[0, \alpha_0 z + C\gamma^{1/3} + (\alpha_{\gamma} - \alpha_0)]\ge z, \; \forall \; z \le L_{ \gamma}^\varepsilon - C\gamma^{1/3}\right]  \ge 1 - C\gamma^{1/3}.
		\end{equation*}
		As $\bm{\nu}^{\varepsilon}_{0-} = \nu_{0-}$, where $\nu_{0-}$ is a deterministic measure, and $\bm{\nu}_{0-}^*$ is almost surely distributed as $\nu_{0-}$ almost surely, then we have by Lemma \ref{app:lem: WEAK STOCHASTIC UPPER BOUND ON THE MEASURE OF THE SINGULAR SYSTEM AT CONTINUITY POINTS}
		\begin{equation*}
			\prob\left[ \bm{\nu}_{0-}^*[0, \alpha(0)z + C\gamma^{1/3} + \alpha(\gamma) - \alpha(0)]\ge z, \; \forall \; z \le L_{\gamma}^* - C\gamma^{1/3}\right]  \ge 1 - C\gamma^{1/3}.
		\end{equation*}
		for $\gamma \in \mathbb{T}$. Now the rest of the proof follows similar arguments as above by choosing $\gamma_n \in \mathbb{T}$ such that $\gamma_n \downarrow 0$.
	\end{proof}
	
	\begin{proposition}[Gr\"onwall Type Inequality I]\label{prop: GENERALISED GRONWALL TYPE INEQUALITY}
		Suppose $a,\,\tilde{\alpha},\, \tilde{\beta} \in \R_+$ such that $a \ge 0,\,0< \tilde{\beta}<1,\, \tilde{\alpha} > 0$. Suppose $g(t)$ is a nonnegative, nondecreasing continuous function defined on $0 \le t < T$, $g(t) \le M$ (constant), and suppose $u(t)$ is nonnegative and bounded on $0 \le t < T$ with 
		\begin{equation*}
			u(t) \le a + g(t)\int_0^t (t-s)^{\tilde{\beta}-1}s^{\tilde{\alpha} - 1 }u(s) \diff{} s
		\end{equation*}
		on this interval. Then
		\begin{equation*}
			u(t) \le a\left[1 + \sum_{n \ge 1 } g_t^n C_n t^{n(\tilde{\alpha} + \tilde{\beta} - 1)}\right],\quad 0 \le t< T,
		\end{equation*}
		where
		$C_0\coloneqq 1, 
		C_{n+1} \coloneqq B\left((n+1)\tilde{\alpha} + n\tilde{\beta} - n,\tilde{\beta} \right)C_n,$
		with $B(\tilde{\alpha},\,\tilde{\beta}) \coloneqq \int_0^1(1 - \tilde{s})^{\tilde{\beta}  - 1}\tilde{s}^{\tilde{\alpha}-1} \diff{} \tilde{s}.$
	\end{proposition}
	\begin{proof}
		Let $B\phi_t = g_t\int_0^t(t-s)^{\tilde{\beta} - 1} s^{\tilde{\alpha} - 1}\phi_s \diff{} s$, $t \ge 0$, for localling integrable functions $\phi$. Then $u_t \le a(t) + Bu_t$ implies
		\begin{equation*}
			u_t \le \sum_{k = 0}^{n-1}B^ka + B^nu_t.
		\end{equation*}
		Let us prove that 
		\begin{equation}\label{eq: LINEAR GRONWALL EQUATION MY VERSION}
			B^n(1)_t \le _nt^{n(\tilde{\alpha} + \tilde{\beta} - 1)}g_t^n
		\end{equation}
		and $B^nu_t \rightarrow 0$ as $n \rightarrow +\infty$ for each t in $0 \le t < T$.
		\\[1 ex]
		\noindent \underline{Step 1: $B^n(1)_t \le C_nt^{n(\tilde{\alpha} + \tilde{\beta} - 1)}g_t^n$.}
		For $n = 1$,
		\begin{align*}
			B(1)_t &= g_t \int_0^t(t-s)^{\tilde{\beta} -1}s^{\tilde{\alpha}-1} \diff{}s,\quad \text{set } \tilde{s} = s/t\\
			&= g_t \int_0^1 t^{\tilde{\beta} -1}(1 - \tilde{s})^{\tilde{\beta} - 1}t^{\tilde{\alpha} -1}\tilde{s}^{\tilde{\alpha}-1} t \diff{}\tilde{s}\\
			&= g_t t^{\tilde{\alpha} + \tilde{\beta} - 1} B(\tilde{\alpha},\,\tilde{\beta}).
		\end{align*}
		Now, suppose the claim is true for $n = k$, then for $n = k + 1 $
		\begin{align*}
			B^{k+1}(1)_t &= g_t\int_0^t(t-s)^{\tilde{\beta} - 1}s^{\tilde{\alpha} -1}B^{k}(1)_s \diff{}s,\\
			&\le g_t \int_0^t(t-s)^{\tilde{\beta} - 1}s^{\tilde{\alpha} -1}s^{k(\tilde{\alpha} + \tilde{\beta} -1)}g_s^kC_k \diff{} s, \quad \text{by above}\\
			&\le C_k g_{t}^{k+1} \int_0^t(t-s)^{\tilde{\beta} - 1}s^{\tilde{\alpha} -1}s^{k(\tilde{\alpha} + \tilde{\beta} -1)} \diff{} s, \quad \text{set } \tilde{s} = s/t\\
			&= C_k g_{t}^{k+1} \int_0^1 t^{\tilde{\beta} - 1}(1-\tilde{s})^{\tilde{\beta} - 1}t^{\tilde{\alpha} -1}\tilde{s}^{\tilde{\alpha} - 1}t^{k(\tilde{\alpha} + \tilde{\beta} -1)}\tilde{s}^{k(\tilde{\alpha} + \tilde{\beta} -1)} t\diff{} \tilde{s},\\
			&= C_{k+1}g^{k+1}_t t^{(k+1)(\tilde{\alpha} + \tilde{\beta} - 1)}.
		\end{align*}
		Hence the claim is true by the principle of induction.
		\\[1 ex]
		\noindent \underline{Step 2:}
		We observe that $B$ is monotone, that is if $\phi_1 \le \phi_2\, \forall t \in [0,\,T]$, then by the nonnegativity of $g$ we have $B(\phi_1) \le B(\phi_2)$. Also by the linearity of integration, we see also that $B$ is a linear operator. Therefore,
		\begin{equation*}
			B(u)_t = g(t)\int_0^t (t-s)^{\tilde{\beta}-1}s^{\tilde{\alpha} - 1 }u(s) \diff{} s \le \norm{u}_{L^\infty} g(t)\int_0^t (t-s)^{\tilde{\beta}-1}s^{\tilde{\alpha} - 1 } \diff{} s = \norm{u}_{L^\infty} B(1)_t
		\end{equation*}
		Therefore, by linearity, monotonicity and step 1
		\begin{equation*}
			B^n(u)_t \le \norm{u}_{L^\infty} B^n(1)_t \le \norm{u}_{L^\infty} C_n g_t^nt^{n(\tilde{\alpha} + \tilde{\beta} - 1)}
		\end{equation*}
		\\[1 ex]
		\noindent \underline{Step 3: Summability of $C_n$.}
		By Gautschi's inequality, \cite{GautschiInequality}, we have that for all $x > 0 $ and $s \in (0,\,1)$
		\begin{equation*}
			x^{1-s} < \frac{\Gamma(x+1)}{\Gamma(x+s)} < (x + 1)^{1-s}
		\end{equation*}
		Therefore,
		\begin{align*}
			\frac{C_{n+1}}{C_n} &= B\left((n+1)\tilde{\alpha} + n\tilde{\beta} - n,\,\tilde{\beta}\right),\\
			&= \frac{\Gamma\left((n+1)\tilde{\alpha} + n\tilde{\beta} - n\right)\Gamma(\tilde{\beta})}{\Gamma\left((n+1)\tilde{\alpha} + (n+1)\tilde{\beta} - n\right)},\\
			&= \Gamma(\tilde{\beta})\left[\frac{\Gamma\left((n+1)\tilde{\alpha} + (n+1)\tilde{\beta} - (n+1) + 1- \tilde{\beta}\right)}{\Gamma\left((n+1)\tilde{\alpha} + (n+1)\tilde{\beta} - (n+1) + 1\right)}\right],\\
			&\le \Gamma(\tilde{\beta})(n+1)^{-\tilde{\beta}}(\tilde{\alpha} + \tilde{\beta} -1)^{-\tilde{\beta}} \quad \text{ by Gautschi's Inequality}.\\
		\end{align*}
		Hence $C_{n+1}/C_n \rightarrow 0$ and $n \rightarrow +\infty$. Hence by the ratio test, we have that $C_n$ is summable.
		\\[1 ex]
		\noindent \underline{Step 4: Summability of $B^n(u)_t$.}
		By step $3$, we have that
		\begin{equation*}
			\frac{C_{n+1}\norm{u}_{L^\infty}t^{(n+1)(\tilde{\alpha} + \tilde{\beta} - 1)}g_t^{n+1}}{C_{n}\norm{u}_{L^\infty}t^{n(\tilde{\alpha} + \tilde{\beta} - 1)}g_t^{n}} = \norm{u}_{L^\infty}t^{\tilde{\alpha} + \tilde{\beta} - 1}g_t\frac{C_{n+1}}{C_n} \xrightarrow[]{n \rightarrow +\infty} 0
		\end{equation*}
		Therefore by the ratio test then the comparison test, we have that $B^n(u)_t$ is summable. Hence $B^n(u)_t \rightarrow 0$ as $n \rightarrow +\infty$.
		\\[1 ex]
		\noindent \underline{Step 5:}
		As $u_t \le a + B(u)_t$, then it is clear by the Principle of Induction, by using the monotonicity and linearity of $B$, we have $u_t \le \sum_{j = 1}^{N-1}aB^j(1)_t + B^N(u)_t $. Hence taking limiting as $N \xrightarrow[]{} \infty$, by step 1 and step 4, we conclude $u_t \le \sum_{j \ge 0} aC_jt^{j(\tilde{\alpha} + \tilde{\beta} - 1)}g_t^j$. The proof is now complete.
	\end{proof}
	
	
	\begin{proof}[Proof of Proposition \ref{prop: MAIN RESULT PROPOSITION IN THE CONTINUOUS CASE}]
		This proof is analogous to that of Proposition \ref{prop: MAIN RESULT PROPOSITION IN THE GENERAL CASE}. Most of the details have been skipped for brevity. Choose $t_0 \in (0 ,\, t_{\mathrm{explode}})$.
		
		\noindent \underline{Step 1: Regularity of $L$ and decomposition into integral form.}
		As $L \in \mathcal{C}^1([0,\,t_{\mathrm{explode}}))$, by the Fundamental Theorem of Calculus we have $L_t - L_s \le K(t - s)$ for any $t,s \in [0,\,t_{\mathrm{explode}})$ with $t > s$ and $K = \sup_{u \le t_0} \nnorm{L^\prime_u}$. Now we may write $L$ as
		\begin{equation*}
			L_t = \int_0^t\kappa^\varepsilon(t-s)L_s\diff{s} + \left[1 - \int_0^t\kappa^\varepsilon(t-s)\diff{s}\right]L_t + \int_0^t\kappa^\varepsilon(t-s)(L_t-L_s)\diff{s}.
		\end{equation*}
		Observe
		\begin{equation*}
			\left[1 - \int_0^t\kappa^\varepsilon(t-s)\diff{s}\right]L_t \le {2K\varepsilon} \qquad \textnormal{and} \qquad  \int_0^t\kappa^\varepsilon(t-s)(L_t-L_s)\diff{s} \le {K\varepsilon}.
		\end{equation*}
		Therefore
		\begin{equation*}
			L_t = \int_0^t\kappa^\varepsilon(t-s)L_s\diff{s} + \Psi^\varepsilon(t) \qquad \text{ where } \nnorm{\Psi^\varepsilon(t)} \le {3K\varepsilon} \, \quad \forall \,t \in [0,\,t_0].
		\end{equation*}
		
		\noindent \underline{Step 2: Comparison between the delayed loss and the instantaneous loss.}
		As in Proposition \ref{prop: MAIN RESULT PROPOSITION IN THE GENERAL CASE}, we have 
		\begin{equation*}
			0 \le \nnnorm{L_t - L_t^\varepsilon} \le K c_1 \int_0^t\int_0^u \frac{\kappa^\varepsilon(u - s)\nnnorm{L_s - L_s^\varepsilon}}{\sqrt{ t- u}} \diff s \diff u- + K c_1 \int_0^t \frac{\nnnorm{\Psi^\varepsilon(s)}}{\sqrt{t - s}}.
		\end{equation*}
		Note as $\nnnorm{\Psi^\varepsilon(t)} \le 3K\varepsilon$ for all $t \in [0,\,t_0]$, we see that the second term above is bounded above by  $C_{K,t_0} \varepsilon$. Therefore,
		\begin{align}
			0 \le \nnorm{L_t - {L}^\varepsilon_t} &\le \; K c_1 \int_0^t\int_s^t \frac{\kappa^\varepsilon(u - s)\nnnorm{L_s - L_s^\varepsilon}}{ \sqrt{ t- u}} \diff u \diff s + C_{K,t_0} \varepsilon\nonumber \\
			&= \; K c_1 \int_0^t \nnorm{L_s - {L}_s^\varepsilon} \rho^\varepsilon(t,s) \diff s + C_{K,t_0} \varepsilon, \label{eq: UPPERBOUND ON THE DIFFERENCES OF THE LOSSES IN THE CONTINUOUS CASE}
		\end{align}
		where
		\begin{equation*}
			\rho^\varepsilon(t,s) = \int_s^t  \frac{\kappa^{\varepsilon}(u-s)}{\sqrt{t-u}} \diff u.
		\end{equation*}
		\noindent \underline{Step 3: Bounds on $ \rho^\varepsilon(t,s)$.}
		As in Proposition \ref{prop: MAIN RESULT PROPOSITION IN THE GENERAL CASE}, the presence of $\kappa$ in $\rho^\varepsilon$ makes the function too general to do any analysis, hence we shall construct polynomial bounds on $\rho^\varepsilon$. Then we can apply generalised versions of Gr\"onwall's Lemma. Recall $t \ge s$, 
		
		\noindent \underline{Case 1:} ${t-s} \le \varepsilon$\\[1 ex]
		\begin{align*}
			\rho^\varepsilon(t,s) &= \int_s^t  \frac{\kappa^{\varepsilon}(u-s)}{\sqrt{t-u}} \diff u \qquad \qquad \textrm{let $\tilde{u} = \frac{u-s}{\varepsilon}$}\\
			&= \int_0^{\frac{t-s}{\varepsilon}} \frac{\kappa(\tilde{u})}{\sqrt{t-s - \varepsilon\tilde{u}}} \diff \tilde{u} \le \frac{\norm{\kappa}_{L^\infty}}{\varepsilon^{1/2}} \int_0^{\frac{t-s}{\varepsilon}} \frac{\diff \tilde{u}}{\sqrt{\frac{t-s}{\varepsilon} - \tilde{u}}}\\
			&= \frac{2\norm{\kappa}_{L^\infty}(t- s)^{1/2}}{\varepsilon} \le \frac{2\norm{\kappa}_{L^\infty}}{(t- s)^{1/2}}
		\end{align*}
		\noindent \underline{Case 2:} ${t-s} > \varepsilon$\\[1 ex]
		As the support of $\kappa^\varepsilon$ is in $[0,\,\varepsilon]$
		\begin{align*}
			\rho^\varepsilon(t,s) &= \int_s^t  \frac{\kappa^{\varepsilon}(u-s)}{\sqrt{t-u}} \diff u = \int^{s + \varepsilon}_s  \frac{\kappa^{\varepsilon}(u-s)}{\sqrt{t-u}} \diff u \\
			&\le \frac{\norm{\kappa}_{L^\infty}}{\varepsilon} \int^{s + \varepsilon}_s  \frac{\diff u}{\sqrt{t-u}} = {2\norm{\kappa}_{L^\infty}}\left[\frac{(t-s)^{1/2} - (t - s - \varepsilon)^{1/2}}{\varepsilon}\right].
		\end{align*}
		\\[1 ex]
		\noindent \underline{Step 4: Gr\"onwall type argument.}
		Now that we have sufficiently simplified $\rho^\varepsilon$, we may put \eqref{eq: UPPERBOUND ON THE DIFFERENCES OF THE LOSSES IN THE CONTINUOUS CASE} into a form where we may apply Gr\"onwall's inequality. By step 4 case 1 and \eqref{eq: UPPERBOUND ON THE DIFFERENCES OF THE LOSSES IN THE CONTINUOUS CASE}, we have for $t \le \varepsilon$
		\begin{align*}
			\nnorm{L_t - {L}^\varepsilon_t} \le K c_1 \int_0^t 2\norm{\kappa}_{L^\infty} (t - s)^{-1/2} \nnorm{L_s - {L}_s^\varepsilon} \rho^\varepsilon(t,s) \diff s +  C_{K,t_0} \varepsilon.
		\end{align*}
		By step 4 case 2 and \eqref{eq: UPPERBOUND ON THE DIFFERENCES OF THE LOSSES IN THE CONTINUOUS CASE}, we have for $t > \varepsilon $
		\begin{align*}
			\nnorm{L_t - {L}^\varepsilon_t} &\le K c_1 \int_0^{t - \varepsilon} \nnorm{L_s - {L}_s^\varepsilon} \rho^\varepsilon(t,s) \diff s + K c_1 \int_{t - \varepsilon}^t \nnorm{L_s - {L}_s^\varepsilon} \rho^\varepsilon(t,s) \diff s+C_{K,t_0} \varepsilon\\
			&\le 2 K c_1 \norm{\kappa}_{L^\infty} \int_0^{t - \varepsilon} \left[\frac{(t-s)^{1/2} - (t - s - \varepsilon)^{1/2}}{\varepsilon}\right] \nnorm{L_s - {L}_s^\varepsilon} \diff s \\
			&\phantom{\le}+ 2 K c_1 \norm{\kappa}_{L^\infty} \int_{t - \varepsilon}^t (t -s)^{-1/2}\nnorm{L_s - {L}_s^\varepsilon} \diff s + C_{K,t_0} \varepsilon\\
			&\le 2 K c_1 \norm{\kappa}_{L^\infty} \sum_{j \ge 2 } \tilde{C}_j \varepsilon^{j-1} \int_0^{t - \varepsilon} (t-s)^{\frac{-2j+1}{2}}\nnorm{L_s - {L}_s^\varepsilon} \diff s \\
			&\phantom{\le}+ 2 K c_1 \norm{\kappa}_{L^\infty} \int_{0}^t (t -s)^{-1/2}\nnorm{L_s - {L}_s^\varepsilon} \diff s+ C_{K,t_0} \varepsilon,
		\end{align*}
		where the second term in the last line captures the higher order terms from Taylor's Theorem. The Monotone Convergence Theorem allows us to swap integrals and sums. We now consider
		\begin{equation*}
			\tilde{C}_j \varepsilon^{j-1} \int_0^{t - \varepsilon} (t-s)^{\frac{-2j+1}{2}} \diff s.
		\end{equation*}
		We shall proceed in two cases. 
		\\[1 ex]
		\noindent \underline{Case 1:} $\varepsilon < t \le 2\varepsilon$\\[1 ex]
		\begin{equation*}
			\tilde{C}_j \varepsilon^{j-1} \int_0^{t - \varepsilon} (t-s)^{\frac{-2j+1}{2}} \diff s \le \tilde{C}_j \varepsilon^{j-1} \varepsilon^{\frac{-2j+1}{2}} \int_0^{t - \varepsilon}  \diff s 
			= {\tilde{C}_j \varepsilon^{\frac{1}{2}}}
		\end{equation*}
		where the first inequality follows from the fact that $(-2j + 1)/2 < 0$ as $j \ge 2$ and $t - s \in [\varepsilon,\,t]$ for $s \in [0,\,t-\varepsilon]$. 
		\\[1 ex]
		\noindent \underline{Case 2:} $ t >  2\varepsilon$\\[1 ex]
		We observe that
		\begin{equation*}
			\tilde{C}_j \varepsilon^{j-1} \int_0^{\varepsilon} (t-s)^{\frac{-2j+1}{2}} \diff s 
			\le \tilde{C}_j \varepsilon^{j-1} \varepsilon^{\frac{-2j+1}{2}} \int_0^{\varepsilon} \diff s 
			= {\tilde{C}_j \varepsilon^{\frac{1}{2}}}
		\end{equation*}
		and
		\begin{align*}
			\tilde{C}_j \varepsilon^{j-1} \int_\varepsilon^{t - \varepsilon} (t-s)^{\frac{-2j+1}{2}} \diff s &= \tilde{C}_j \varepsilon^{j-1} \frac{2}{2j-3}\left.(t - s)^{\frac{-2j+3}{2}}\right|_{s = \varepsilon}^{t - \varepsilon}\\ 
			&= \frac{2\tilde{C}_j \varepsilon^{j-1}}{2j-3} \left[\varepsilon^{\frac{-2j+3}{2}} - (t  - \varepsilon)^{\frac{-2j+3}{2}}\right]\\
			&\le \frac{2\tilde{C}_j \varepsilon^{j-1}\varepsilon^{\frac{-2j+3}{2}}}{2j - 3}\\
			&= \frac{2\tilde{C}_j \varepsilon^{1/2}}{2j - 3} 
			\le {2\tilde{C}_j \varepsilon^{1/2}}, 
		\end{align*}
		where the last inequality follows from the fact that $j \ge 2$. 
		Therefore, we have
		\begin{align*}
			\tilde{C}_j \varepsilon^{j-1} \int_0^{t - \varepsilon} (t-s)^{\frac{-2j+1}{2}} \diff s 
			&= \tilde{C}_j \varepsilon^{j-1} \int_0^{\varepsilon} (t-s)^{\frac{-2j+1}{2}} \diff s + \tilde{C}_j \varepsilon^{j-1} \int_\varepsilon^{t - \varepsilon} (t-s)^{\frac{-2j+1}{2}} \diff s\\
			&\le {3\tilde{C}_j \varepsilon^{1/2}},
		\end{align*}
		for all $t > \varepsilon$. As $L$ and ${L}^\varepsilon$ are bounded by $1$, so independent of $t$ being greater or less than $\varepsilon$ we have
		\begin{align*}
			\nnorm{L_t - {L}^\varepsilon_t} &\le  2 K c_1 \norm{\kappa}_{L^\infty} \int_{0}^t (t -s)^{-1/2}\nnorm{L_s - {L}_s^\varepsilon} \diff s+ {12 K c_1 \norm{\kappa}_{L^\infty}\varepsilon^{1/2} \sum_{j \ge 2} \tilde{C}_j} + C_{K,t_0} \varepsilon \\
			&= 2 K c_1 \norm{\kappa}_{L^\infty} \int_{0}^t (t -s)^{-1/2}s^{-\gamma}\nnorm{L_s - {L}_s^\varepsilon} \diff s+ C_{K,t_0}\varepsilon^{1/2},
		\end{align*}
		for any $t \in [0,\,t_0]$. Lastly by Proposition \ref{prop: GENERALISED GRONWALL TYPE INEQUALITY} using $\tilde{\beta} = 1/2$ and $\tilde{\alpha} = 1$, then $\tilde{\alpha} + \tilde{\beta}  - 1 > 0$ as $\gamma < 1/2$ and
		\begin{align*}
			\nnorm{L_t - {L}^\varepsilon_t} &=  C_{K,t_0}\varepsilon^{1/2}\sum_{n \ge 0} (2 K c_1 t_0^{1/2} \norm{\kappa}_{L^\infty})^n C_n.
		\end{align*}
		This completes the proof.
	\end{proof}

	\section{Further numerical analysis}\label{app:sec: FURTHER ANALYSIS AND DISCUSSION ON NUMERICS}
	
	In Section \ref{sec: SECTION ON THE NUMERICAL SCHEMES}, we considered six examples to compare the theoretical rate of convergence with that obtained numerically. The parameters used for each simulation are given in \cref{tab: PARAMETERS OF THE NUMERICAL SIMULATIONS}.

	\begin{table}[H]
		\centering
		\scriptsize
		\begin{tabular}{|c|c|c|c|c|c|c|c|c|c|}
			\hline
			Simulation & CC1\tablefootnote{Continuous case 1} & CC2\tablefootnote{Continuous case 2} &  DC1\tablefootnote{Discontinuous case 1} & DC2\tablefootnote{Discontinuous case 2} & CNC1\tablefootnote{Common noise case 1: with increasing path} & CNC2\tablefootnote{Common noise case 2: with decreasing path}\\
			\hline
			Initial Condition & $\operatorname{Unif}[0.25,\, 0.35]$&$\Gamma(2.1,0.5)$&$\Gamma(1.2,\,0.5)$&$\Gamma(1.4,0.5)$&$\operatorname{Unif}[0.25,\, 0.35]$&$\operatorname{Unif}[0.25,\, 0.35]$ \\
			\hline
			$\alpha$ &$0.5$&$1.3$&$0.9$&$2$&$0.5$&$0.5$ \\
			\hline
			$\Delta_t$ &$10^{-6}$&$10^{-6}$&$10^{-9}$&$10^{-9}$&$10^{-6}$&$10^{-6}$ \\
			\hline
			$t_{\operatorname{max}}$ &$0.1$&$0.1$&$10^{-4}$&$10^{-4}$&$0.1$&$2 \times 10^{-2}$ \\
			\hline
		\end{tabular}
		\caption{Parameters of numerical simulations in Section \ref{sec: SECTION ON THE NUMERICAL SCHEMES}}
		\label{tab: PARAMETERS OF THE NUMERICAL SIMULATIONS}
	\end{table}
	
	With the chosen parameters, we generated the convergence graphs in \cref{fig: FIGURE IN THE CONTINUOUS SETTING}, \cref{fig: FIGURE IN THE DISCONTINUOUS SETTING} and \cref{fig: FIGURE IN THE COMMON NOISE SETTING} 
	from $\operatorname{Error}(\varepsilon_{n})$, where $\varepsilon_{n} \vcentcolon = \varepsilon \times \Delta^n$, with $\varepsilon$ and $\Delta$ as positive constants, and $\operatorname{Error}$ is the 
	corresponding difference between the smoothed and limiting loss functions.
	Assuming a power law relationship between the error and the parameter $\varepsilon$, 
	\begin{equation*}
		\operatorname{Error}(\varepsilon) \approx A \varepsilon^\beta,
	\end{equation*}
	where $A$ and $\beta$ are constants,
	we performed a linear regression on $\operatorname{Log} \operatorname{Error}(\varepsilon)$ versus $\operatorname{Log} \varepsilon$, which determined the line of best fit shown in the plots. The slope, shown in \cref{tab: GRADIENT OF THE REGRESSION LINE WITH FIXED DELTA T}, represents our best estimate for the rate of convergence for each setting.
	
	\begin{table}[H]
		\centering
		\begin{tabular}{|c|c|c|c|c|c|c|}
			\hline
			Setting & CC1 & CC2 & DC1 & DC2 & CNC1 & CNC2 \\
			\hline
			\hline
			Rate & $1.0202$ & $0.9635$ & $0.9295$ & $0.8126$ & $0.7621$ & $0.8144$ \\
			\hline
		\end{tabular}
		\caption{Gradient of the regression line}
		\label{tab: GRADIENT OF THE REGRESSION LINE WITH FIXED DELTA T}
	\end{table}
	
	To assess if the estimated slope corresponds to an asymptotic value, 
	we also conducted an alternative analysis of the rate of convergence. 
	By computing the ratio between two consecutive errors, we observe
	\begin{equation*}
		\frac{\operatorname{Error}(\varepsilon_{n+1})}{\operatorname{Error}(\varepsilon_n)} \approx \frac{A\Delta^{n\beta + \beta}\varepsilon^\beta}{A \Delta^{n\beta}\varepsilon^\beta} \approx \Delta^\beta,
	\end{equation*}
	and taking logarithms with base $\Delta$, we may deduce
	\begin{equation*}
		\operatorname{Log}_{\Delta}\left( \frac{\operatorname{Error}(\varepsilon_{n+1})}{\operatorname{Error}(\varepsilon_n) }\right) \approx \beta.
	\end{equation*}
	By using the relationship that $\operatorname{Log}_a b = \operatorname{Log}_c b / \operatorname{Log}_c a$ for any $a,\, b,\, c > 0$, we obtain approximate expressions for $\beta$ as follows:
	\begin{equation}\label{eq: SECOND and only equation in the table}
		\footnotesize
		\beta_n := \operatorname{Log}_{\Delta}\left( \frac{\operatorname{Error}(\varepsilon_{n+1})}{\operatorname{Error}(\varepsilon_n) }\right) = \frac{\operatorname{Log}({\operatorname{Error}(\varepsilon_{n+1})}) - \operatorname{Log}({\operatorname{Error}(\varepsilon_{n})})}{\operatorname{Log}(\Delta)} =  \frac{\operatorname{Log}({\operatorname{Error}(\varepsilon_{n+1})}) - \operatorname{Log}({\operatorname{Error}(\varepsilon_{n})})}{\operatorname{Log}(\varepsilon_{n+1}) - \operatorname{Log}(\varepsilon_n)}.
	\end{equation}
	where $\operatorname{Log}$ represents the logarithm with respect to any base. 
	\begin{table}[!t]
		\centering
		\small
		\begin{tabular}{|c|c|c|c|c|c|c|c|c|c|}
			\cline{2-10}
			\multicolumn{1}{c|}{}&\multicolumn{9}{c|}{$\varepsilon_n$}\\
			\hline
			Simulation & $n = 1$ & $n = 2$ & $n = 3$ & $n = 4$ & $n = 5$ & $n = 6$ & $n = 7$ & $n = 8$ & $n = 9$ \\
			\hline
			\hline
			CC1 & $0.9395$ & $0.9596$ & $1.0348$ & $0.9428$ & $0.9649$ & $1.0884$ & $1.0954$ & $1.1309$ & $0.9805$ \\
			CC2 & $0.9315$ & $1.0231$ & $0.9101$ & $0.8885$ & $0.9012$ & $0.5896$ & $1.3951$ & $1.1744$ & $1.0025$ \\
			DC1 & $0.9195$ & $1.0594$ & $0.8032$ & $1.0674$ & $1.2587$ & $0.9238$ & $0.3265$ & $1.5408$ & $0.0566$ \\
			DC2 & $0.5304$ & $0.7907$ & $0.5235$ & $1.1918$ & $0.7092$ & $0.6909$ & $0.8841$ & $1.2225$ & $0.5127$ \\
			CNC1 & $0.7646$ & $0.7054$ & $0.8223$ & $0.8060$ & $0.9489$ & $0.4754$ & $0.8792$ & $0.6219$ & $0.8243$ \\
			CNC2 & $0.7258$ & $0.7915$ & $0.8005$ & $0.8219$ & $0.8305$ & $0.7787$ & $0.7749$ & $0.8241$ & $1.0809$ \\
			\hline
		\end{tabular}
		\caption{Gradient between adjacent points in the Log-Log plots in \cref{fig: FIGURE IN THE CONTINUOUS SETTING}, \cref{fig: FIGURE IN THE DISCONTINUOUS SETTING} and \cref{fig: FIGURE IN THE COMMON NOISE SETTING}}
		\label{tab: ALTERNATIVE RATE OF CONVERGENCE ANALYSIS TABLE}
	\end{table}
	From Table \ref{tab: ALTERNATIVE RATE OF CONVERGENCE ANALYSIS TABLE}, it is evident that in the cases of CC1 and CC2, the rate of convergence approaches $1$ asymptotically. However, for all other scenarios, there appears to be no distinct pattern or clear convergence of the gradients, which generally lie between $1/2$ and $1$.
	
	Furthermore, we investigated the sensitivity of the convergence rate analysis to the choice of $\Delta_t$.
	It is clear that for meaningful approximations to the smoothed system it is needed that $\Delta_t$
	is sufficiently small compared to $\varepsilon$, which necessitates extremely small time steps and
	makes the simulation of the particle systems computationally costly.
	To assess whether $\Delta_t$ is sufficiently small, we generated in \cref{fig: FIGURE WHERE WE RANGE THE DELTA T'S} and \cref{tab: TABLE WITH RESPECT TO CHANGING THE TIME DISCITISATION TO COMPUTE THE RATE OF CONVERGENCE OF THE REGRESSION LINE} rate of convergence plots with different values of $\Delta_t$. For each $\Delta_t$, we selected several values of $\varepsilon$ that were uniformly spaced (after taking logarithms) within the interval $[\Delta_t \times 10^{-2.5},\, \Delta_t \times 10^{-1}]$. The findings indicate that the estimated rate of convergence remains consistent with respect to variations in $\Delta_t$.
	
	\begin{table}[H]
		\centering
		\begin{tabular}{|c|c|c|c|c|c|c|c|c|}
			\cline{2-9}
			\multicolumn{1}{c|}{}&\multicolumn{8}{c|}{$\Delta_t$}\\
			\hline
			Simulation & $10^{-4.5}$ & $10^{-5}$ & $10^{-5.5}$ & $10^{-6}$ & $10^{-7.5}$ & $10^{-8}$ & $10^{-8.5}$ & $10^{-9}$ \\
			\hline
			\hline
			CC1 & $0.836$ & $0.944$ & $0.987$ & $0.967$ & $-$ & $-$ & $-$ & $-$ \\
			CC2 & $0.988$ & $1.014$ & $1.039$ & $0.970$ & $-$ & $-$ & $-$ & $-$ \\
			DC1 & $-$ & $-$ & $-$ & $-$ & $0.821$ & $0.940$ & $0.968$ & $0.909$ \\
			DC2 & $-$ & $-$ & $-$ & $-$ & $0.539$ & $0.829$ & $0.728$ & $0.850$ \\
			CNC1 & $0.775$ & $0.750$ & $0.875$ & $0.751$ & $-$ & $-$ & $-$ & $-$ \\
			CNC2 & $0.685$ & $0.793$ & $0.806$ & $0.830$ & $-$ & $-$ & $-$ & $-$ \\
			\hline
		\end{tabular}
		\caption{Gradient of the line of best fit in \cref{fig: FIGURE WHERE WE RANGE THE DELTA T'S} (if plotted)}
		\label{tab: TABLE WITH RESPECT TO CHANGING THE TIME DISCITISATION TO COMPUTE THE RATE OF CONVERGENCE OF THE REGRESSION LINE}
	\end{table}

	\begin{figure}[H]
		\makebox[\linewidth][c]{
			\begin{subfigure}[b]{0.4\columnwidth}
				\centering
				\includegraphics[width=\columnwidth]{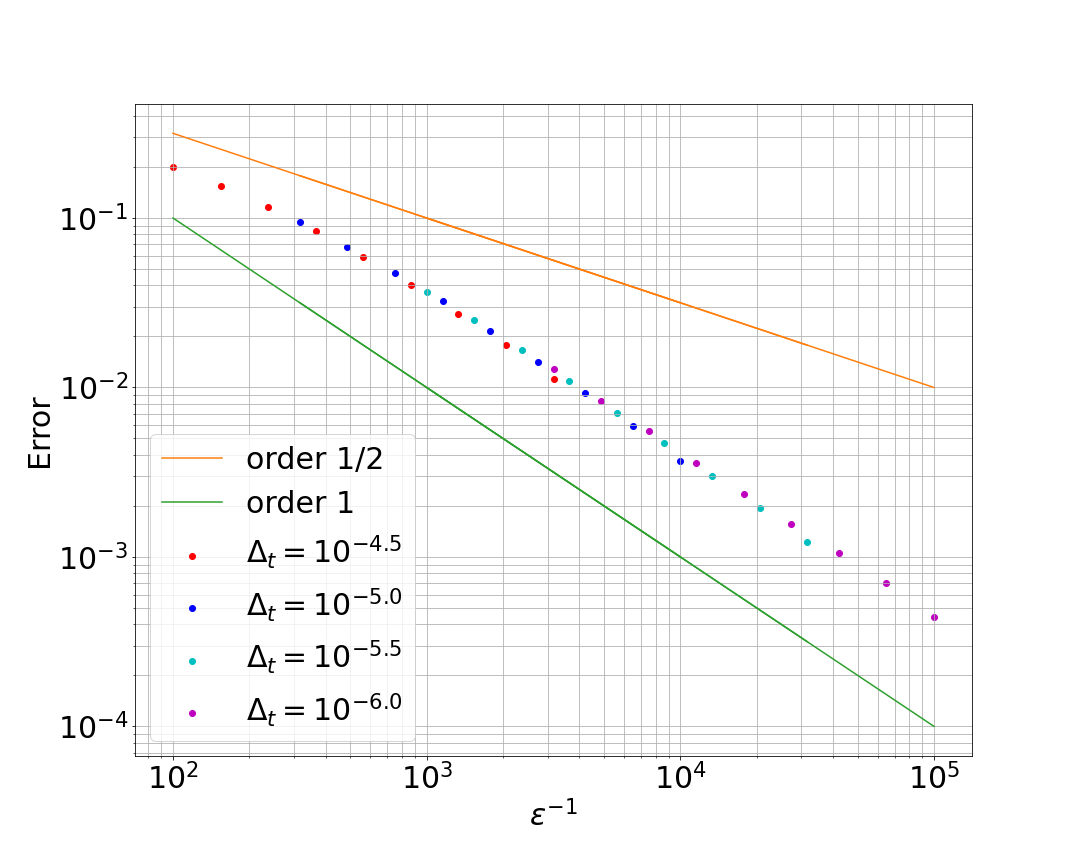}
				\caption{CC1}
				
			\end{subfigure}
			\hfill
			\begin{subfigure}[b]{0.4\columnwidth}
				\centering
				\includegraphics[width=\columnwidth]{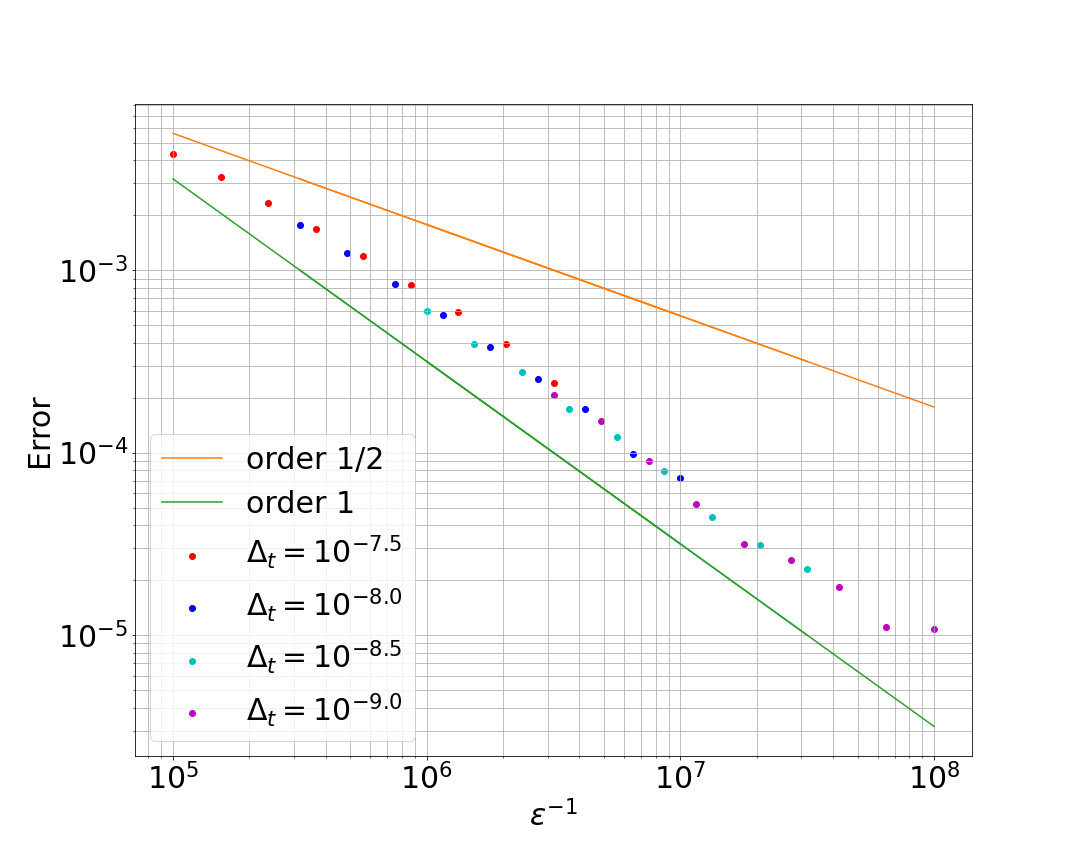}
				\caption{DC1}
				
			\end{subfigure}
			\hfill
			\begin{subfigure}[b]{0.4\columnwidth}
				\centering
				\includegraphics[width=\columnwidth]{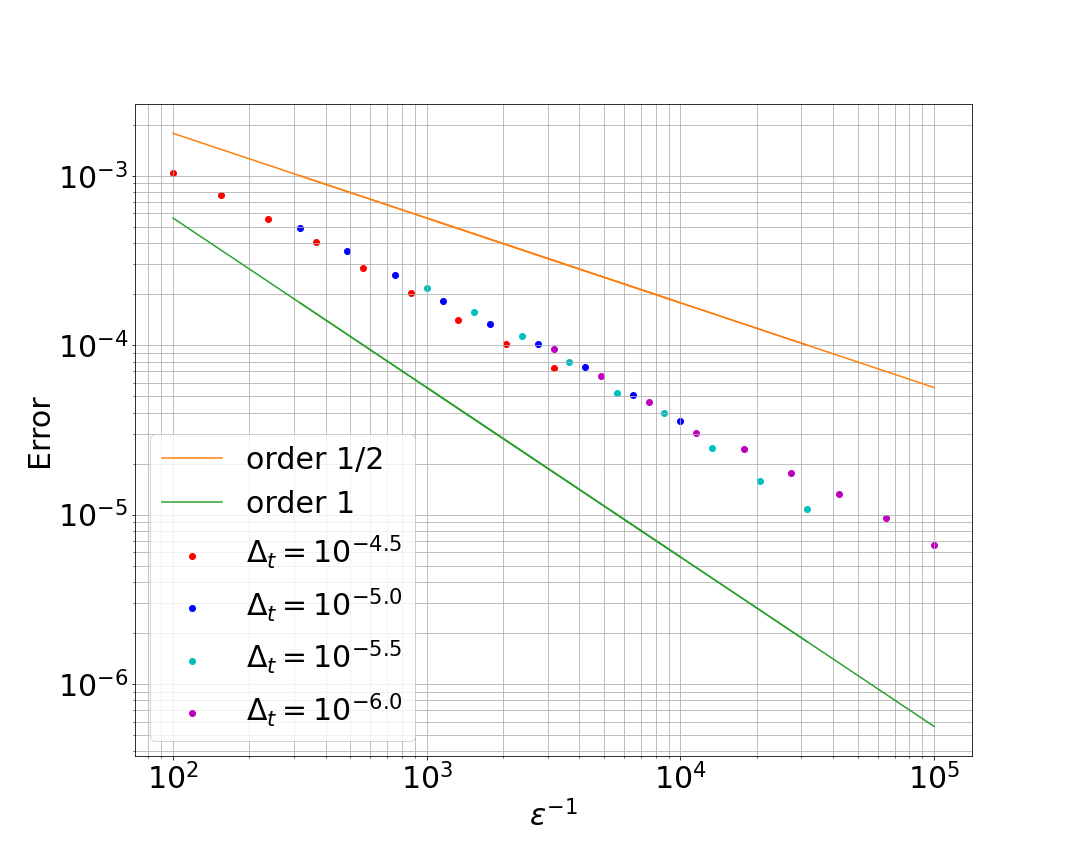}
				\caption{CNC1}
				
			\end{subfigure}
		}
		
		\makebox[\linewidth][c]{
			\begin{subfigure}[b]{0.4\columnwidth}
				\centering
				\includegraphics[width=\columnwidth]{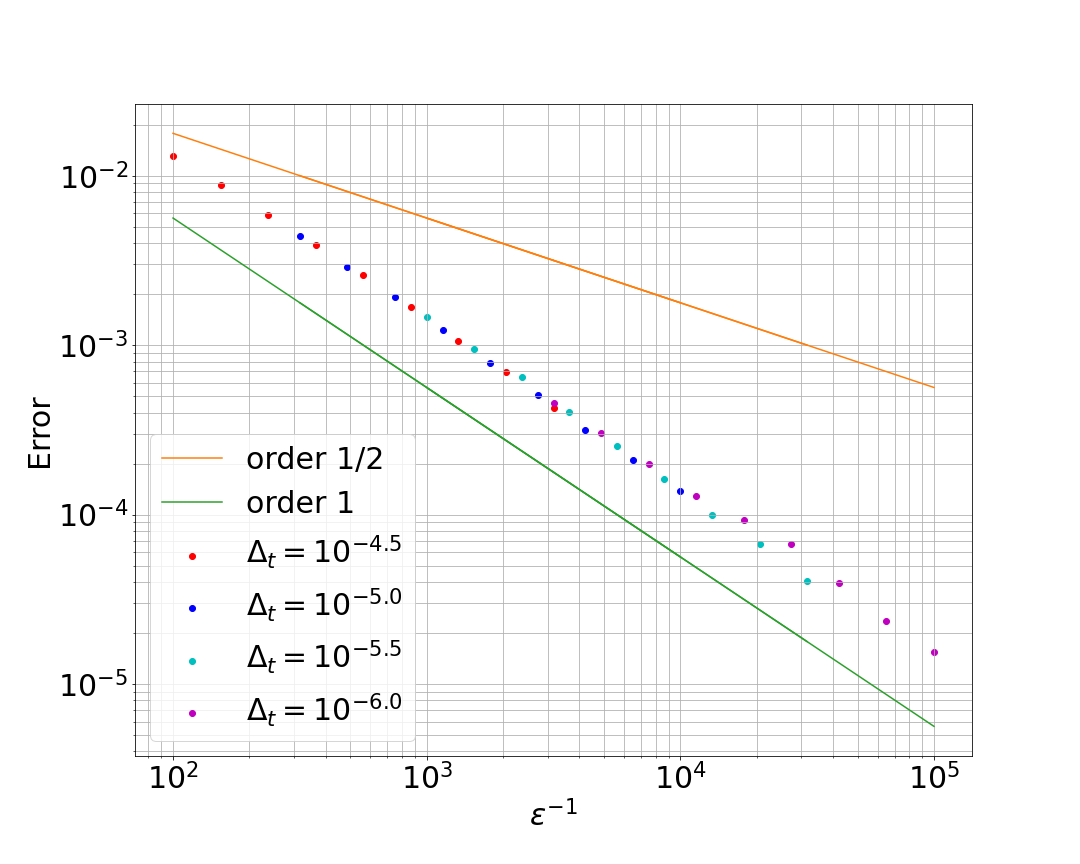}
				\caption{CC2}
				
			\end{subfigure}
			\hfill
			\begin{subfigure}[b]{0.4\columnwidth}
				\centering
				\includegraphics[width=\columnwidth]{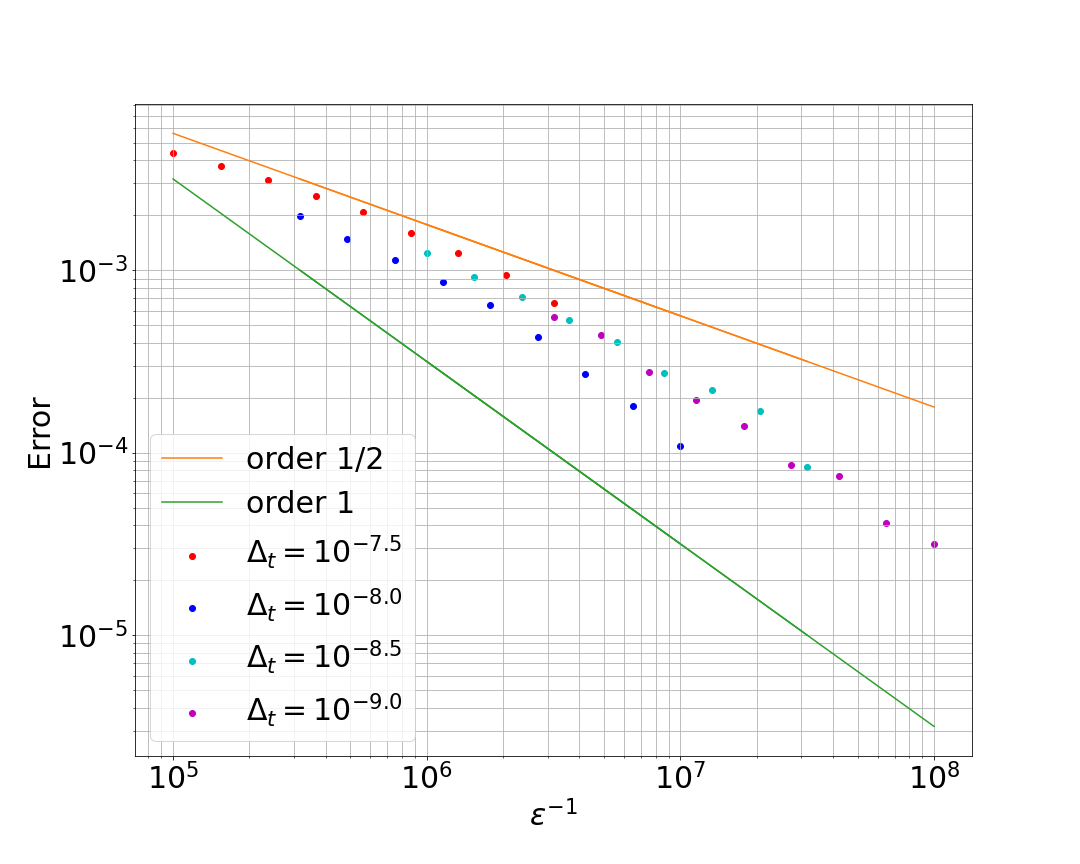}
				\caption{DC2}
				
			\end{subfigure}
			\hfill
			\begin{subfigure}[b]{0.4\columnwidth}
				\centering
				\includegraphics[width=\columnwidth]{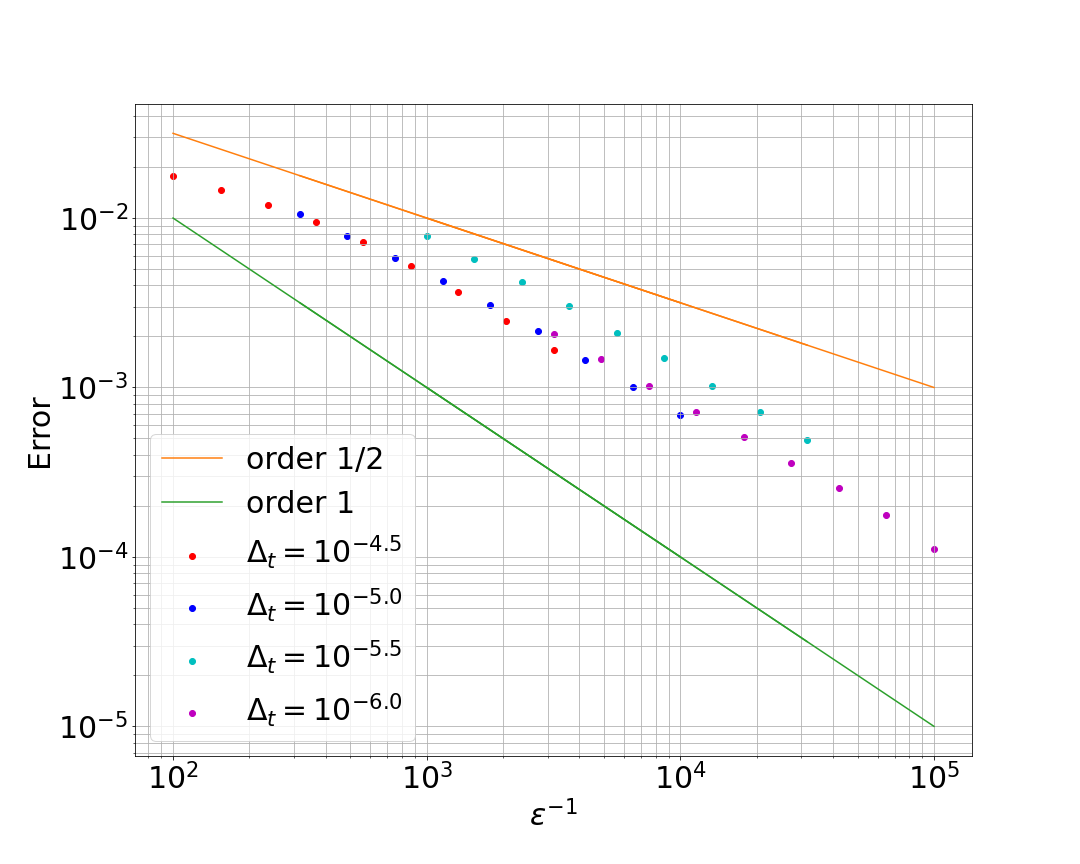}
				\caption{CNC2}
				
			\end{subfigure}
		}
		\caption{Sensitivity of rate of convergence with respect to changes in $\Delta_t$}
		\label{fig: FIGURE WHERE WE RANGE THE DELTA T'S}
	\end{figure}

	Finally, in \cref{fig: CASE WHEN WE USE AN INITIAL CONDITION THAT IS HOLDER REGULAR BUT THERE IS NO DISCONITNUITY} and \cref{tab: RATE OF CONVERGENCE PLOTS USING AN ALTERNATIVE INITIAL CONDITION}, we investigated a scenario where the initial condition is H\"older continuous near the boundary without any observed jump discontinuity in the simulations, specifically, $X_{0-} \sim^d \Gamma(1.5,2)$ with $\alpha = 1.3$. By \citep[Theorem~1.1]{delarue2022global}, the limiting loss function is $1/2$-H\"older continuous at $0$. The rate of convergence appears to be between $1/2$ and $1$ in this setting.
	
	\begin{table}[!h]
		\centering
		\begin{tabular}{|c|c|c|c|c|}
			\hline
			$\Delta_t$ & $10^{-4.5}$ & $10^{-5}$ & $10^{-5.5}$ & $10^{-6}$ \\
			\hline
			Gradient &$0.913$ & $0.863$ & $0.798$ & $0.66$ \\
			\hline
		\end{tabular}
		\caption{Gradient of the line of best fit in \cref{fig: CASE WHEN WE USE AN INITIAL CONDITION THAT IS HOLDER REGULAR BUT THERE IS NO DISCONITNUITY}}
		\label{tab: RATE OF CONVERGENCE PLOTS USING AN ALTERNATIVE INITIAL CONDITION}
	\end{table}
	
	\begin{figure}[!t]
		\makebox[\linewidth][c]{
			\begin{subfigure}[b]{0.55\columnwidth}
				\centering
				\includegraphics[width=\columnwidth]{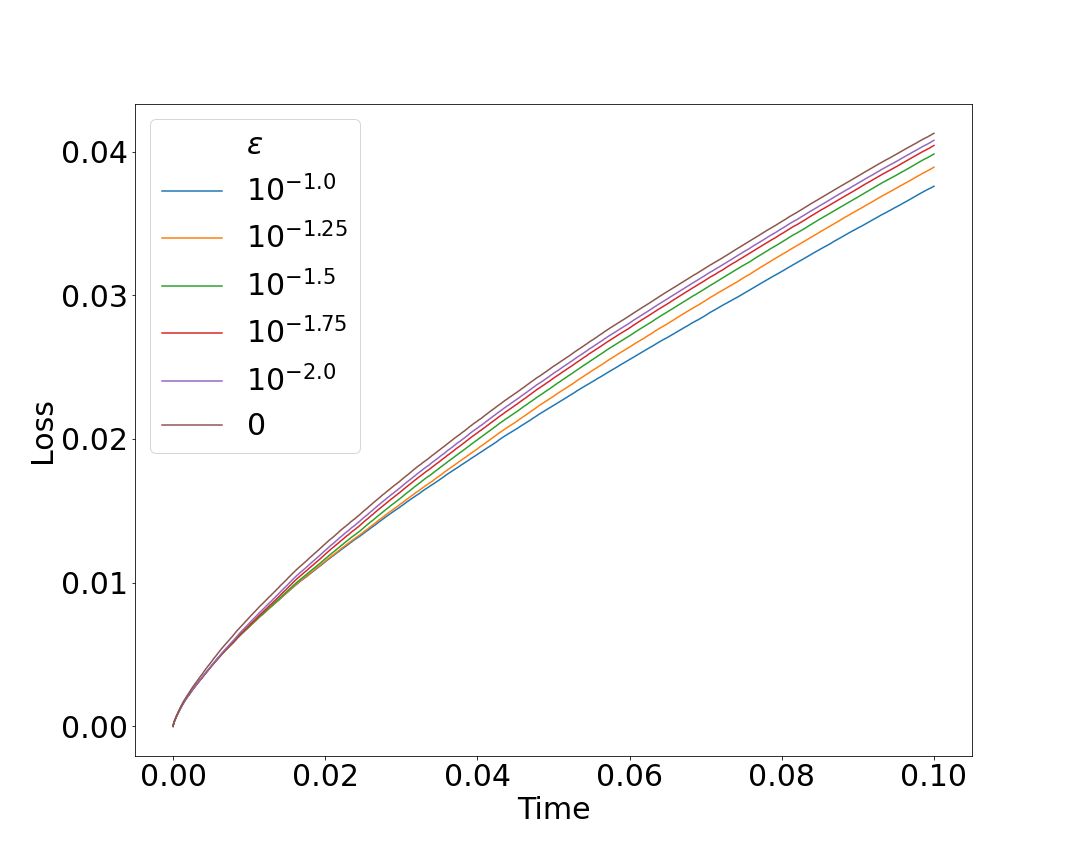}
				\caption{$X_{0-} \sim^d \Gamma(1.5,2),\,\alpha = 1.3$}
				
			\end{subfigure}
			\hfill
			\begin{subfigure}[b]{0.55\columnwidth}
				\centering
				\includegraphics[width=\columnwidth]{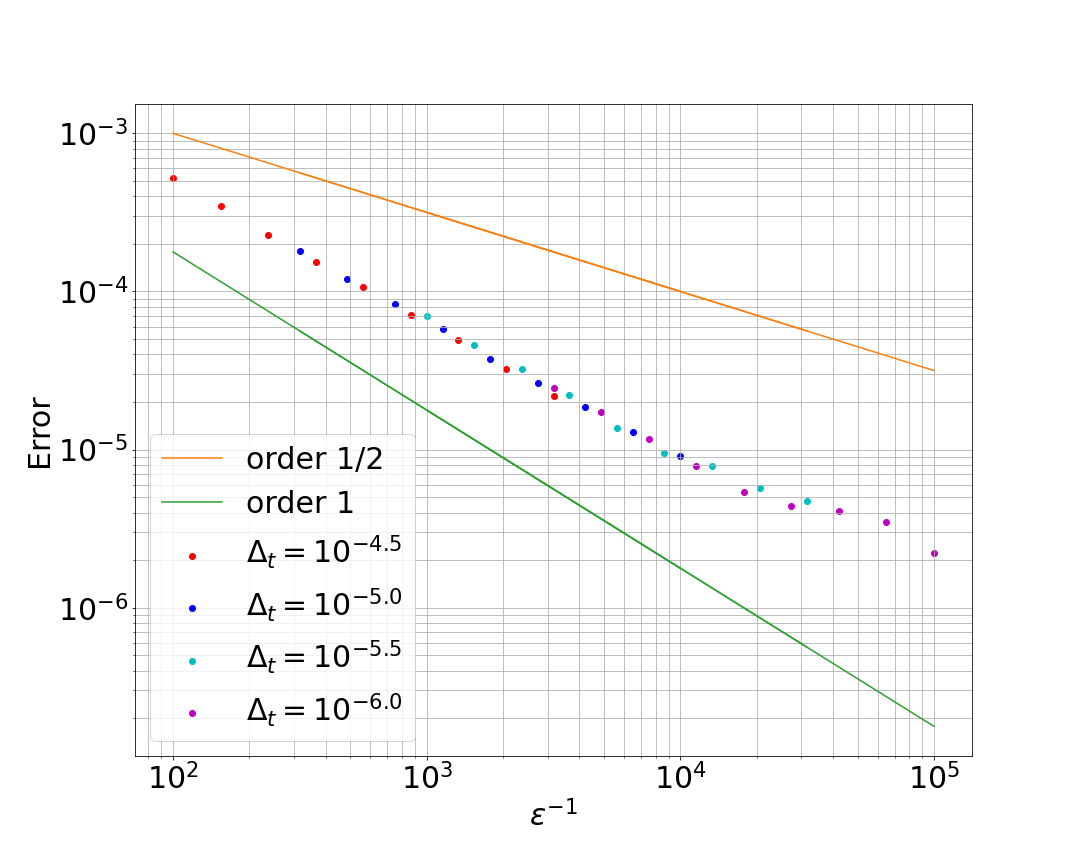}
				\caption{Rate of Convergence}
				\label{fig: RATE OF CONVERGENCE PLOTS USING AN ALTERNATIVE INITIAL CONDITION}
			\end{subfigure}
		}
		\caption{Sensitivity of the rate of convergence with respect to changes in $\Delta_t$}
		\label{fig: CASE WHEN WE USE AN INITIAL CONDITION THAT IS HOLDER REGULAR BUT THERE IS NO DISCONITNUITY}
	\end{figure}
	
\end{appendices}

\paragraph{Acknowledgements.} This research has been supported by the EPSRC Centre for Doctoral Training in Mathematics of Random Systems: Analysis, Modelling and Simulation (EP/S023925/1). AP thanks Philipp Jettkant for discussions on this material. Further, the authors wish to thank two anonymous referees for their careful reading and suggestions for improvement.




\printbibliography

	



\end{document}